\documentclass[a4paper, 11pt]{article}

\usepackage{mathrsfs,amssymb,amsmath, amsthm,tikz, mathdots, amsfonts, graphicx, subfigure, tikz,xcolor, float, pgf, mathrsfs, verbatim}

\usetikzlibrary{arrows}
\usepackage[top=1.31in,bottom=1.29in,left=1.1in,right=0.9in]{geometry}
\usepackage[numbers]{natbib}
\usepackage[all]{xy}

\definecolor{darkgreen}{rgb}{0,.6,0}

\setcitestyle{open={},close={}}

\numberwithin{equation}{section}\theoremstyle{definition}
\swapnumbers

 \newtheorem{Theorem}[equation]{Theorem}
 \newtheorem{Prop}[equation]{Proposition}
 \newtheorem{Lemma}[equation]{Lemma}
 \newtheorem{Cor}[equation]{Corollary}
 \newtheorem{Notation}[equation]{Notation}
 \newtheorem{Defn}[equation]{Definition}
 \newtheorem{Remark}[equation]{Remark}

\newtheorem{Situation}[equation]{Situation}

 \makeatletter
\def\enumerate{\begingroup\ifnum\@enumdepth>3\@toodeep\else
      \advance\@enumdepth\@ne
      \edef\@enumctr{enum\romannumeral\the\@enumdepth}%
      \topsep\z@\parskip\z@
      \list{\csname label\@enumctr\endcsname}
        {\@nmbrlisttrue\let\@listctr\@enumctr
         \parsep\z@\itemsep\z@\topsep\z@
         \setcounter{\@enumctr}{0}
         \def \fMakelabel##1{\hss\llap{\rm ##1}}
       }\fi}

 \makeatother

 \def\tril{\tikz\draw[red, fill=red, thick](0,0)--(0.6em,0)--(0.6em,-0.6em)--cycle; \tikz \draw [dotted, thick] (0,0)--(0.6em,0)--(0,-0.6em);}

 \def\trir{\tikz\draw[dotted, thick](0.6em,0)-- (0,0)--(0.6em,-0.6em); 
 \tikz \draw [cyan, fill=cyan,thick] (0,0)--(0.6em,0)--(0,-0.6em);}
 
 \def\pKL{\raisebox{-0.15em}{\tikz{\draw(0,0)--(0.8em,0)--(0em,-0.8em)--cycle;
\draw(0em,-0.2em)--(0.6em,-0.2em);
\draw(0em,-0.4em)--(0.4em,-0.4em);
\draw(0em,-0.6em)--(0.2em,-0.6em);}}}

 \def\KL{\raisebox{-0.15em}{\tikz{\draw(0,0)--(0.8em,0)--(0em,-0.8em)--cycle;
}}}

 \def\Square{\raisebox{-0.15em}{\tikz{\draw(0,0)--(0.8em,0)--(0.8em,-0.8em)--(0,-0.8em)--cycle;
}}}

 \def\diag{\raisebox{-0.15em}{\tikz{\draw(0,0)--(0.8em,-0.8em);
}}}

 \def\UR{\raisebox{-0.15em}{\tikz{\draw(0,0)--(0.8em,0)--(0.8em,-0.8em)--cycle;
}\,}}

 \def\CC{\raisebox{0.15em}{\tikz[scale=1.4]{\draw(0.4em,-0.4em)--(0.8em,0);
}}}

 \def\UP{\raisebox{0.15em}{\tikz[scale=0.9]{\draw(0,0)--(1.2em,0)--(0.6em,-0.6em)--cycle;
}}}

 \def\RP{\raisebox{-0.15em}{\tikz[scale=1.34]{\draw(0.8em,-0.8em)--(0.8em,0)--(0.4em,-0.4em)--cycle;
}\,}}

 \def\UPC{\raisebox{0.15em}{\tikz[scale=1.4]{\draw(0,0)--(0.8em,0)--(0.4em,-0.4em)--cycle; \draw(0.6em,-0.4em)--(1em,0 em);
}}}

 \def\RPC{\raisebox{-0.15em}{\tikz[scale=1.4]{\draw(0.8em,-0.8em)--(0.8em,0)--(0.4em,-0.4em)--cycle;\draw(0.3em,-0.3em)--(0.8em,0.2 em);
}\,}}

 \def\pUP{{\tikz{\draw(0,0)--(1.2em,0)--(0.6em,-0.6em)--cycle; \draw(0.15em,-0.15em)--(1.05em,-0.15em);
\draw(0.3em,-0.3em)--(0.9em,-0.3em);
\draw(0.45em,-0.45em)--(0.75em,-0.45em);
}}}

 \def\pUPt{\,\raisebox{0.1em}{\tikz[scale=0.42]{\draw(0,0)--(1.2em,0)--(0.6em,-0.6em)--cycle;
\draw(0.2em,-0.2em)--(1em,-0.2em);
\draw(0.4em,-0.4em)--(0.8em,-0.4em);
}}}

 \def\pUPs{\,\raisebox{0.1em}{\tikz[scale=0.5]{\draw(0,0)--(1.2em,0)--(0.6em,-0.6em)--cycle; 
\draw(0.2em,-0.2em)--(1em,-0.2em);
\draw(0.4em,-0.4em)--(0.8em,-0.4em);
}}}

 \def\URs{\raisebox{-0.1em}{\,\tikz[scale=0.7]{\draw(0,0)--(0.8em,0)--(0.8em,-0.8em)--cycle;
}\,}}

 \newcommand{\on}{\operatorname}
 \newcommand{\mc}{\on{main}}
\newcommand{\lmc}{\on{l.\! main}}
\newcommand{\rmc}{\on{r.\! main}}
\newcommand{\minc}{\on{minor}}

\newcommand{\core}{\on{core}}
\newcommand{\Arm}{\on{Arm}}
\newcommand{\Leg}{\on{Leg}}
\newcommand{\Limb}{\on{Limb}}
\newcommand{\stab}{\on{stab}}

\def\leq{\leqslant}

\def\C{\mathbb C}
\def\N{\mathbb N}

\def\F{\mathbb F}

\def\cO{\mathcal O}

\def\cA{\mathcal A}

\def\cL{\mathcal L}
\def\cR{\mathcal R}

\def\cD{\mathcal D}
\def\cV{\mathcal V}

\def\fX{\mathfrak X}

\def\fD{\mathfrak D}
\def\fS{\mathfrak S}
\def\fC{\mathfrak C}

\def\fB{\mathfrak B}
\def\fA{\mathfrak A}

\def\Pl{\mathbb{PL}}

\DeclareMathOperator{\Hom}{Hom}

\DeclareMathOperator{\Ind}{Ind}

\DeclareMathOperator{\Stab}{Stab}

\DeclareMathOperator{\main}{main}

\DeclareMathOperator{\verge}{verge}
\DeclareMathOperator{\suppl}{suppl}

\DeclareMathOperator{\supp}{supp}

\DeclareMathOperator{\Lie}{Lie}
\DeclareMathOperator{\Mat}{Mat}

\begin{document}

\title{On coadjoint orbits for $p$-Sylow subgroups\\ of finite classical groups}

\author{Qiong Guo$^{*}$, Markus Jedlitschky$^{**}$, Richard Dipper$^{**}$\\ \\$^{*}${\footnotesize College of Sciences, Shanghai Institute of Technology} \\ {\footnotesize 201418 Shanghai, PR China}
\smallskip \\$^{**}$ {\footnotesize Institut f\"{u}r Algebra und Zahlentheorie}\\ {\footnotesize Universit\"{a}t Stuttgart, 70569 Stuttgart, Germany}
\setcounter{footnote}{-1}\footnote{\scriptsize E-mail: qiongguo@hotmail.com, markus.jedlitschky@gmail.com, richard.dipper@mathematik.uni-stuttgart.de}
\setcounter{footnote}{-1}\footnote{
{\scriptsize This work was  supported by NSFC no.11601338.}}
\setcounter{footnote}{-1}
\setcounter{footnote}{-1}}
\date{}


\maketitle

\begin{center}
\today
\end{center}

\begin{abstract}
Kirillov's orbit theory provides a powerful tool for the investigation of irreducible unitary representations of many classes of Lie groups. In a previous paper we used a modification hereof, called monomial linearisation, to construct a monomial basis of the regular representation of $p$-Sylow subgroups $U$ of the finite classical groups of untwisted type. In this sequel to this article we determine the stabilizers of special orbit generators and show, that for the groups of Lie type $\fB_n$ and $\fD_n$ a subclass of the orbit modules decompose the $U$-modules affording the Andr\'{e}-Neto supercharacters into a direct sum of submodules. Moreover these special orbit modules are either isomorphic or have no irreducible constituent in common, and each irreducible $U$ module is up to isomorphism constituent of precisely one of these.
\end{abstract}

\section{Introduction}
To determine the irreducible complex characters of group $\tilde U_n(q)$ of upper unitriangular $n\times n$-matrices over the finite field $\F_q$ is a notorious hard problem. Indeed, doing this simultaneously for all $n\in \N$ and all prime powers $q=p^k$ is known to be a wild problem. Being confronted with an undoable wild problem, one can,  try to loosen up some requirements of the task to obtain an easier problem which actually can be solved. For the finite unitriangular group $\tilde U_n(q)$ supercharacter theory does precisely this. Andr\'{e} in [\cite{andre}] using Kirillov theory [\cite{kirillov}], and Yan [\cite{yan}] (by different methods) introduced a set of $\tilde U_n$-characters, called supercharacters [\cite{super}], which are pairwise orthogonal and contain every irreducible complex character  of $\tilde U_n(q)$ as constituent of precisely one supercharacter. They also considered certain pairwise disjoint unions of conjugacy classes, called superclasses, such that every conjugacy class of $\tilde U_n(q)$ is contained in precisely one superclass, supercharacters are constant on superclasses and supercharacters are in bijection with superclasses. Actually, this Andr\'{e}-Yan construction was axiomatized by Diaconis and Isaacs in [\cite{super}] and named supercharacter theory.

To extend Andr\'{e}'s or Yan's method to the $p$-Sylow subgroups $U_n$ of other finite classical groups directly does not work, basically since these are not algebra groups. Andr\'{e} and Neto [\cite{andreneto1},\cite{andreneto2},\cite{andreneto3}] hence defined supercharacter theories for $U_n$ (of non twisted Lie type) by restricting certain supercharacters from the overlying full unitriangular group $\tilde U_N(q)$ to $U_n$, respectively intersecting superclasses of $\tilde U_N(q)$ with $U_n$. This, however, is a rather coarse supercharacter theory and it is desirable to find a finer one, based on monomial orbits.

In his doctoral thesis the second named author used a different approach, called  monomial linearisation, which can be considered as generalization of Kirillov's orbit method. He applied this to the even orthogonal groups and obtained a further decomposition of the Andr\'{e}-Neto supercharacters into characters, which are either orthogonal or coincide. Moreover every irreducible character of $U_n$ occurs in precisely one of those as irreducible constituent.

In a previous paper [\cite{GJD}] we exhibited a monomial linearisation for the $p$-Sylow subgroups $U_n$ of finite groups of Lie types $\fB_n, \fD_n$ and $\fC_n$, by constructing a monomial basis for the group algebra $\C U_n$. Thus we  extended the construction of [\cite{markus}] to all classical groups of untwisted Lie type. This decomposes $\C U_n$ as $U_n$-modules into the direct sum of orbit modules. We showed that every orbit module is isomorphic to a so called staircase orbit module and classified the staircase orbits by certain elements of the monomial basis of $\C U_n$, named cores.

In this sequel to [\cite{GJD}] we investigate even more specialized orbits, called main separated orbits and show, that every irreducible $\C U_n$-module occurs as irreducible constituent in one of those. Our goal for this paper is to prove, that the characters afforded by main separated orbit modules coincide or are orthogonal. So far we managed to prove this for $U_n$ of type $\fB_n, \fD_n$, such that indeed every irreducible $U_n$-character is up to isomorphisms constituent of precisely one character afforded by a main separated orbit module. For type $\fC_n$ we need an extra (quite restrictive) hypothesis and hence can show this result only for a certain subset of main separated orbit modules.

Thus we have constructed a decomposition of the Andr\'{e}-Neto supercharacters into characters of $U_n$ afforded by main separated orbit modules for the untwisted orthogonal groups in general and the symplectic groups for some special ones.

Section \ref{setup} is preliminary collecting the main definitions and results of [\cite{GJD}]. In section \ref{secleftu} we define the notation of main separated orbits and show, that every irreducible $\C U_n$-module is constituent of a main separated orbit module. After some preparation in section \ref{secverge} we prove our main result in the last section and draw some consequences of this.

\section{Set up}\label{setup}

 In this preliminary section we collect notation and some basic facts from [\cite{GJD}]. Throughout  $U=U(\fB_n),U(\fC_n) $ or $U(\fD_n)$  denotes the $p$-Sylow subgroup of the finite group of Lie type $\fB_n, \fC_n$ and $\fD_n$ respectively, constructed in [\cite{GJD}], where the describing characteristic of the group is the odd prime $p$.

\begin{Notation}\quad
\begin{enumerate}
\item [1)] Let $e_{ij}$ denote the $N\times N$-matrix over $\F_q$ having entry $1$ at position $(i,j)$ and zeros elsewhere.
\item [2)]  For any $r\times s$-matrix $A$, we denote the $(i,j)$-th entry of $A$ by $A_{ij}$ and  let $A^t$ be the transposed $s\times r$ matrix.
\item [3)]Let $N \in \N$. Define the {\bfseries mirror map}
$\bar{}:\; \{1, \dots, N\} \rightarrow \{1, \dots, N\}:\; i \longmapsto \bar i := N+1-i$, which 
mirrors every entry on $\frac{N+1}{2}$. It satisfies $\bar{\bar i} = i$ and
$ i<j\leq k  \Leftrightarrow \bar k \leq \bar j < \bar i$  for all $i,j,k \in \{1,\dots,N\}$.
\item [4)] Let $\Square$  $=$  $\{ (i,j) \; | \; 1 \leq i,j \leq N \},$
and $\UR=$  $\{ (i,j) \in \Square \; | \; i < j \}$.
\item [5)] Denote the diagonal by $\diag :=\{(i,j) \in \Square \; | \; i=j \}$ and 
half of the anti-diagonal by ${\CC}:=\{ (i,j)\in \UR \; | \;  i = \bar j \}$. 
\item [6)] Let  ${\UP}= \{ (i,j) \in \Square \; | \; i < j < \bar i \}$ and 
${\RP}=\{ (i,j)\in \Square \; | \; \bar j < i < j \}$.
Thus $\UR=\UP\cup \CC \cup \RP$.
\item[7)] Set $\UPC=\UP\cup \CC$ and $\RPC=\RP\cup \CC$.
\item[8)]Let
$\pUP=\UP $ in types $\fB_n, \fD_n$ and let $\pUP=\UPC$ in type $\fC_n.$
\item[9)] For $A\in \Mat_{N\times N}(q)$ define the {\bf support} of $A$ to be 
$
\supp A=\{(i,j)\in \Square\,|\,A_{ij}\neq 0\}
$
and for $S\subseteq \square$, define 
$V_S=\{A\in \Mat_{N\times N}(q)\,|\, \supp A\subseteq S\}$. In particular, for the Dynkin type $\fX_n=\fB_n, \fC_n$ or $ \fD_n$, set $V=V(\fX_n)=V_{\pUPs}= \bigoplus_{(i,j)\in \pUPs} \F_q e_{ij}$.  
\item[10)] For $S\subseteq \square$, define  $\pi_S:\Mat_{N\times N}(q) \longrightarrow V_S$ by mapping $A\in \Mat_{N\times N}(q)$  to $ \sum_{(i,j)\in S}A_{ij}e_{ij}$. In particular set $\pi = \pi_{\pUPs}$.\hfill$\square$
\end{enumerate}
\end{Notation}

\begin{Theorem}[{[\cite{GJD}]}]\label{unique}
Each $u\in U$ is uniquely determined by the entries on positions in $\pUP$ and hence the restriction of $\pi = \pi_{\pUPs}$ to $U$ is a bijection from $U$ onto $V=V_{\pUPs}$. \hfill$\square$
\end{Theorem}

More precisely for $u\in U$ and  $(r,s)\in \RP$ we have shown in [\cite{GJD}] that 
$u_{rs}$  is determined by the  entries on the positions which are to the left of $(r,s)$ or to the left of or on  $(\bar s, \bar r)$, which we illustrate as follows:
\begin{equation}\label{dependleft} 
\includegraphics[width=0.38\textwidth]{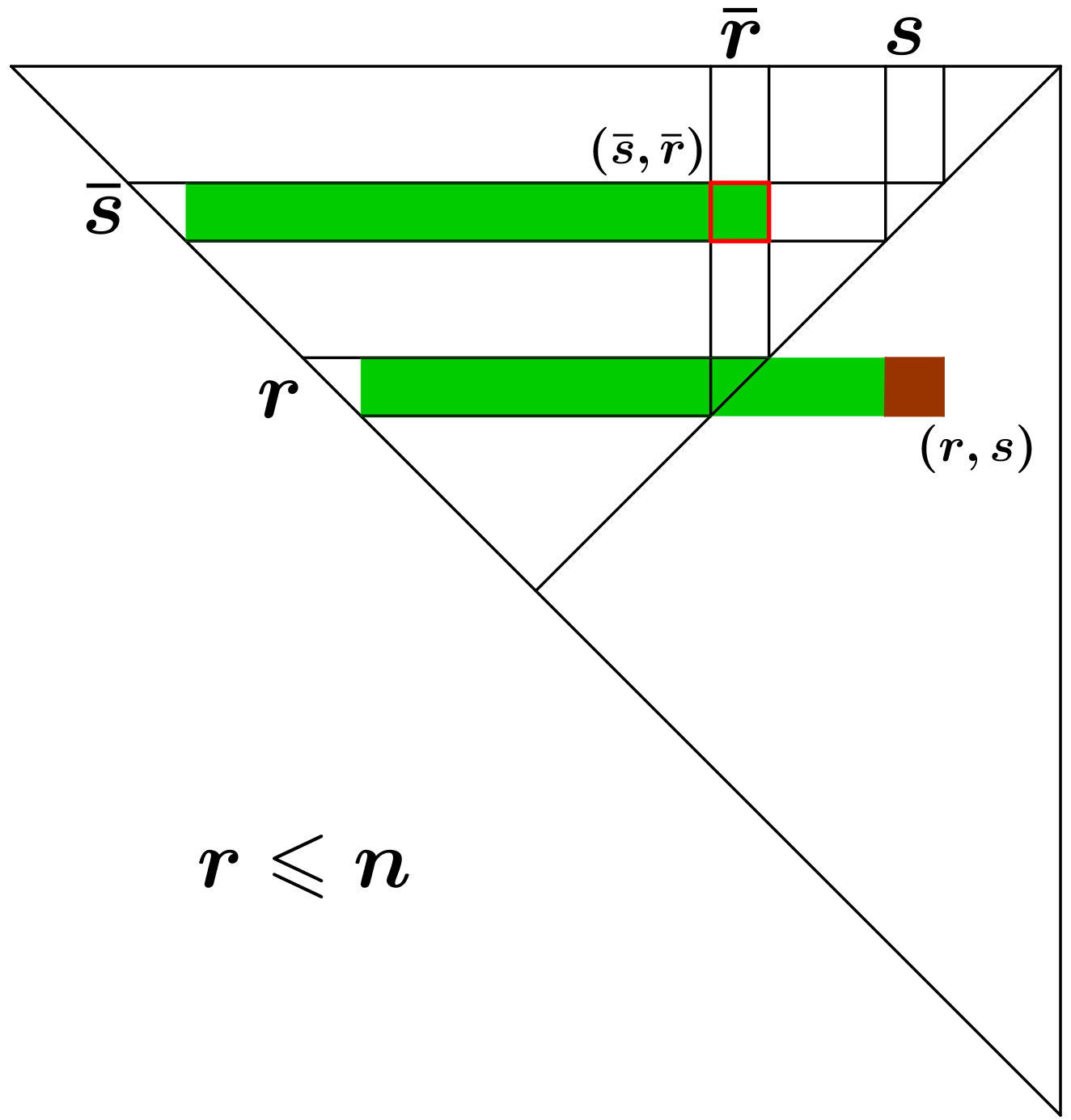}
\quad\quad
\includegraphics[width=0.38\textwidth]{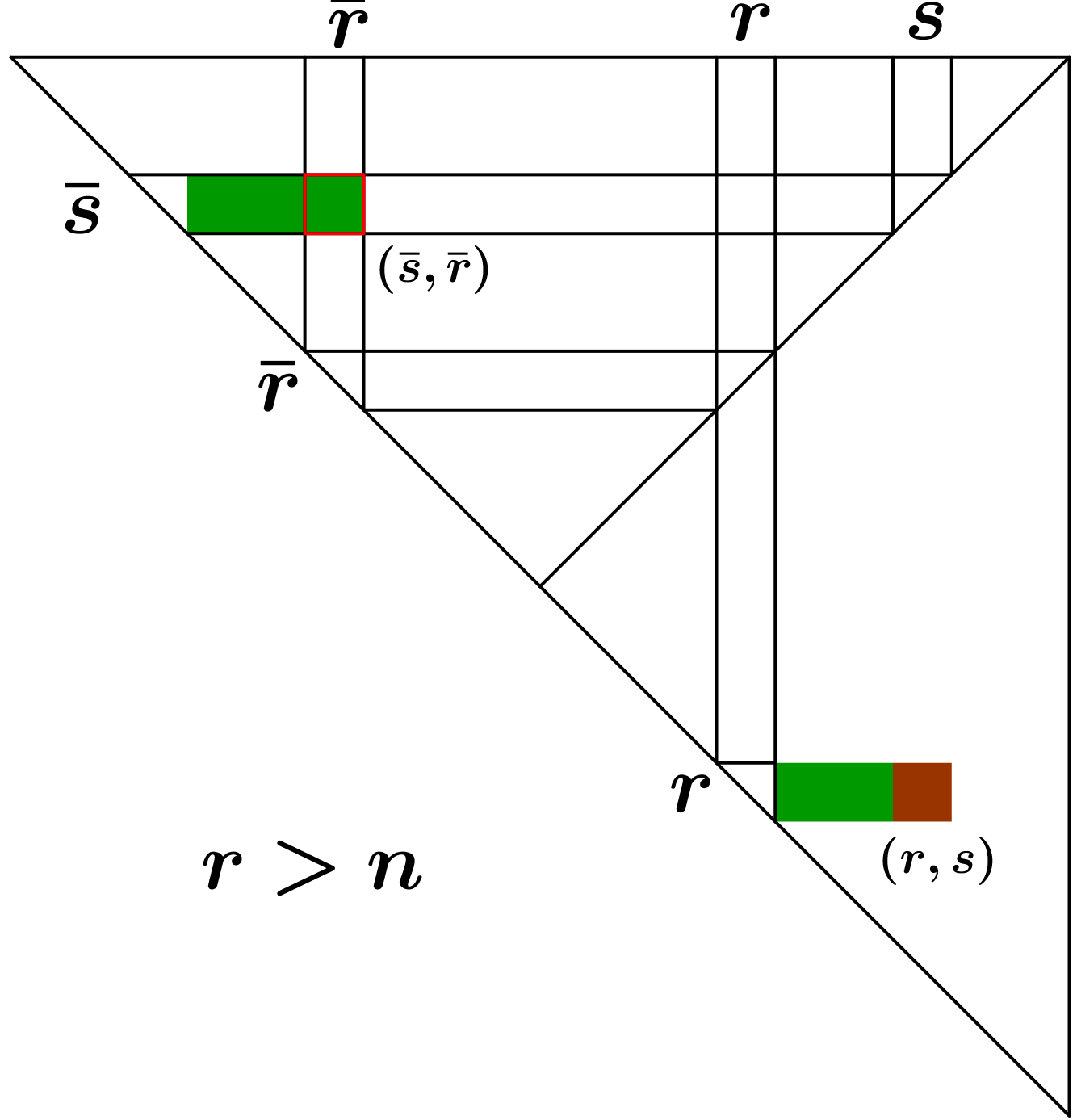}
\end{equation}
Indeed by [\cite{GJD}, 3.14] the entry at position $(r,s)$ of matrix  $u\in U$ can be expressed in terms of  entries of $u$ at the green positions, the entry at the red boxed position $(\bar s, \bar r)$ having coefficient $\pm 1$. From this one obtains easily the formulas in [\cite{GJD},3.22] which is
\begin{equation}\label{prs+}
p_{rs}=\pm t_{\bar s\, \bar r}-\sum\nolimits_{r<l<s}t_{\bar s\, \bar l}p_{rl}
\end{equation}
This equation states that we can express the entry at position $(r,s)$ of matrix  $u\in U$ as  polynomial with coefficients in the prime field $\F_p$ in the entries at positions 
$R_{rs}:= \{ (i,j) \in \pUP \; | \; \bar s \leq i \leq r \text{ and } j \leq \bar r \}$
 of the green boxes in \ref{dependleftresult} , the polynomial attached to position $(\bar s, \bar r)$ being the constant polynomial $\pm 1$.

\begin{equation}\label{dependleftresult} 
\includegraphics[width=0.45\textwidth]{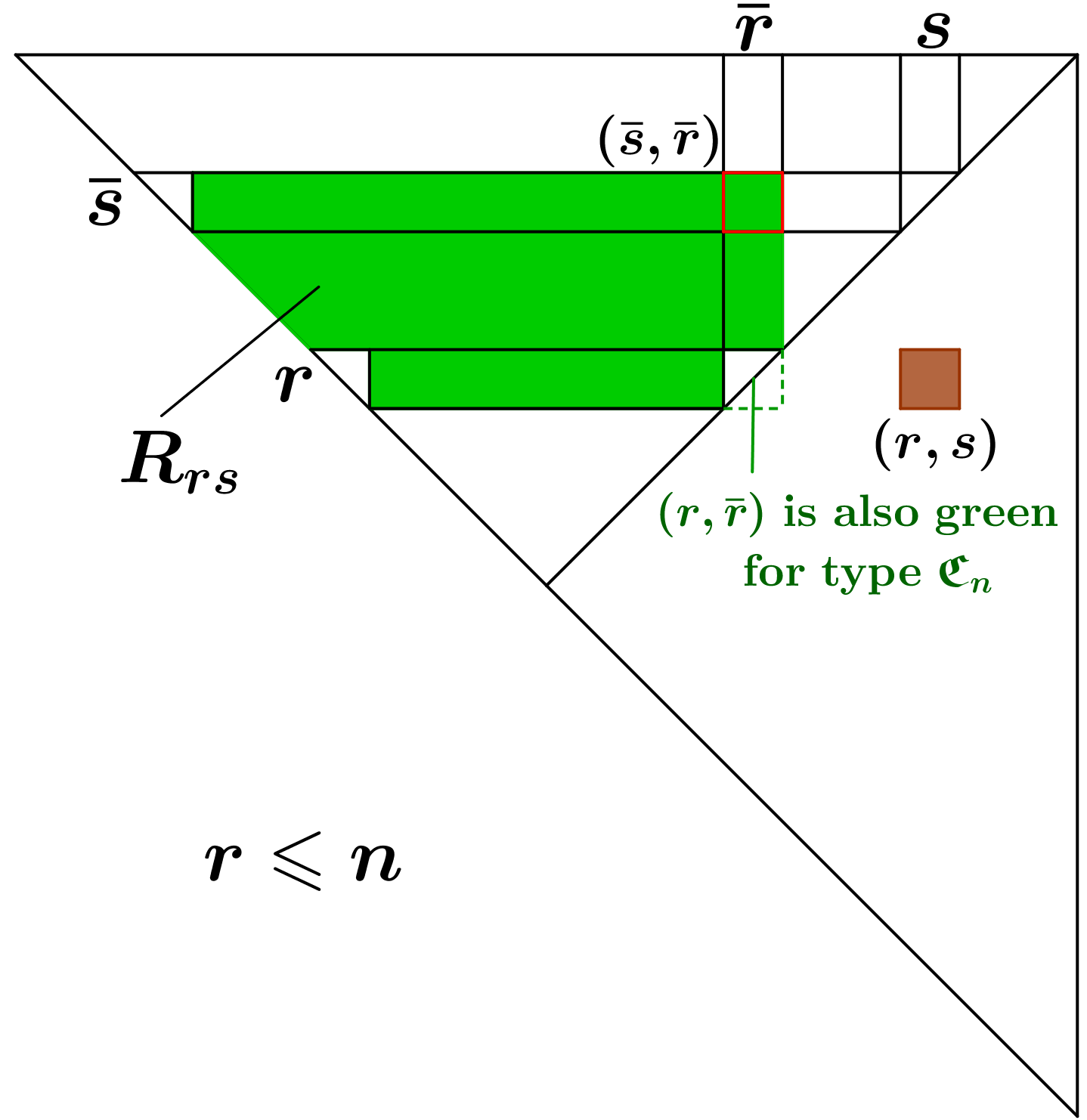}
\quad\quad
\includegraphics[width=0.45\textwidth]{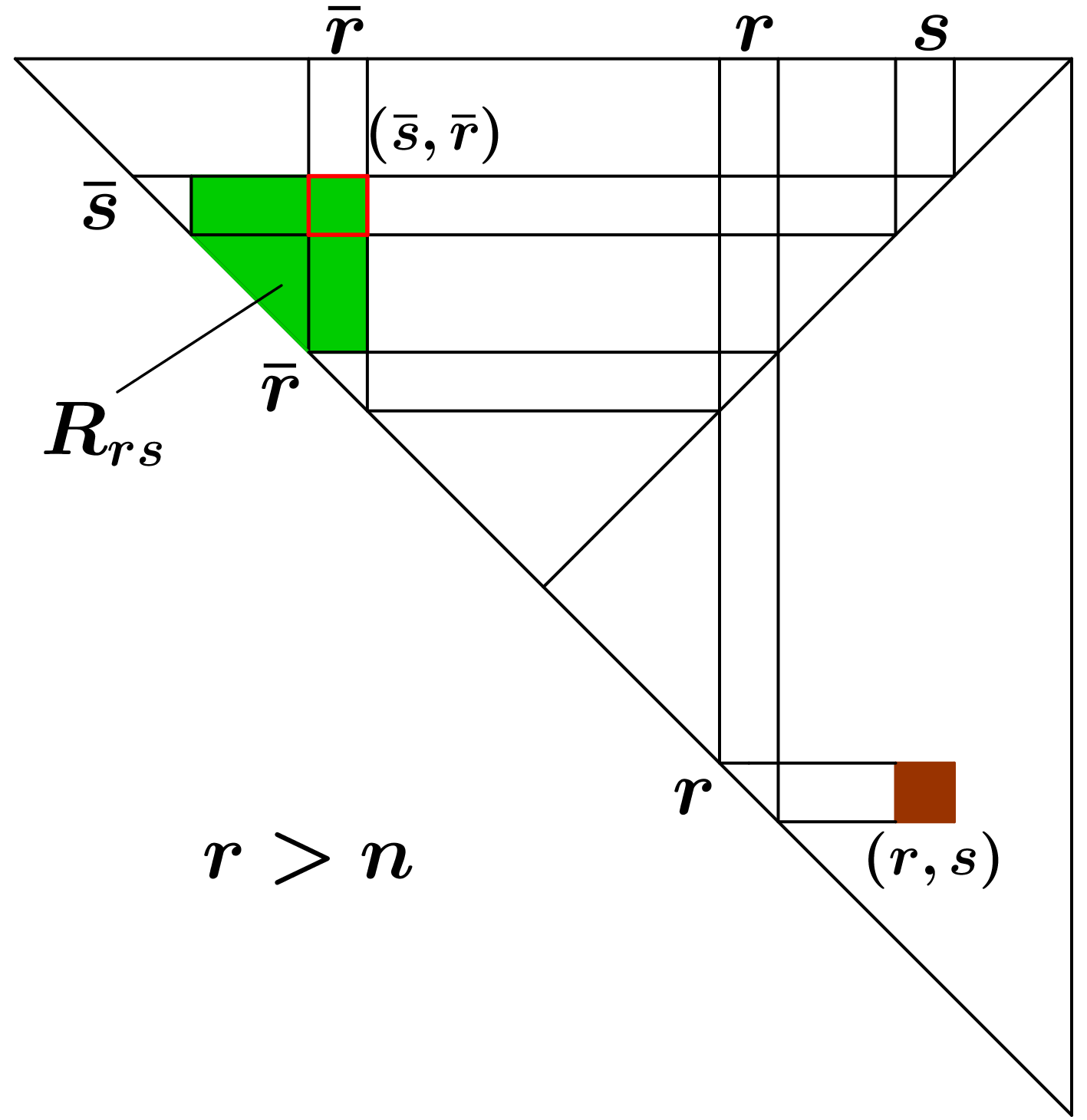}
\end{equation}

\begin{Notation}\label{typeArootsubgroups}
 We identify  the root system  of type $\fA_{N-1}$ with $\{(i,j)\,|\,1\leqslant i, j\leqslant N, i\neq j\}$ and denote it  by $\tilde \Phi$. Then $\tilde \Phi=\tilde \Phi^+\cup \tilde \Phi^-$, where $\tilde \Phi^+=\UR$ and $\tilde \Phi^-=\{(i,j)\in \tilde \Phi\,|\,i>j\}$. Let $\widetilde U = U_N(q)$ be the group of all upper unitriangular $N\times N$-matrices over $\F_q$ and $1 = 1_{\widetilde{U}}$ the identity element of $\widetilde U$. For $1\leqslant i<j \leqslant N$ and $\alpha\in \F_q$ let $\tilde x_{ij}(\alpha)=1+\alpha e_{ij}\in \widetilde U$.  Let $\tilde X_{ij}=\{\tilde x_{ij}(\alpha)\,|\, \alpha\in \F_q\}$, then $\tilde x_{ij}(\alpha)\tilde x_{ij}(\beta)=\tilde x_{ij}(\alpha+\beta)$ for $\alpha, \beta\in \F_q$, and hence $\tilde X_{ij}\cong (\F_q, +)$. These are the root subgroups of $\widetilde U$ of type $\fA_{n-1}$.\hfill$\square$
\end{Notation}
 For type $\fB_n, \fC_n$ and $\fD_n$, the  root subgroups of $U$ are given as follows: 
\begin{Defn}\label{rootsubgroupsU}
Let $(i,j) \in \pUP$. We define  \vspace{-.4cm}
\begin{center}\begin{tabular}{p{.14\textwidth}p{.49\textwidth}p{.33\textwidth}}
in type $\fB_n$: &  $x_{ij}(\alpha) = 1+\alpha e_{ij}- \alpha e_{\bar j\, \bar i} 
=\tilde x_{ij}(\alpha)\tilde x_{\bar j\, \bar i}(-\alpha)$ &  where $\alpha \in \F_q$,\; if $j \ne n+1$, \\[.2cm]
& $x_{i,n+1}(\alpha) = 1+\alpha e_{i, n+1} - \alpha e_{n+1,\bar i} - \frac{1}{2}\alpha^{2}e_{i, \bar i}$ &\\[.2cm]
& $\phantom{x_{i,n+1}(\alpha)}=\tilde x_{i,n+1}(\alpha)\tilde x_{n+1,\, \bar i}(-\alpha)\tilde x_{i \bar i} (-\frac{1}{2}\alpha^2)$ &   where $\alpha \in \F_q$, \\
& $\phantom{x_{i,n+1}(\alpha)}$\\[-.1cm]
in type $\fC_n$: &  $x_{ij}(\alpha) = 1+\alpha e_{ij}- \alpha e_{\bar j\, \bar i}=\tilde x_{ij}(\alpha)\tilde x_{\bar j\, \bar i}(-\alpha) $  &  where $\alpha \in \F_q$,\; if $j \leq n$, \\[.1cm]
&  $x_{ij}(\alpha) = 1+\alpha e_{ij} + \alpha e_{\bar j\, \bar i}=\tilde x_{ij}(\alpha)\tilde x_{\bar j\, \bar i}(\alpha)$ &  where $\alpha \in \F_q$,\; if $n < j <  \bar i$, \\[.1cm]
& $x_{i\bar i}(\alpha) = 1+\alpha e_{i\, \bar i}=\tilde x_{i \,\bar i}(\alpha)$ & where $\alpha \in \F_q$, \\[.3cm]
in type $\fD_n$: &  $x_{ij}(\alpha) = 1+\alpha e_{ij}- \alpha e_{\bar j\, \bar i}=\tilde x_{ij}(\alpha)\tilde x_{\bar j\, \bar i}(-\alpha) $ &  where $\alpha \in \F_q$.
\end{tabular}
\end{center}
We define $X_{ij}=\{x_{ij}(\alpha)\,|\, \alpha\in \F_q\}$. Then $X_{ij}\cong (\F_q, +)$ is the root subgroup of $U$ associated to the position $(i,j)\in \pUP$. \hfill$\square$
\end{Defn}

\begin{Defn}\label{Upattern}
 Let $J \subseteq \pUP$ be a set satisfying 
\begin{align*} \text{i}) && (i,j), (j,k) \in J &\quad\quad \Rightarrow \quad\quad (i,k) \in J  \\ \text{and} \quad \text{ii}) && (i,j),(\bar k, \bar j) \in J, \; (i,k) \in \pUP &\quad\quad \Rightarrow \quad\quad (i,k) \in J. \end{align*}
Then $J$ is a {closed subset} of $\pUP$.  In type $\fA_{n-1}$, $J\subseteq \UR$ is closed if condition i)  is satisfied. Note that  $(i,j), (j,k)\in \pUP$ implies $(i,k)\in \pUP$. Thus considered as type $\fA$-set, $\pUP$ is closed in $\UR= \tilde \Phi^+$.  \hfill $\square$
\end{Defn}

We have by Theorem 3.28 of [\cite{GJD}] :
 \begin{Theorem}\label{anyorder}\label{def of generator}
Let $J \subseteq \pUP$ (resp. $J\subseteq \UR$) be closed. Define the {\bfseries pattern subgroup} $U_J$ (resp. $\widetilde U_J$) to be the subgroup of $U$ (resp. $\widetilde U$) generated by all root subgroups $X_{ij}$ (resp. $\tilde X_{ij}$) with $(i,j) \in J$. For any fixed linear ordering on $J$,   each $u \in U_J$ (resp. $\widetilde U_J$) can be uniquely written as a product of ${x}_{ij}(\lambda)$'s (resp. $\tilde x_{ij}(\lambda)$'s), where $(i,j)$ runs through $J$ and $\lambda$ runs through $\F_q$ with the product taken in that fixed order. \hfill$\square$
\end{Theorem}

We briefly describe the monomial linearisation of the right regular representation of the group algebra $\C U$. We set $A.g = \pi(Ag)$, for $g\in \widetilde U$ and $A\in V=V_{\pUPs}$, where $Ag$ is  the ordinary matrix multiplication and again $\pi = \pi_{\pUPs}$. Then this defines an action of $\widetilde U$ on $V$ by vector space automorphisms. We fix once and for all a non-trivial linear character $\theta:\F_q \longrightarrow \C^*$ of the additive group of the field $\F_q$. With $V^* = \Hom_{\F_q}(V, \F_q)$ the space $\hat V = \Hom ((V,+),\C^*)$ of linear characters  of the additive group of $V$ is given as $\hat V = \{ \theta\circ \tau\,|\, \tau \in V^*\}$. From the action of $\widetilde U$ on $V$ we derive a permutation action $(\chi,g)\mapsto \chi.g$ of $\widetilde U$ on $\hat V$ defined by $\chi.g(A) = \chi(A.g^{-1})$ for $\chi\in\hat V, g\in\widetilde U$ and $A\in V$.

The coordinate functions $\epsilon_{ij}:V\longrightarrow \F_q: A\mapsto A_{ij}$ for $(i,j)\in \pUP$ are the basis elements  of $V^*$ dual to the natural basis of $V$ consisting of matrix units. Thus for every $\tau\in V^*$ we find a $B\in V$ such that $\tau= \sum_{(i,j)\in\pUPs}B_{ij}\epsilon_{ij}$. We then write $\tau = \tau_{_B}$ and define $\chi_{_B}\in \hat V$ to be $\theta\circ\tau_{_B}$.

In order to describe the action of $\widetilde U$ on $\hat V$ explicitly, we use the following standard bilinear form on matrices:

\begin{equation}\label{traceform}
\kappa: V_\square \times V_\square \rightarrow \F_q: (A, B)\mapsto tr(A^tB)=\sum_{(i,j)\in \square} A_{ij}B_{ij}=\sum_{(i,j)\in \supp A\cap \supp B} A_{ij}B_{ij}.
\end{equation}
Then  $\kappa|_{V_S\times V_S}$ is a non-degenerate symmetric bilinear form  for all $S\subseteq \square$. Moreover,
\[
\kappa(B^tA, C)=\kappa(A, BC)=\kappa(AC^t,B)
\quad\text{ for all }A,B, C\in V_\square.
\]
The linear character $\chi_{_A}$ for $A\in V$ is now given as
$$
\chi_{_A}=\theta\circ\kappa(A,-):V\longrightarrow \C^*:B\mapsto \theta(\kappa(A,B))\in \C^*.
$$
Moreover, for $g\in \widetilde U$ we have 
\begin{equation}\label{permacthatV}
\chi_{_A }.g (B) = \chi_{_A}(B.g^{-1}) = \theta\circ\kappa(A,B.g^{-1}) = \theta\circ\kappa(A.g^{-t},B),
\end{equation}
for all $B\in V$ and hence $\chi_{_A} .g = \chi_{A.g^{-t}}$.

The restriction $f$ of the map $\pi = \pi_{\pUPs}$ to $\widetilde U$ is a surjective right $1$-cocycle, that is it satisfies $f(gh) = f(g).h+f(h)$. Moreover its restriction to $U$ is bijective. We extend this map  by linearity to the respective group algebras and obtain a $\C$-linear map, which is not compatible with the action of the groups on its group algebras and their ``.''-action on $ \C \hat V$. In order to achieve this one has to deform the permutation action of $\widetilde U$  on $\hat V$ into a monomial action using the 1-cocycle  $f$:

\begin{Theorem}\label{monomial}[\cite{GJD}, 4.9].
For $g\in \widetilde U$ and $A\in V$ we define 
\[
\chi_{_A}g=z \chi_{_A}.g=z \chi_{_{A.g^{-t}}}=z {\chi_{\pi_{_{\pUPt}}(Ag^{-t})}},
\]
with $z=\theta(\kappa(A, \pi_{_{\pUPt}}(g^{-1})))=\chi_{_{A}}(f(g^{-1}))\in \C^*$. Extending this by linearity makes  $\C\hat V$ into a $\widetilde U$-module such that $f:\C \widetilde U \rightarrow \C\hat V$ is $\C \widetilde U$-linear. Moreover, the restriction of $f$ to $\C U$ defines a $\C U$-isomorphism from the right regular representation of $U$ onto $\C \hat V \cong \C^V$. \hfill$\square$
\end{Theorem}

\begin{Notation}\label{lide}
We use frequently the canonical $\C$-algebra isomorphism from $\C ^V$ to $\C V$ given by $\tau\mapsto \sum_{v\in V}\tau(v)v$. For $A\in V$ we may identify therefore
$\chi_{_{-A}} \in \C^V$ with $[A]:=\sum_{v\in V}\overline{ \chi_{_A}(v)} v\in \C V$, and set $[A]_{ij}=A_{ij}$ for $(i,j)\in \pUP$. Thus
$\hat V=\{[A]\,|\,A\in V\}$. 
 Let $A\in V, g\in \widetilde U$. Then by theorem \ref{monomial} we have $\chi_{_A}.g=\chi_{_B}$, where $B=\pi_{\pUPt}(Ag^{-t})$. Thus $[A].g=[\pi_{\pUPt}(Ag^{-t})]$.  \hfill$\square$
\end{Notation}

\begin{Remark}\label{star f}
 Using \ref{lide} it is not hard to check, that the inverse of the $U$-isomorphism $f:\C U_{\C U}\rightarrow \C\hat V\cong \C^V$ of \ref{monomial}  is given by $f^*:\tau \mapsto \tau\circ f \in \C^U\cong \C U$ for $\tau\in \C^V \cong \C V$. \hfill$\square$
\end{Remark}

\begin{Prop}\label{TruncatedColumnOperationG}[\cite{GJD}, 5.6].
Let $A\in V$ and let $\tilde x_{ij}(\alpha)\in U_N(q)$ with $1\leqslant i<j\leqslant N$ and $\alpha\in \F_q$. Then $[A].\tilde x_{ij}(\alpha)$ arises from $A$ by adding $-\alpha$ times column $j$ to column $i$ in $A$ and setting entries outside of $\pUP$ to zero. This action is called {\bf restricted column operation}. \hfill$\square$
\end{Prop}

\begin{Remark}\label{4action}
We shall picture $[A]$ for  $A\in V$ by a triangle $\pUP$ filled with elements of $\F_q$. 
By lemma \ref{rootsubgroupsU} the elements of $X_{ij}$ with $(i,j)\in \pUP$ can be written as products of elements of certain subgroups $\tilde X_{st}\in U_N(q)$ with $(s,t)\in \tilde \Phi^+$. Combing this with
\ref{TruncatedColumnOperationG} we can illustrate the ``$.$''-action of the root subgroups $X_{ij}$  of $U$ on $\hat V$ in theorem \ref{monomial} by the figures below, where $\tilde n=n+1$ for type $\fB_n$, $\tilde n=n$ otherwise, $\epsilon=-1$ for type $\fC_n$ and $\epsilon=1$ otherwise:
\end{Remark}

\begin{figure}[H]
\begin{center}
    \subfigure{\includegraphics[width=0.43\textwidth]{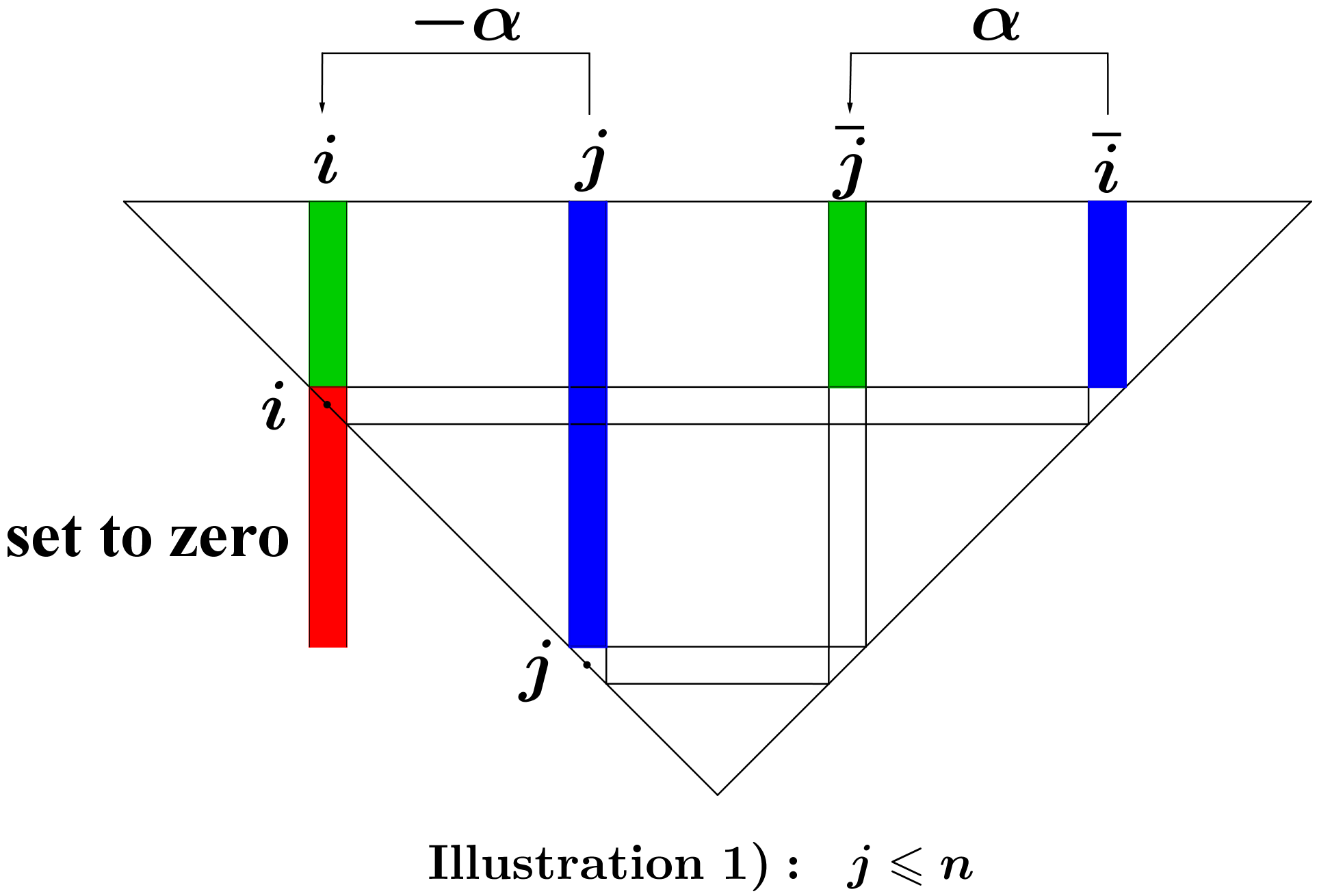}}
\quad
    \subfigure{\includegraphics[width=0.43\textwidth]{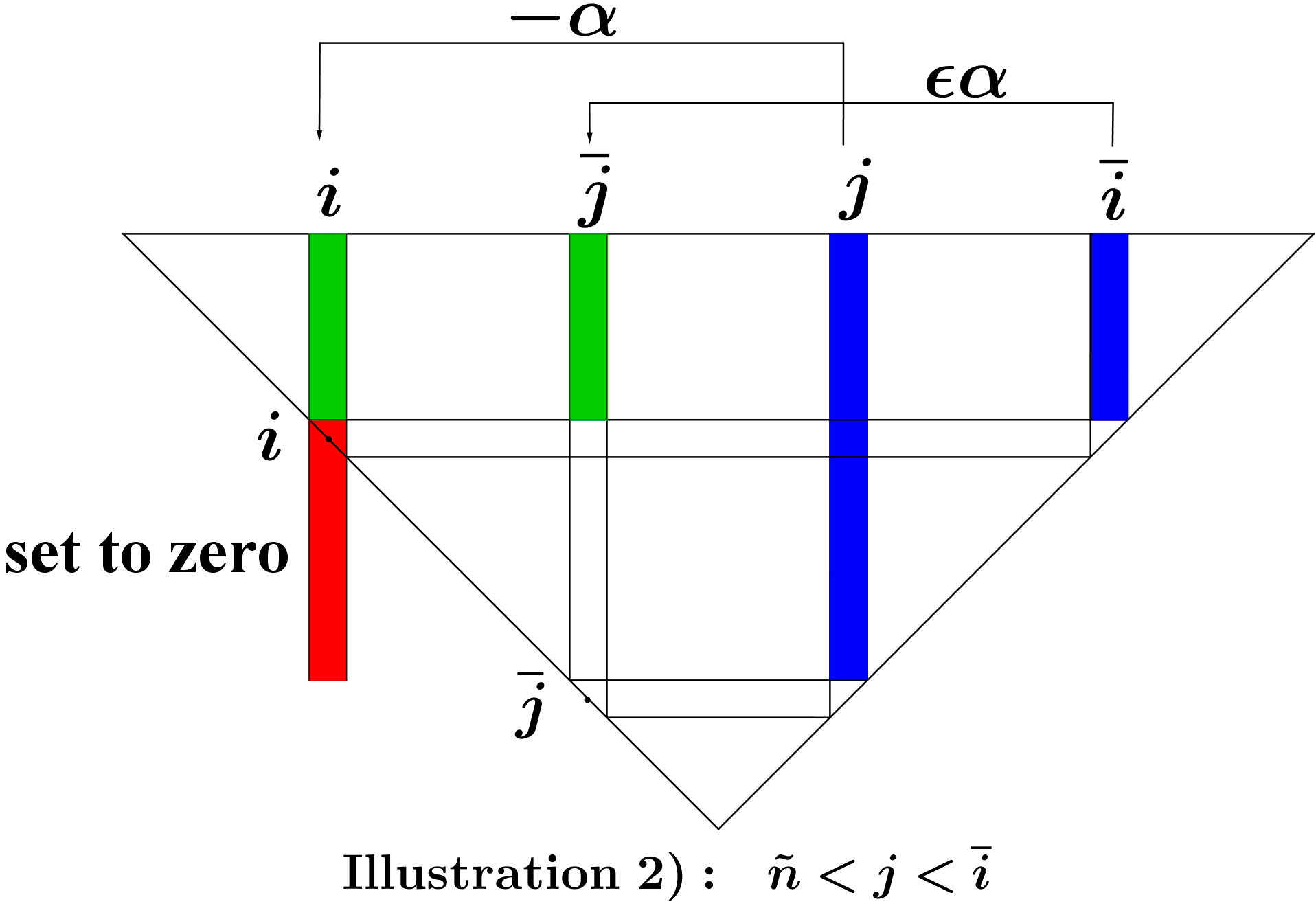}}
\end{center}
\end{figure}
\vspace{-.8cm}
\begin{figure}[H]
\begin{center}
    \subfigure{\includegraphics[width=0.42\textwidth]{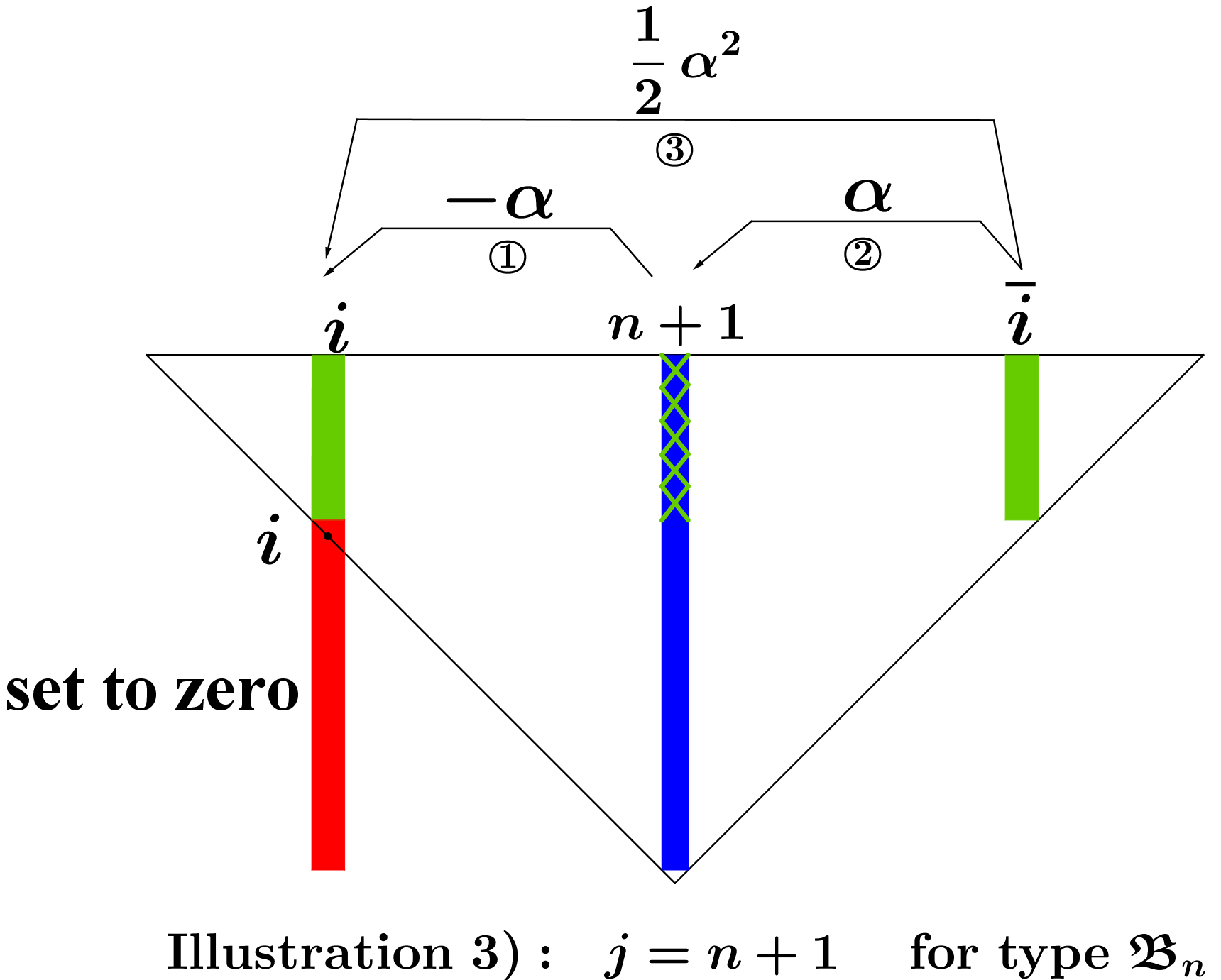}}
\quad
    \subfigure{\includegraphics[width=0.422\textwidth]{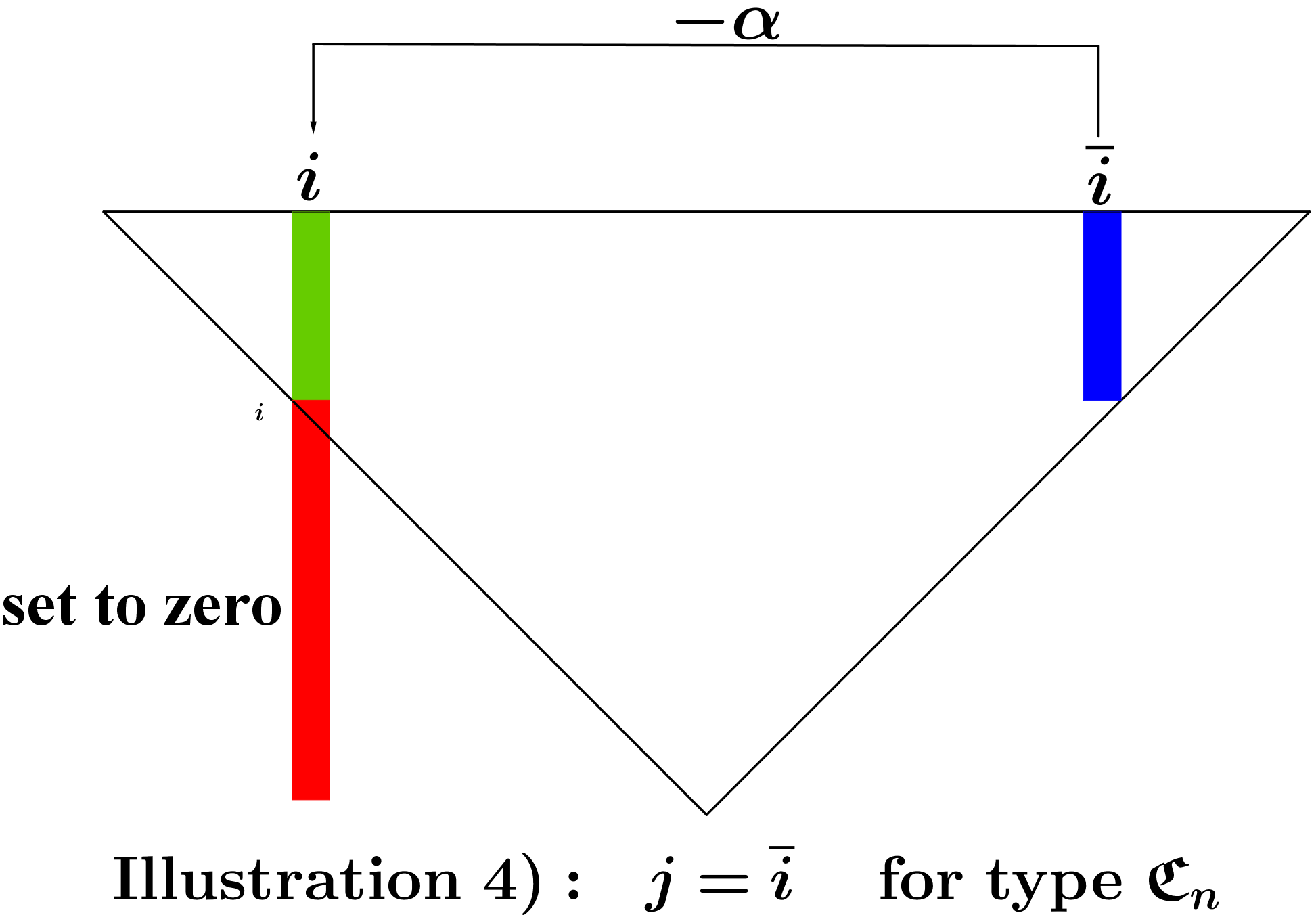}}
\end{center}
\end{figure}
\vspace{-1.5cm}
\begin{equation} \label{illTruncatedColumnOperationG}
\end{equation}

\begin{Remark}
All the illustrations above  are self-explaining, in view of lemma \ref{rootsubgroupsU} and proposition \ref{TruncatedColumnOperationG}, except maybe no. 3). Here we need to stick to the order given by \[x_{i, n+1}(\alpha)=\tilde x_{i,n+1}(\alpha)\tilde x_{n+1,\, \bar i}(-\alpha)\tilde x_{i \bar i} (-\frac{1}{2}\alpha^2),\]
labelled by \textcircled{1}, \textcircled{2} and \textcircled{3} in the illustration. \hfill$\square$
\end{Remark}

\begin{Remark}\label{actbyseq}
For $A\in V$ let $
\cO_A=\{[A].u\,|\,u\in U\}
$
be the {\bf $U$-orbit}
 of $[A]\in \hat V$ under the ``.''-action of $U$.
Note that in view of  \ref{rootsubgroupsU},  every $[B]\in \cO_{_A}$ ($A\in V$) can be obtained from $[A]$ by a sequence of restricted column operations in \ref{4action}. \hfill$\square$
\end{Remark}

Since for $1\leqslant i<j \leqslant N, A\in V$ and $\alpha\in \F_q$, the matrix $\tilde x_{ij}(\alpha) A$ is obtained by  adding $\alpha$ times row $j$ to row $i$ in $A$, the vector space $V=V_{\pUPs}$ is $\widetilde U$-invariant under left multiplication. 
Since $\pUP\subseteq \Phi^+$ is closed, $\widetilde U_{\pUPs}$ is a pattern subgroup of $\widetilde U$ with associated Lie algebra $V_{\pUPs} =\Lie(\widetilde U_{\pUPs})$ and right and left 1-cocycle $\pi_{\pUPt}=f:\widetilde U_{\pUPs}\rightarrow V_{\pUPs}:u\mapsto u-1$. By the left sided version of theorem \ref{monomial} $\C \hat V_{\pUPs}$ becomes a monomial left $\C \widetilde U_{\pUPs}$-module with monomial basis $\hat V_{\pUPs}$. The action of $\tilde x_{ij}(\alpha),\, (i,j)\in \pUP, \, \alpha\in \F_q$ on $\hat V_{\pUPs}$ is given by
\begin{equation}
\tilde x_{ij}(\alpha)[A]=\theta(\alpha A_{ij})[B]
\end{equation}
 where $B$ is obtained from $A$ by a {\bf restricted row operation}  which adds $-\alpha$ times row $i$ to row $j$ and projects the resulting matrix into $V_{\pUPs}$.

\begin{equation}\label{rowformula}
 \includegraphics[width=0.45\textwidth]{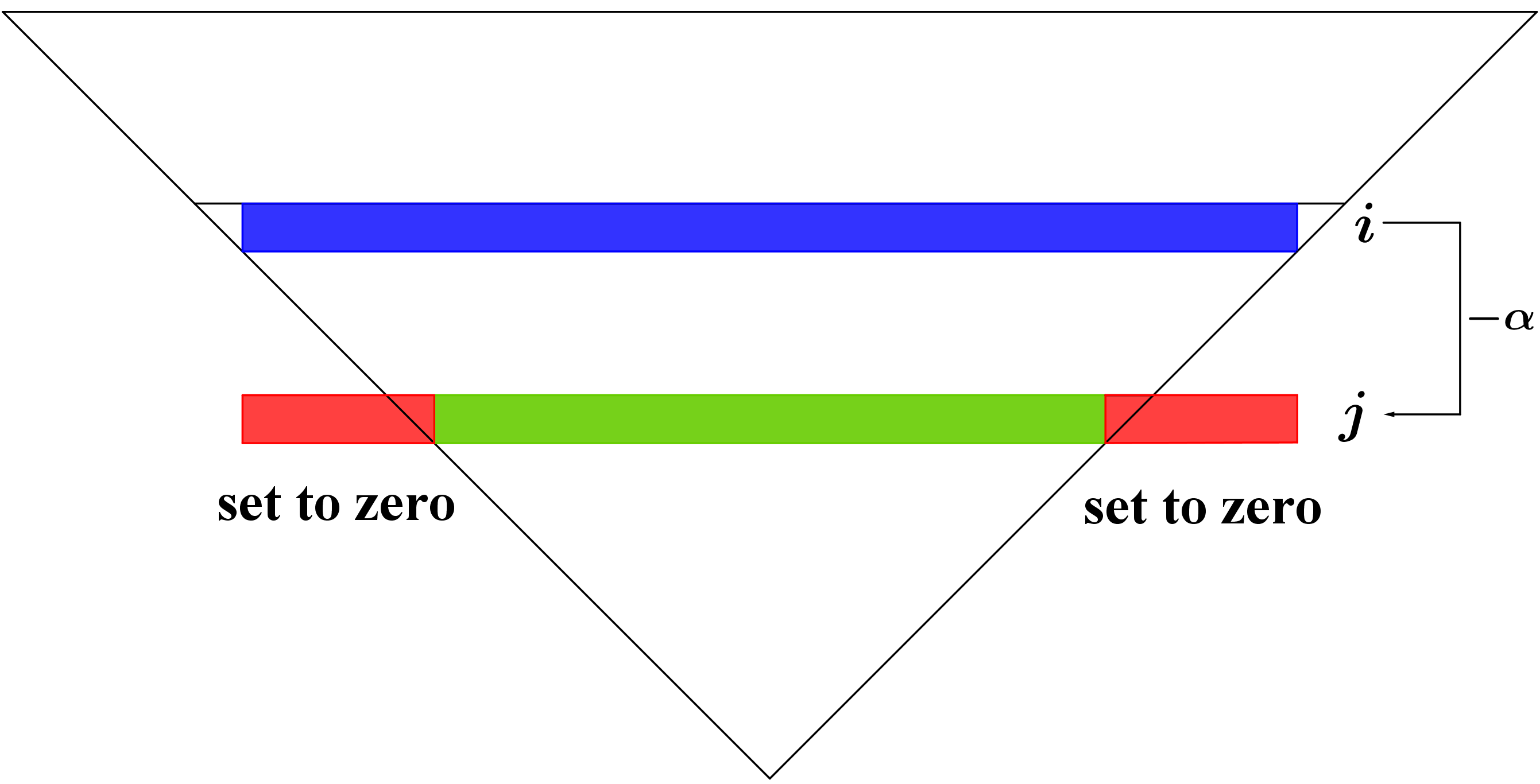}
\end{equation}
 
In general, the monomial left $\widetilde U_{\pUPs}$-action on $\hat V_{\pUPs}$ does not commute with the monomial right $U$-action, but there are special cases, where this holds.

 \begin{Notation}We set 
 \[
 \tril=\{(i,j)\in \pUP\,|\,j\leqslant \tilde n\}\quad \text{ and } \quad
 \trir=\{(i,j)\in \pUP\,|\,j>\tilde n\}.
 \]
Again $\tilde n=n \text{ for type } \fC_n, \fD_n\text{ and } \tilde n=n+1\text{ for type } \fB_n$. Note that $\pUP=\tril\quad\dot{\cup}\quad \trir$, and $\CC\subseteq \trir$ for type $\fC_n$.  \hfill$\square$
 \end{Notation}
 
\begin{Defn}
Suppose $A\in V$. We call $(i,j) \in \pUP$ a {\bfseries main condition} of $A$ (or of $[A]$) if $A_{ij}$ is the rightmost non-zero entry in the $i$-th row. We call a main condition $(i,j)$ {\bfseries left main condition} if $(i,j) \in \tril$,  and {\bfseries right main condition} if $(i,j)\in \trir$. Let
\begin{eqnarray*} 
\mc(A) & =& \big\{ (i,j) \in \pUP \; \big| \; \text{$(i,j)$ is a main condition of $A$} \big\}, \\
\lmc(A) & =& \big\{ (i,j) \in \pUP \; \big| \; \text{$(i,j)$ is a left main condition of $A$} \big\} \subseteq \tril, \\
\rmc(A) &= &\big\{ (i,j) \in \pUP \; \big| \; \text{$(i,j)$ is a right main condition of $A$}\big\}\subseteq \trir .
\end{eqnarray*} 
Note that we have $\mc(A)=\lmc(A)\;\dot\cup\; \rmc(A)$, and $\mc(A)=\mc(B)$ for all $[B]\in \cO_{_A}$ in view of remark \ref{actbyseq}.\hfill$\square$
\end{Defn}

\begin{Defn}
Let $[A]\in \hat V$.  We call $[A]$ a {\bfseries staircase character},  if the elements of $\mc(A)$  lie in different columns and adopt a similar notation for $U$- and $U_N(q)$-orbits $\cal{O}$, and for $U$- and $U_N(q)$-orbit modules $M$.\hfill$\square$
\end{Defn}

\begin{Defn}Let $[A]\in \hat V$ be a staircase character. 
\begin{itemize}
\item [1)]  The positions $(i, \bar j)$ with $(i,j)\in \rmc(A)$ are called {\bf minor conditions} of $A$ (or $[A]$, or $\cO_A$) and we denote the set of minor conditions of $[A]$ by $\minc(A)$ (or $ \minc(\cO_A)$). Obviously those depend only on $\rmc(A)$ and hence the minor conditions coincide for all $[B]\in \C \cO_{A}$ justifying the notation. Note that, $ (i, \bar j)\in \tril$ in types $\fB_n, \fD_n$, 	and for type $\fC_n$, $ (i, \bar j)$ is either in $ \tril$	or in $\diag$ if  $(i,\bar i )\in \rmc(A)$ and $j=\bar i$.		
\item [2)] A position $(i,j)\in \tril$ is  called {\bf supplementary condition} for $A$ or $\cO_{_A}$,  if $(i,j)$ is on the left of some minor condition or some left main condition of $A$, in the same column as some minor condition of $A$ and is not itself a minor or main condition.  For all Lie types the set of supplementary conditions is denoted by $\suppl (A)$. Note that $\suppl(A)\subseteq \tril\setminus\{\text{column $\tilde n$}\}$ where $\tilde n=n$ in types $\fC_n, \fD_n$ and  $\tilde n=n+1$ for type $\fB_n$. 
\item[3)] The {\bf core} of $A$ or $\cO_A$ is defined to be 
\[
\core( A)=\begin{cases}
\main(A)\cup \minc(A)\cup \suppl(A)  & \text{ if  $U$ is of type $\fB_n$ or $\fD_n$}\\
\main(A)\cup \suppl(A)  & \text{ if  $U$ is of type $\fC_n$}.
\end{cases}
\]
Note that $\core( A)$ is determined by  $\main(A)$.
\item [4)] We define the {\bf verge} of $A$ to be
$
\verge(A)=\sum_{(i,j)\in \main(A)} A_{ij}e_{ij}.
$
The linear character $[A]\in \hat V$ is called {\bf verge character}, if $A=\verge(A)$.
\item [5)]  $[A]\in \hat V$ is called {\bf core character}, if $\main(A)\subseteq \supp(A)\subseteq \core (A)$. \hfill$\square$ 
\end{itemize}
\end{Defn}

\begin{Defn}
Let $[A]\in \hat V$ be a staircase character, and let $(i,j)\in \pUP$. Then $(i,j)$ is called a {\bf place} of $A$ if $(i,j)\notin \core(A)$ and it is to the left of a main condition. 
Note the set of places of $A$ is determined by $\main(A)$ uniquely. It is hence denoted by $\Pl(A)=\Pl(\cO_A)=\Pl(\main(A))$.\hfill$\square$
\end{Defn}

\begin{Defn}\label{lowerhook}
For $[A]\in \hat V$ and $(i,j)\in \pUP$ we define:
\begin{eqnarray*}
\Arm:\quad\cA(i,j)=\{(\bar j, a)\in \UP\};\quad\quad
\Leg:\quad\cL(i,j)=\{(a,j)\in \UP\,|\,a>i\}
\end{eqnarray*}
\[\Limb^\circ(A)=\bigcup\nolimits_{(i,j)\in \mc(A)}\cA(i,j)\cup \cL(i,j)\]
In  type $\fC_n$, we set $\fC(A)=\{(\bar j, j)\in \CC\,|\, (i,j)\in \rmc(A)\cap \UP\}.$ In addition, we define
\[
\Limb(A)=\begin{cases}\Limb^\circ(A), &\text{ if } U \text{ is of type } \fB_n \text{ or } \fD_n,\\ \Limb^\circ(A)\,\dot \cup\, \fC(A), &\text{ if } U \text{ is of type } \fC_n, \end{cases}
\]
and $J(A)=\pUP\setminus \Limb(A)$. \hfill$\square$
\end{Defn}

Here is an illustration of  $\cL(i,j)$ and $\cA(i,j)$ with $(i,j)\in \mc(A)\cap \UP$, where $M$ stands for main and $m$ for minor condition: 
\begin{equation}\label{lowerhookpic}
\includegraphics[width=0.436\textwidth]{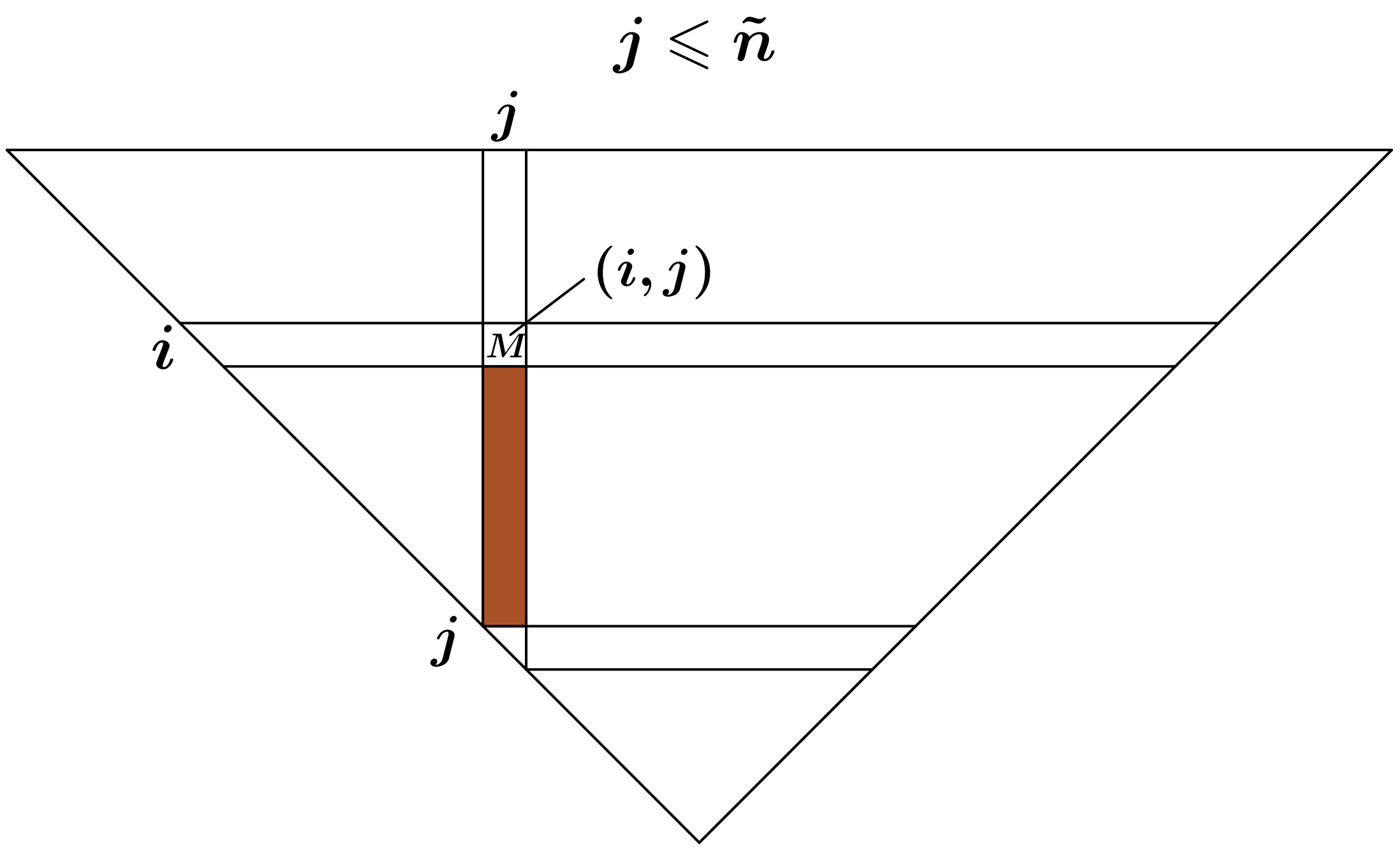}
\quad
\includegraphics[width=0.47\textwidth]{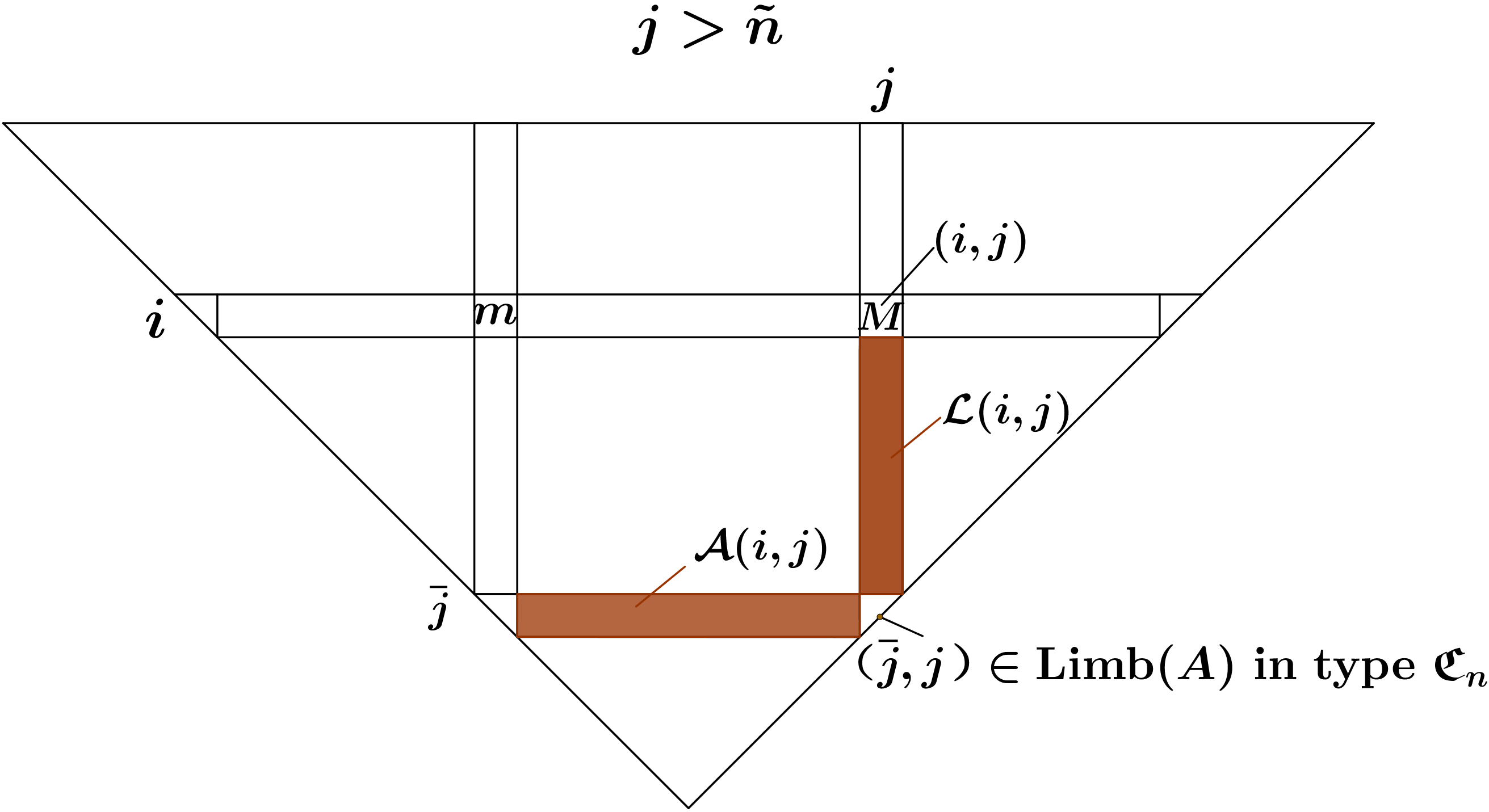}
\end{equation}

\section{Left $U$-action and separated main conditions}\label{secleftu}
In [\cite{GJD}] we have seen that every orbit module contained in $\C \hat V$ is isomorphic to some staircase orbit module. Since $\C \hat V$ is isomorphic to the right regular representation $\C U_{\C U}$ we conclude, that every irreducible $\C U$-module occurs as irreducible constituent in some staircase orbit module. In an analogous construction for $\widetilde U=U_N(q)$ (done in [\cite{andre}] and [\cite{yan}]), all orbits modules are as well isomorphic to staircase ones and all orbit modules are either isomorphic or afford orthogonal characters. We would like to have an analogous statement for our orbit modules in $\C \hat V\cong \C U_{\C U}$. Unfortunately this is not true in general, there exist non-isomorphic staircase orbit modules, which do have irreducible constituents in common. However, there exists a subclass of staircase orbits, called {\bf main separated}, satisfying that their orbit modules are either isomorphic or afford orthogonal characters, provided the main conditions are contained in $\UP$. Moreover, every irreducible $\C U$-module is constituent of one of those, if $U$ is of type $\fB_n$ or $\fD_n$. 
That main conditions are separated means that we have no main condition on arm $\cA(i,j)$ in illustration \ref{lowerhookpic} (where $i<\bar j$). We do, however allow position $(\bar j,j)$ to be a main condition, if column $j$ has no main condition $(i,j)$ with $i<\bar j$. More precisely: 

\begin{Defn}
 The the main conditions $\main(A)$ of a character $[A]\in \hat V$  are called {\bf separated} if 
${\Limb}(A)\cap \main(A)=\emptyset$. If so, we call $[A]$ and $\cO_A$ {\bf main separated}.\hfill$\square$
\end{Defn}

 Note, that if $[A]$ is main separated, then $[A]$ is staircase as well, since a character $[A]$ which is not staircase, contains at least one column with  two main conditions, such that the lower one is on the leg of the higher one or is on the anti-diagonal with the higher one being a right main condition. 
 
As a first application we have the following result, which later on will turn out to be very useful:

\begin{Lemma}\label{HookArmsZero} Let $[A]\in\hat V$ be main separated and let $1\leq i<\tilde n$ be such that column $\bar i$ contains a main condition $(k,\bar i)$ with $k< i$.   Then for all $[B]\in \cO_A$  the $i$-th row of $B$  is a zero row.\hfill$\square$
\end{Lemma}
\begin{proof} Under the assumption of the lemma $\text{row } i\cap \UP$ of $A$ is the arm $\cA(k,\bar i)$. Since $k<i$, position $(i,\bar i)$ cannot be a main condition, since $[A]$ is staircase. Since $[A]$ is main separated, there can be no main condition in row $i$ of $B$ for any $[B]\in\cO_A$ and the result follows.
\end{proof}

By general theory, every $U$-linear map from any right ideal $I$ of $\C U$ into $\C U$ is obtained by left multiplication $\lambda_x: I\longrightarrow \C U: y\mapsto xy$ by some $x\in \C U$, since $\C U$ is a self-injective algebra. We may use the right $\C U$-module isomorphism $f: \C U\longrightarrow \C \hat V\cong \C^V$ of \ref{monomial} and its inverse $f^{-1}=f^*$ of \ref{star f} to induce  a right $U$-linear left action $\lambda_x$ for $x\in U$ on $\C \hat V$ through multiplication of $x$ on $\C U$. We obtain:
\begin{Lemma}\label{leftactionall}
Let $x\in U, A\in V$. Then, identifying  $\C U_{\C U}$ and $\C \hat V$ by $f$ and $f^*= f^{-1}$ we have:
\begin{equation}\label{leftmult1}
\lambda_x [A]=\sum\nolimits_{u\in U}{ \theta\circ \kappa(-x^{-t}A, u)}\pi(u).
\end{equation}
\end{Lemma}
\begin{proof}
Recall that $[A]=\sum_{B\in V}{ \chi_{_{-A}}(B)} B$. Then
\begin{eqnarray*}
\lambda_x f^*([A])&=&\lambda_x f^*(\sum_{B\in V}{ \chi_{_{-A}}(B)} B)
=\lambda_x \sum_{B\in V}{ \theta\circ \kappa(-A, B)} f^*(B)\\
&=&\lambda_x \sum_{u\in U}{ \theta\circ \kappa(-A, \pi(u))} u
= \sum_{u\in U}{ \theta\circ \kappa(-A, u)}x u\\
&=& \sum_{u\in U}{ \theta\circ \kappa(-A, x^{-1}u)}u=\sum_{u\in U}{ \theta\circ \kappa(-x^{-t}A, u)}u,
\end{eqnarray*}
and the claim follows by applying $f$ on both sides of the formula above.
\end{proof}

\begin{Remark}\label{TypeAleftaction}
Note that $V=V_{\pUPs}$ is invariant under left multiplication by elements of $\widetilde U$. Hence $\widetilde U$ acts on $\C^V$. Indeed it permutes $\hat V \subseteq \C^V$. More precisely $u.\tau$ with $ u \in \widetilde U, \tau\in\C^V$ acts on $V$ by $u.\tau(B) = \tau (u^{-1}B)$ for $B\in V$. For the character $[A]\in\hat V, A\in V$ we have then:
\begin{eqnarray*}
(u.[A])(B)&=&\theta\circ \kappa(-A,u^{-1}B)=\theta\circ \kappa(-u^{-t}A,B)\\
&=&\theta\circ \kappa(-\pi(u^{-t}A),B)=[u^{-t}.A](B)
\end{eqnarray*}
setting $u^{-t}.A=\pi(u^{-t}A)$.

However, note that in the bilinear form of equation \ref{leftmult1} we cannot replace  $u$ by $\pi(u)$, since the intersection of the support of $-x^{-t}A$ and of $u$ is not contained in $\pUP$ in general. In fact, the projection $\pi = \pi_{\pUPs}$ is not a left $1$-cocycle on $\widetilde U$, not even on $U$ (but on $U_{\pUPs}$, see [\cite{GJD}, Lemma 6.4]). The left action by $\lambda_x$ does not take $[A]$ to a multiple of $x.[A]=[u^{-t}.A]$ in general, but into a linear combination of many characters $[C]\in \hat V$. We shall see an example, but first we present a condition in \ref{leftactionSupp} below, such that $\lambda_u$ with $u\in U$ acts  monomially on $\hat V$. \hfill$\square$
\end{Remark}

\begin{Notation}
Let $\KL=\{(i,j)\in \square\,|\, 1\leqslant j <\bar i \}$. Thus $\KL$ consists of all positions in $\square$  above the anti-diagonal.  Set $\pKL=\KL$ for types $\fB_n$ and $\fD_n$,  and $\pKL=\KL\cup \CC$ for type $\fC_n$. \hfill$\square$
\end{Notation}

\begin{Cor}\label{leftactionSupp}
Let $x\in U, A\in V$ such that $\supp(x^{-t}A)\subseteq \pKL$. Then we have
\[
\lambda_x [B]=\chi_{_{-B}}(x^{-1})[x^{-t}.B] \quad \text{ for all } B\in \cO_A.
\]
\end{Cor}
\begin{proof} Let $[B] \in \cO_A$. It is easy to see, that $x^{-t}B$ arises from $B$ by taking $B$ and adding scalar multiples of rows to lower rows. Since for each $[B] \in \cO_A$, the most right hand side non zero entries are on the main conditions with $B_{ij}=A_{ij}$ for all $(i,j)\in \main(A)=\main(B)$.  Hence  if $\supp(x^{-t}A)\subseteq \pKL$, then $\supp(x^{-t}B)\subseteq \pKL$ for all $[B] \in \cO_A$. By \ref{leftactionall}, we have:
\begin{eqnarray}
\lambda_x [B]&=&\sum\nolimits_{u\in U}{ \theta\circ \kappa(-x^{-t}B, u)}\pi(u)=\sum\nolimits_{u\in U}{ \theta\circ \kappa(-x^{-t}B, u-1+1)}\pi(u)
\nonumber\\&=&\theta\circ\kappa(-x^{-t}B,1)\cdot\sum\nolimits_{u\in U}{ \theta\circ \kappa(-x^{-t}B, u-1)\pi(u)}
\nonumber\\&=&\theta\circ\kappa(-B,x^{-1})\cdot\sum\nolimits_{u\in U}{ \theta\circ \kappa(-x^{-t}B, u-1)\pi(u)}
\nonumber\\&=&\chi_{_{-B}}(x^{-1})\cdot\sum\nolimits_{u\in U}{ \theta\circ \kappa\big(-\pi(x^{-t}B), \pi(u)\big)\pi(u)} \label{leftactionsupp1}
\end{eqnarray}
since $\supp(x^{-t}B)\cap\supp(u-1)\subseteq \pUP$ and $\pi(u)=\pi(u-1)$.
Then the statement holds observing the following equation:
\begin{equation*}
\sum_{u\in U}\theta\circ \kappa\big(-\pi(x^{-t}B), \pi(u)\big)\pi(u)=\sum_{\pi(u)\in V}{ \chi_{-\pi(x^{-t}B)}\big(\pi(u)\big)\pi(u)}=[\pi(x^{-t}B)]=[x^{-t}.B].
\end{equation*}
\end{proof}

\begin{Remark}\label{isotostaircase}
 The left action  of $u\in \widetilde U_{\pUPs}=\{u\in \widetilde U\,|\, \supp(u-1)\subseteq \pUP\}\leqslant \widetilde U$ on $\cO_A, A\in V$ defined in \ref{rowformula} commutes with the right $U$-action provided $\supp(u^{-t}A)\subseteq \pKL$ by [\cite{GJD}, Lemma 6.4]. We used this to prove in  [\cite{GJD}, Corollary 6.8] that every orbit module $\C \cO_A$ is isomorphic to a staircase orbit module, for $[A]\in V$. Replacing $\widetilde U_{\pUPs}$ by $U$ and applying \ref{leftactionSupp} yields this result as well by an easy calculation basically following the steps in the proof of [\cite{GJD}, Lemma 6.4].
\end{Remark}

\begin{Remark}\label{LeftTildeUandUAction}
 For $(i,j)\in \pUP$ and $\alpha\in \F_q$ one sees easily that $x_{ij}(\alpha).[A]=\tilde x_{ij}(\alpha).[A]$. Thus the left action of $U$ on $\hat V $ is essentially the same as the left action of $\widetilde U_{\pUPs}$ on $\hat V$. More precisely, if
\[
u=\prod\nolimits_{(i,j)\in \pUPs} x_{ij}(\alpha_{ij})\in U, \quad \alpha_{ij}\in \F_q,
\]
we set 
\[
\tilde u=\prod\nolimits_{(i,j)\in \pUPs} \tilde  x_{ij}(\alpha_{ij})\in \widetilde U_{\pUPs}\leqslant U_N(q).
\]
Then $u.[A]=\tilde u.[A]$ for all $u\in U$. We define $u[A]=\lambda_u[A]=\chi_{_{-A}}(u^{-1})u.[A]$ if $\supp(u^{-t}A)\subseteq \pKL$. \hfill$\square$
\end{Remark}

We next use left multiplication $\lambda_x, x\in U$ on staircase orbit modules $\C \cO_A$. $[A]\in \hat V$ to embed $\C \cO_A$ into a direct sum of main separated orbit modules. Since $\C \hat V \cong \C U$ we have
\begin{equation}\label{E1}
\lambda_x [A]=\sum_{C\in V} \mu_{_C}[C]
\end{equation}
with uniquely determined $\mu_{_C}\in \C$ for all $C\in V$. 
Our first auxiliary result gives a formula for $\mu_{_C}$:
\begin{Lemma}\label{Aux1}
In \ref{E1} we have:
$
\mu_{_C}=|U|^{-1}\sum_{u\in U} \theta\kappa(-x^{-t}A, u)\chi_{_C}(\pi(u)).
$
\end{Lemma}
\begin{proof}
Recall that $\chi_{_C}(\pi(u))=\theta\kappa(C,\pi(u))=\theta\kappa(C,u)$, since $\supp(C)\subseteq \pUP$. Replacing in \ref{E1} $C$ by $B\in V$ and calculating $[B]$ we get, using \ref{leftactionall}
\begin{eqnarray*}
\sum_{B\in V}\mu_{_B}[B]&=&\sum_{B\in V}\mu_{_B}\sum_{C\in V}{ \theta\kappa(-B,C)} C=\sum_{u \in U}\big(\sum_{B\in V}\mu_{_B} \theta\kappa(-B,u)\big) \pi(u)\\&=&\lambda_x [A]=\sum_{u\in U}{ \theta\kappa(-x^{-t}A, u)}\pi(u),
\end{eqnarray*}
and hence, comparing coefficients
\begin{equation}\label{E2}
\theta\kappa(-x^{-t}A, u)=\sum_{B\in V} \mu_{_B}\theta\kappa(-B,u)=\sum_{B\in V}\mu_{_B} \chi_{_{-B}}(\pi(u)).
\end{equation}
We apply orthogonality relations and obtain, where $\langle \chi_{_C}, \chi_{_B} \rangle$ denotes the standard inner product of the irreducible character $\chi_{_B}, \chi_{_C}, B, C\in V$.
\begin{eqnarray*}
\mu_{_C}&=&\sum_{B\in V}\mu_{_B}\langle \chi_{_C}, \chi_{_B} \rangle
=|V|^{-1}\sum_{B\in V} \sum_{D\in V} \mu_{_B}\chi_{_C}(D)\overline{\chi_{_B}(D)}\\
&=&|V|^{-1} \sum_{D\in V}\big(\sum_{B\in V} \mu_{_B}{\chi_{_{-B}}(D)}\,\big)\chi_{_C}(D)\\&=&|U|^{-1} \sum_{u\in U}\big(\sum_{B\in V} \mu_{_B}{\chi_{_{-B}}(\pi(u))}\,\big)\chi_{_C}(\pi(u))\\&\stackrel{\ref{E2}}{=}&
|U|^{-1} \sum_{u\in U} \theta\kappa(-x^{-t}A, u)\chi_{_C}(\pi(u)),
\end{eqnarray*}
as desired.
\end{proof}

\begin{Situation}\label{situation}
In the following we shall be concerned with the following special setup:
\begin{itemize}
\item $[A]\in \hat V$ is a staircase core character.
\item $(i,k)\in \rmc(A)$, hence in particular $\tilde n<k\leq N$. (We mark this position blue in \ref{leftmul2} below).
\item $\bar k<j<k,$ hence $(j,k)\in \RP$.
\item $x=x_{ij}(-\lambda A_{ik}^{-1})\in U$ for some $\lambda\in \F_q$. \hfill$\square$
\end{itemize}
\end{Situation}
Let $\tilde x=\tilde x_{ij}(-\lambda A_{ik}^{-1})$.  Then $( x^{-t}A)_{ab}=(\tilde x^{-t}A)_{ab}$ for all $(a,b)\in \UR \cup\diag$ and hence in \ref{Aux1} we can replace $\theta\kappa(-x^{-t}A, u)$ by $\theta\kappa(-\tilde x^{-t}A, u)$.  
Note that $\tilde x^{-t}A$  is obtained by adding $\lambda A_{ik}^{-1}$ times row $i$ of $A$ to row $j$ in $A$, as illustrated below:
\begin{equation}\label{leftmul2} 
\includegraphics[width=0.45\textwidth]{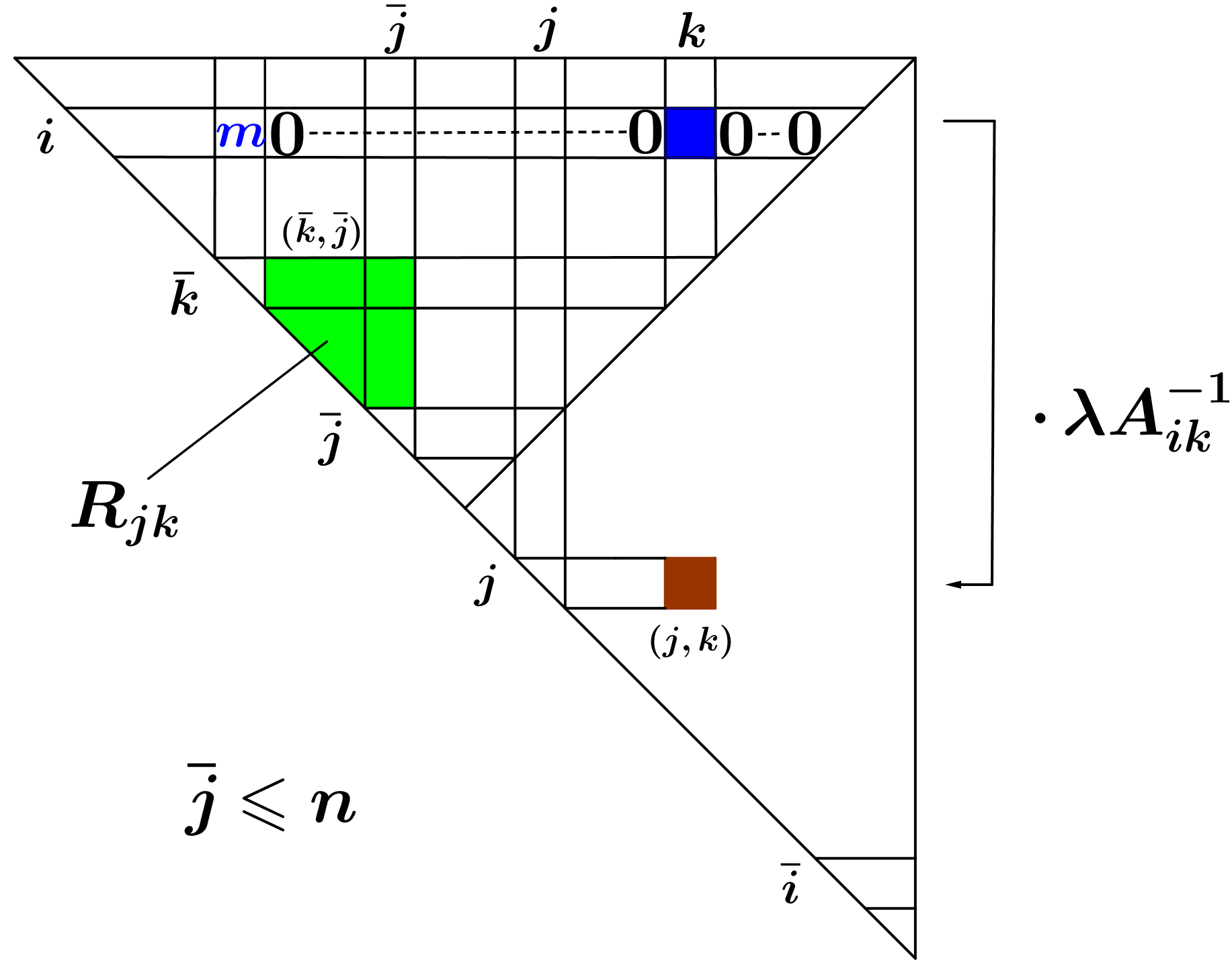}
\quad\quad
\includegraphics[width=0.442\textwidth]{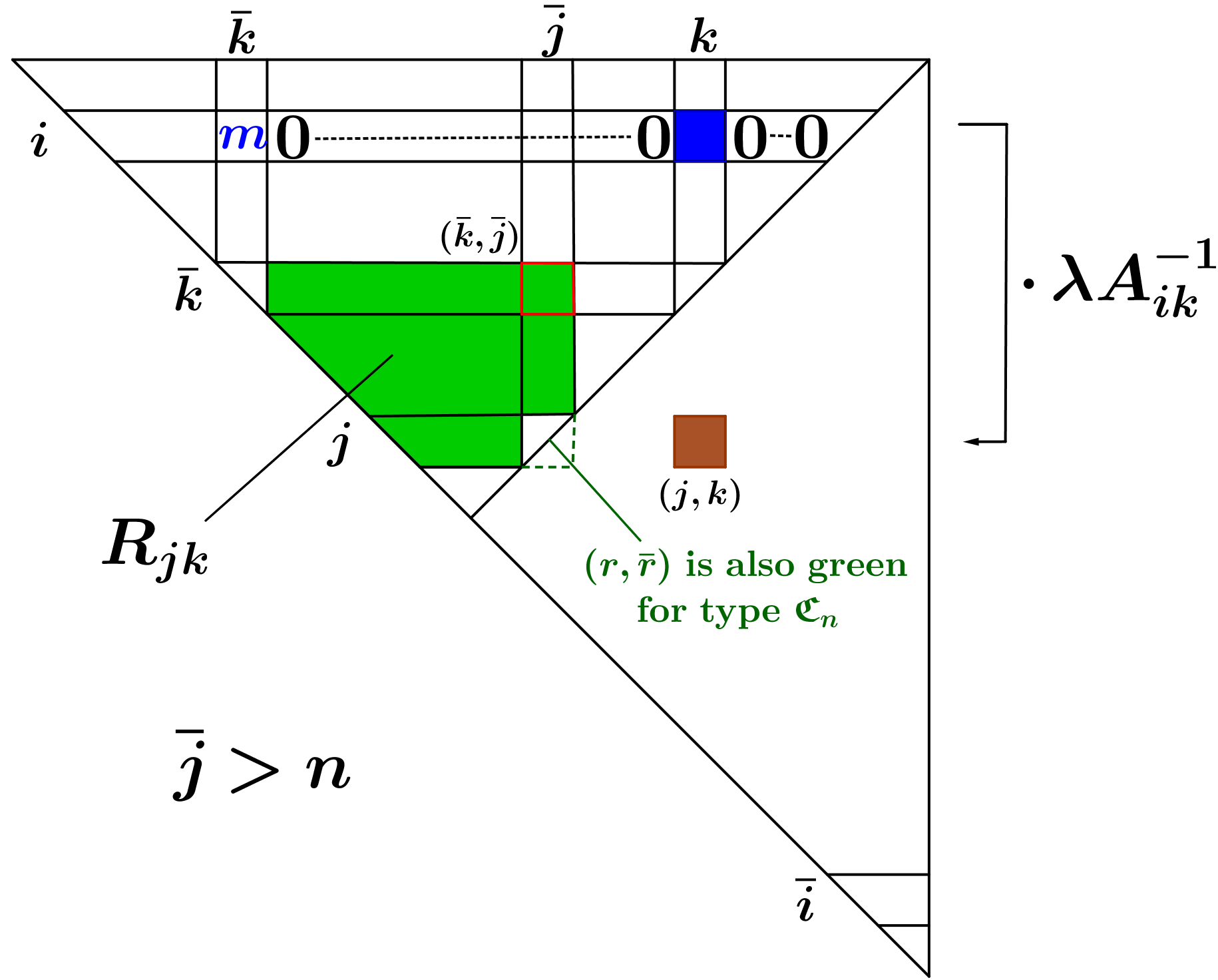}
\end{equation}

\medskip

Since $[A]$ is a core, $A$ and $x^{-t}A$ coincide at all positions in $\UR\cup \diag$ except the position  $(j,k)$ where
$(x^{-t}A)_{jk} = A_{jk}+ \lambda e_{jk}.$
Since $\supp(u)\in \UR\cup \diag$, we have in view of \ref{traceform}
\begin{equation}\label{leftmult4}
\theta\kappa(x^{-t}A,u) = \theta\kappa(A,u)\theta(\lambda u_{jk}) = \prod_{(a,b)\in\pUPs}\theta(A_{ab}u_{ab})\theta(\lambda u_{jk}).
\end{equation}

By \ref{unique} (see illustration \ref{dependleft}) the entry $u_{jk}$ of $u\in U$ at the brown position in  \ref{leftmul2} above can be expressed as $\pm u_{\bar k\bar j} + g(u)$, where $g: U \longrightarrow \F_q$ is a polynomial function with coefficients in $\F_q$ in the entries of $u\in U$ at positions in the set $R_{jk}\setminus \{(\bar k,\bar j)\}\subseteq \pUP$  marked green in \ref{leftmul2} (see also \ref{dependleft}). We shall denote this set now by $K$. Likewise we set $J= \pUP\setminus R_{jk}$. As a consequence we obtain:
\begin{equation}\label{leftmult5}
\theta(\lambda u_{jk}) = \theta(\pm\lambda u_{\bar k\bar j})\theta(\lambda g(u)).
\end{equation}

Observing that by \ref{traceform} 
\begin{equation}\label{leftmult6}
\chi_{_C}(u) = \prod_{(a,b)\in \pUPs} \theta(C_{ab}u_{ab}),
\end{equation}
and dividing up $\pUP$ into a disjoint union $\pUP = J\cup\{(\bar k,\bar j)\}\cup K$ equations \ref{leftmult4}, \ref{leftmult5} and \ref{leftmult6} imply immediately
\begin{eqnarray*}\label{leftmult7}
\theta\kappa(-x^{-t}A,u)\chi_{_C}(u) &=& \prod_{(a,b)\in\pUPs}\theta(-A_{ab}u_{ab})\theta(\lambda u_{jk})\cdot\prod_{(a,b)\in \pUPs} \theta(C_{ab}u_{ab})\\
&=& \Big(\prod_{(a,b)\in J}\theta\big((C-A)_{ab}u_{ab}\big)\Big)\Big(\theta\big((C_{\bar k\bar j}\pm\lambda-A_{\bar k\bar j})u_{\bar k\bar j}\big)\Big)\cdot \psi_K(u).
\end{eqnarray*}
 Here $\psi_K: V\longrightarrow \C$ is a map such that $\psi_K(u)$ for $u\in U$ depends only on the entries of $u$ at positions in $K=R_{jk}\setminus\{(\bar k, \bar j)\}$, and hence only on $V_K$.
Applying  \ref{Aux1} we have shown:

\begin{Lemma}\label{Aux2}
Under the assumptions of \ref{situation} and for $C\in V$ we have: 
\begin{eqnarray}
\mu_{_C} &=&|U|^{-1}\sum_{u\in U} \theta\kappa(-x^{-t}A, u)\chi_{_C}(\pi(u))\nonumber\\
&=&|U|^{-1} \sum_{u\in U}\Big(\prod_{(a,b)\in J}\theta\big((C-A)_{ab}u_{ab}\big)\Big)\Big(\theta\big((C_{\bar k\bar j}\pm\lambda-A_{\bar k\bar j})u_{\bar k\bar j}\big)\Big)\cdot \psi_K(u) \label{leftmult9}.
\end{eqnarray}

\end{Lemma}

We can now formulate and prove our third and last auxiliary result:

\begin{Lemma}\label{Aux3}
In the situation of \ref{situation} the coefficient $\mu_{_C}$ in $\lambda_x[A]=\sum_{C\in V}\mu_{_C}[C]$ is given as:
\begin{equation}\label{leftmult10}
\mu_{_C} = \Big(\prod_{(a,b)\in J}q^{-1}\sum_{\alpha\in\F_q}\theta\big((C-A)_{_{ab}}\alpha\big)\Big)\cdot\Big(q^{-1}\sum_{\alpha\in\F_q}\theta\big((C_{\bar k\bar j}\pm\lambda-A_{\bar k\bar j})\alpha\big)\Big)\cdot\Big(|V_K|^{-1}\sum_{X\in V_K}\psi_K(X)\Big),
\end{equation}
and hence using the Kronecker delta
\begin{equation}\label{leftmult11}
\mu_{_C} = \Big(\prod_{(a,b)\in J}\delta_{C_{ab},A_{ab}}\Big)\cdot\delta_{C_{\bar k\bar j},A_{\bar k\bar j}\pm\lambda}\cdot\Big(|V_K|^{-1}\sum_{X\in V_K}\psi_K(X)\Big).
\end{equation}
\end{Lemma}
\begin{proof} Note that in the sums in \ref{leftmult10}, each combination of $\alpha\in\F_q$ for each $(a,b)\in J$ and $\beta\in\F_q$ together with an $X\in V_K$  provides precisely one  entry for each position in $\pUP$ and hence gives precisely one $u\in U$. Therefore, multiplying out the products of sums in \ref{leftmult10} gives altogether a summation over $u\in U$ and, as can be seen by direct inspection, the right hand side of equation \ref{leftmult9}. 


Now applying the well known formula $\sum_{\alpha\in\F_q}\theta(\alpha)=0$, (saying, that the trivial character of the additive group of $\F_q$ and $\theta$ are orthogonal) and observing, that $\theta(0) = 1$, we conclude, that the first two sums in \ref{leftmult10} are zero and hence $\mu_{_C}=0$ unless the summands in those sums itself are all zero, yielding equation \ref{leftmult11}.
\end{proof}

This implies immediately:
\begin{Cor} \label{clearconnectedmain}
In the situation of \ref{situation} let $\lambda_x[A]=\sum_{C\in V}\mu_{_{_C}}[C]$. Then $\mu_C\neq 0$ implies:
\begin{enumerate}
\item[i)] $A$ und $C$ coincide at all positions $(a,b)\in J$.
\item[ii)] $C_{\bar k\bar j} = A_{\bar k\bar j}\pm\lambda$.
\end{enumerate}
Thus if $\mu_{_C}\neq 0$, then $A$ and $C$ may differ only at positions in $R_{jk}$ (marked green in illustration \ref{leftmul2}).\hfill$\square$
\end{Cor} 


For $(i,j), (k,l) \in \pUP$ we define $(i,j)\preceq (k,l)$ if either $i>k$ or $i=k$ and $j\leqslant l$. Then $\preceq$ is a total ordering on $\pUP$.
For $[A]\in \hat V$ not main separated define $M(A)=M(\cO_A)$ to be the {\bf maximal} position in $\Limb^\circ(A)\cap \main(A)$.

\begin{Lemma} 
Let $[A]\in \hat V$ be not main separated. Then there exists $u\in U$ such that 
$\lambda_u [A]=\sum_{C\in V} \mu_{_C}[C]$
satisfies: If $\mu_{_C}\neq 0$ then $[C]$ is main separated or $M(C)\prec M(A)$.
\end{Lemma}
\begin{proof}
If $[A]$ is not staircase, we may use \ref{isotostaircase} and \ref{LeftTildeUandUAction} to make $[A]$ staircase by some restricted row operations. In particular this removes among possibly others all main conditions sitting on hook legs. From \ref{maximalposition} below one sees immediately, that this preserves $M(A)$ or makes it smaller. 

So from now on we may assume that $[A]$ is a staircase character but not main separated.
 Then $M(A)$ sits on an arm of some right main condition. More precisely,
Let $(\bar k, \bar j)=M(A)$ with $(i,k)\in \rmc(A)$. Then $\bar k<j<k$ and $A_{ik}\neq 0$
 Let $u=x_{ij}(-\lambda A_{ik}^{-1})$ and $\lambda_u[A]=\sum_{C\in V} \mu_{_C}[C]$:  
 \begin{equation}\label{maximalposition} 
\includegraphics[width=0.7\textwidth]{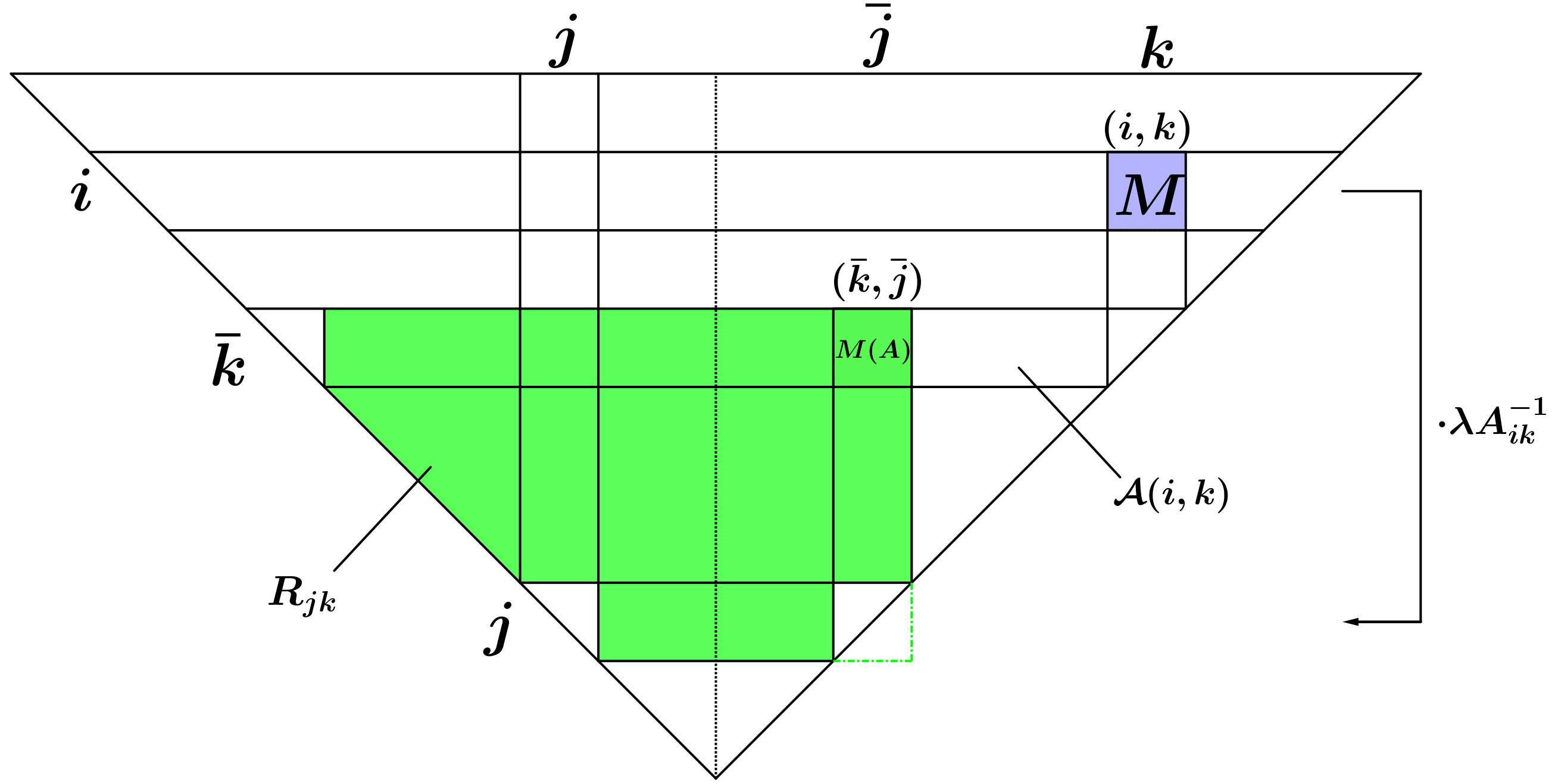}
\end{equation}
 
 Then this meets the situation \ref{situation}  and hence by \ref{clearconnectedmain} we have:
$\mu_C\neq 0$ implies:
\begin{enumerate}
\item[i)] $A$ und $C$ coincide at all positions $(a,b)\in J=\pUP\setminus R_{jk}$.
\item[ii)] $C_{\bar k \bar j} = A_{\bar k \bar j}\pm\lambda$.
\end{enumerate}
Hence choosing $\lambda=\mp A_{\bar k \bar j}$ we get $C_{\bar k \bar j}=0$. Therefore 
\[
M(C)\in \main(C)\subseteq (\main(A)\cup R_{jk})\setminus M(A).
\]
Since the positions in $R_{jk}\setminus M(A)$ are all to the left  of or lower of position $M(A)=(\bar k, \bar j)$, we derive that $M(C)$ is either main separated or $M(C)\prec M(A)$.
\end{proof}

\begin{Cor}
Let $A\in V$. Then $\C \cO_A$ is isomorphic to a submodule of a direct sum of orbit modules $\C \cO_C$ with $[C]\in \hat V$ either main separated or $M(C)\prec M(A)$.
\end{Cor}
\begin{proof}
Choose $u\in U$ such that $\lambda_u$ satisfies previous lemma. Then
$\lambda_u:\C \cO_A\longrightarrow \bigoplus_{C\in I}\C \cO_C$ is injective $\C U$-homomorphism, where $\lambda_u[A]=\sum_{C\in I} \mu_{_C}[C]$
with $I=\{C\in V\,|\, \mu_{_C}\neq 0\}$ and $[C]$ main separated or $M(C)\prec M(A)$ by previous corollary.
\end{proof}

We now can formulate and prove our main result of this section and our first main result of the paper:
\begin{Theorem}\label{Main1}
Every irreducible $\C U$-module is constituent of some orbit module $\C \cO_A$ with $[A]\in \hat V$ main separated.
\end{Theorem}
\begin{proof}
This follows immediately by induction on $M(A)$.
\end{proof}

\section{Core stabilizers}\label{sectionstab}
 In this section we shall determine the stabilizer $\Stab_{U}[A]=\{u\in U\,|\, [A].u=[A]\}$ for  staircase cores $[A]\in \hat V$. In [\cite{GJD}, 8.4] we did this already in the special case that  $[A]$ is a verge and showed in addition  [\cite{GJD}, 8.5]  that 
 \begin{equation}\label{stabsize}
 |\Stab_{U}[A] |=|\Stab_{U}[\verge(A)]|
 \end {equation}
 for all $A\in V$. More precisely, for a verge character $[A]\in \hat V$ we showed that 
 $\Stab_{U}[A]=U_{J(A)}$ where $J(A)=\pUP\setminus\Limb(A)$ is defined in \ref{lowerhook}.
 Now $U_{J(A)}$ is a pattern subgroup of $U$ and hence it is in particular the product of its intersections with the row groups consisting of matrices in $U$ whose only non-zero entries in $\pUP$ are on a particular fixed row. Recall that $\tilde n=n$ for types $\fC_n, \fD_n$ and $\tilde n = n+1$ for type $\fB_n$. It is easy to see that for $1\leq i<\tilde n$, the set 
 $\{(i,i+1), \ldots, (i, \widetilde i)\}\subseteq \Phi^+$ is closed, where $\widetilde i=\bar i-1$ for types $\fB_n, \fD_n$ and  $\widetilde i=\bar i$ for type $\fC_n$. Indeed
 \begin{equation}
 \cR_i=\langle  X_{ij}\,|\, i<j\leq \widetilde i\,\rangle
 \end{equation}
  is a pattern  subgroup of $U$, which is abelian provided $U$ is of type $\fB_n$ or $\fD_n$. For type $\fC_n$ inspecting definitions \ref{rootsubgroupsU} and \ref{Upattern} we see, that the root subgroups $X_{ij}$ and $X_{i\bar j}$ have the non trivial commutator subgroup $X_{i\bar i}$. However, every subgroup of $\cR_i$ generated by root subgroups including $X_{i\bar i}$ is a pattern subgroup in type $\fC_n$.

For a verge $[A]\in \hat V$ we have
  \begin{equation}\label{rowproduct}
 \Stab_{U}[A]=\prod\nolimits_{i=1}^n \big(\Stab_{U}[A]\cap  \cR_i\big)
 \end {equation}

 For staircase cores $[A]\in \hat V$, the stabilizer $\Stab_{U}[A]$ is not  a pattern subgroup of $U$ in general. However \ref{rowproduct} is still true. Our strategy to prove this consists of describing $\Stab_{U}[A]\cap \cR_i$ explicitly (as solution set of a homogeneous system of linear equations) and proving a row wise version of \ref{stabsize}, that is
\begin{equation}\label{stabsizerow}
 |\Stab_{U}[A]\cap \cR_i |=|\Stab_{U}[\verge(A)]\cap \cR_i|.
 \end {equation}
 As a consequence we conclude that 
 \begin{equation}\label{stabsizerowpro}
\left |\prod_{i=1}^n\big(\Stab_{U}[A]\cap \cR_i\big)\right |=\left|\prod_{i=1}^n\big(U_{J(A)}\cap  \cR_i\big)\right |=\left|\Stab_{U}[\verge(A)]\right|= |\Stab_{U}[A]|
 \end {equation}
 proving  \ref{rowproduct} for the staircase core $[A]\in \hat V$.
 
 For the moment let $[A]\in \hat V$ be arbitrary staircase character. We investigate $\Stab_{U}[A]\cap  \cR_i, 1\leq i <\tilde n$. Suppose there is  $1<k<i$ such that $A_{k\bar i}\neq 0$. Let $\alpha\in \F_q$, then by \ref{4action} we have $[A].x_{il}(\alpha)=[B]$ for $i<l\leq  \widetilde i$, where $B$ coincides with $A$ in all columns except column $i$ and $\bar l$. Indeed $B_{k \bar l}=A_{k \bar l}\pm \alpha A_{k \bar i}\not=A_{k \bar l}$ if  $\alpha\neq0$.   From this we conclude immediately that $\Stab_{U}[A]\cap  \cR_i=(1)$.


Assume that $U$ is of type $\fC_n$ and that $A_{k\bar i}=0$ for $1\leq k\leq  i-1$ and $A_{i \bar i}\neq 0$. Inspecting \ref{illTruncatedColumnOperationG} we see that $X_{i \bar i}$ acts on $[A]$ as linear character and hence is contained in $\Stab_{U}[A]$. Moreover if $i<l<\bar i, \alpha\in \F_q$ and $[B]=[A].x_{il}(\alpha)$, then $B$ differs from $A$ only on column $i$ and $\bar l$, again by \ref{illTruncatedColumnOperationG}. In particular $B_{i \bar l}=A_{i \bar l}\pm \alpha A_{i \bar i}\not=A_{i \bar l}$, provided $\alpha\neq 0$. Thus we conclude $\Stab_{U}[A]\cap  \cR_i=X_{i \bar i}$. We have shown:
 \begin{Lemma}\label{typeCstab}
 Let $[A]\in \hat V$ and $1\leq i<\tilde n$. Suppose column $\bar i$ of $A$ is not a zero column. Then $\Stab_{U}[A]\cap \cR_i=(1)$, if $A_{k \bar i}\neq 0$ for some 
$1\leq k\leq  i-1$. If $U$ is of type $\fC_n$ and  $A_{i \bar i}\neq 0$ is the only non-zero entry in column $\bar i$ of $A$, then  $\Stab_{U}[A]\cap \cR_i=X_{i \bar i}$. \hfill$\square$
 \end{Lemma}

\begin{Defn}\label{AntDiagStab} 
 For $1\leq i <\tilde n$ define $\Stab^i_U[A]=\Stab_U[A]\cap \cR_i.$\hfill $\square$

\end{Defn}

 Fix $1\leq i<\tilde n$ such that column $\bar i$ of $A$ is a zero column. Let $i<l\leq\widetilde i$ and $\alpha\in \F_q$. Observing \ref{rootsubgroupsU} and \ref{illTruncatedColumnOperationG}  we get immediately (using notation \ref{typeArootsubgroups}):
 
 \begin{Lemma}\label{sameasuni}
 $[A].x_{il}(\alpha)=[A].\tilde x_{il}(\alpha)$. Hence in particular $[A].x_{il}(\alpha)=[B]\in \hat V$ differs from $A$ only in column $i$ which is obtained from $A$ by subtracting $\alpha$ times column $l$ from column $i$ and projecting the resulting matrix into $V.$\hfill$\square$
 \end{Lemma}
 
 We define the submatrix $A_i$ of $A$ to be the rectangular $(i-1)\times (\widetilde i-i)$-matrix 
\begin{equation}
A_i=\sum_{1\leq \nu<i<\mu\leq\widetilde i} A_{\nu \mu}e_{\nu \mu}.
\end{equation}
That is $A_i$ is obtained by cutting off rows $i,i+1,\ldots, \tilde n, \tilde n+1, \ldots, N$
 and columns $1, 2,\ldots,i, \widetilde i+1, \widetilde i+2, \ldots, N$ from $A$:
\begin{center}
\includegraphics[width=0.75\textwidth]{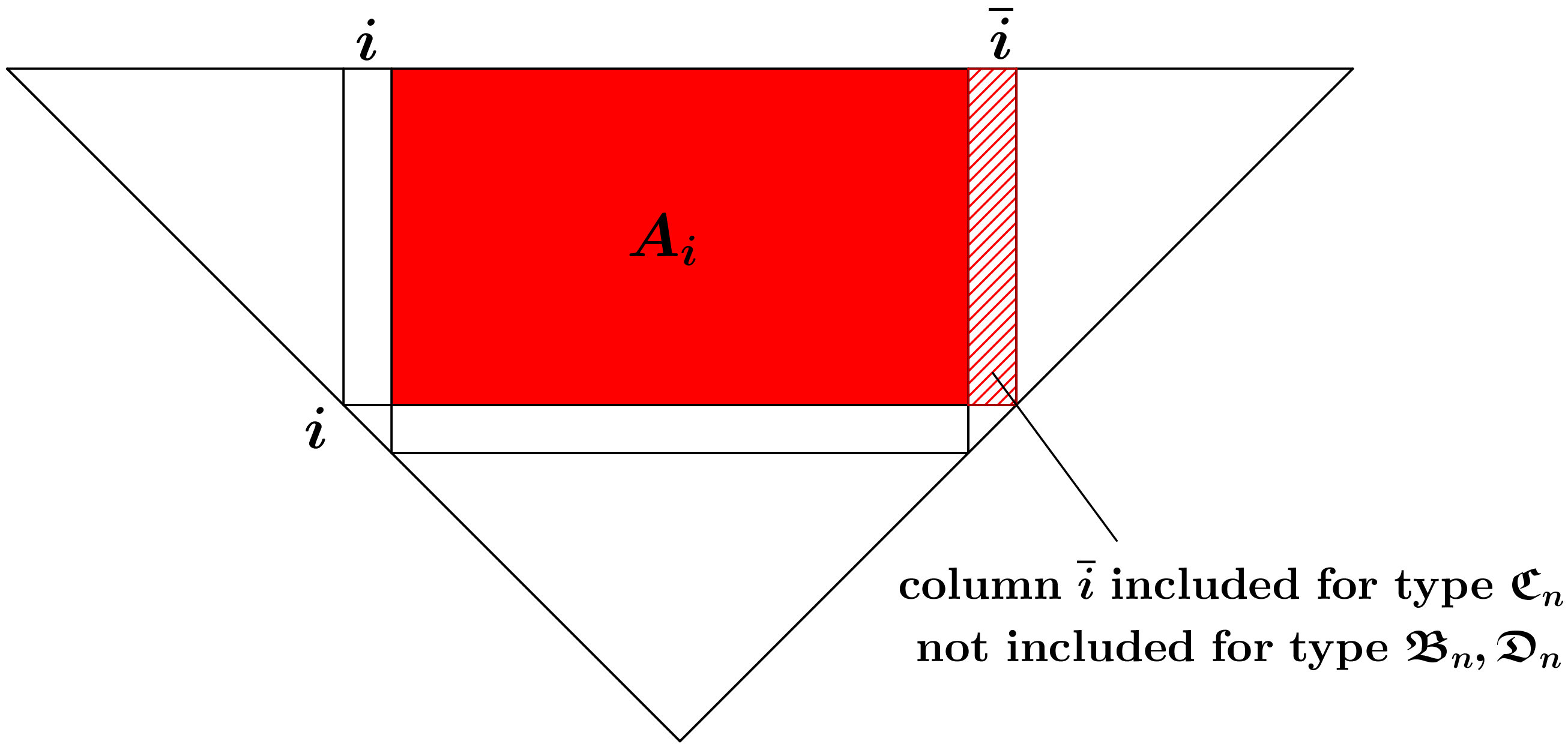}
\end{center}
\vspace{-0.4cm}
 \begin{equation}
\text{\footnotesize{Illustration of the submatrix $A_i$}.}\quad\quad\quad
\end{equation}

Let $\mathfrak a_{i+1}, \ldots, \mathfrak a_{\,\widetilde i}$ be the columns of $A_i$, and $\alpha_{i+1}, \ldots, \alpha_{\,\widetilde i}\in \F_q$. Define $\underline{\alpha}\in \F_q^{\widetilde i-i}$ to be the column vector $\underline{\alpha}=(\alpha_{i+1}, \ldots, \alpha_{\,\widetilde i})^t$. Then by basic linear algebra the column vector 
$\mathfrak c_i:= A_i \underline\alpha \in \F_q^{i-1}$ is obtained as linear combination of the columns of $A_i$ with coefficients 
$\alpha_j, i+1 \leq j\leq \widetilde i$, that is    
\begin{equation}\label{LinAlg}
\mathfrak c_i= A_i \underline\alpha=\alpha_{i+1}\mathfrak a_{i+1}+\cdots +\alpha_{\,\widetilde i}\mathfrak a_{\,\widetilde i}
\end{equation}
If $\mathfrak x_i=(\gamma_{1}, \ldots, \gamma_{i-1})^t\in \F_q^{i-1}$ set $\hat {\mathfrak x}_i=(\gamma_{1}, \ldots, \gamma_{i-1}, 0, \ldots,0)^t\in \F_q^N$. Moreover, for $\underline{\alpha}\in \F_q^{\widetilde i-i}$ as above we set 
$x(\underline\alpha)=x_{i, i+1}(\alpha_{_{i+1}})\cdots x_{i, \widetilde i}(\alpha_{\,_{\widetilde i}})\in \cR_i$, and $\tilde x(\underline\alpha)=\tilde x_{i, i+1}(\alpha_{_{i+1}})\cdots  \tilde  x_{i, \widetilde i}(\alpha_{\,_{\widetilde i}})$. Then  \ref{illTruncatedColumnOperationG} and \ref{sameasuni} imply:
\begin{Lemma}\label{columnvector}
$[A].x(\underline\alpha)=[A].\tilde x(\underline\alpha)=[B],$ where $B$ coincides with $A$ everywhere besides in column $i$. This is  given as $\hat {\mathfrak b}_i=\hat{\mathfrak a}_i-\hat{\mathfrak c}_i
$
where ${\mathfrak a}_i = (A_{1i},\ldots , A_{i-1,i})^t$. Thus $\hat{\mathfrak a}_i$ is the $i$-th column of $A$, and $\mathfrak c_i$ is defined as in \ref{LinAlg}.
\end{Lemma}
\begin{proof}
Lemma \ref{sameasuni} and \ref{LinAlg}  imply immediately, that $B$  coincides with $A$ except in column $i$, whose first $i-1$ entries  (the only non trivial ones) are given as 
\begin{equation}
{\mathfrak b}_i =
\begin{pmatrix}
B_{1 i}\\B_{2 i}\\ \vdots \\B_{i-1,i}
\end{pmatrix}=
\begin{pmatrix}
A_{1 i}\\A_{2 i}\\ \vdots \\A_{i-1,i}
\end{pmatrix}-
A_i
\begin{pmatrix}
\alpha_{_{i+1}}\\\alpha_{ _{i+2}}\\ \vdots \\ \alpha_{\,_{\widetilde i}}
\end{pmatrix}.
\end{equation}
\end{proof}

We wanted to exhibit elements of $\Stab_{U}[A]$. These arise in row  group $\cR_i$ by requiring in lemma \ref{columnvector}   that $B=A$. Thus we get immediately:
\begin{Cor}\label{stabRi}
Let $[A]\in \hat V$ be staircase and $1\leq i <\tilde n$ such that column $\bar i$ of $A$ is zero. Then 
\[
 \Stab^i_U[A] = \Stab_{U}[A]\cap \cR_i=\{x(\underline{\alpha})\,|\,A_i \underline{\alpha}=0, \, \underline{\alpha}\in\F_q^{\widetilde i-i}\}.
\]
\end{Cor}

\begin{Remark}
We remark in passing that in type $\fC_n$ column $\bar i$ of $A_i$ is a zero column by our assumption on $i$. As a consequence we obtain that $A_i\underline{\alpha}=0$ for $\underline{\alpha}=(0,0, \ldots,0,\alpha)^t, \alpha\in \F_q$, and hence $x(\underline{\alpha})=x_{i,\bar i}(\alpha)$ is contained in $\Stab_U^i[A]$, as already stated in \ref{typeCstab}. \hfill $\square$
\end{Remark}

 Let
$\{(i_1, j_1), (i_2, j_2), \ldots, (i_r, j_r)\}$
be the set of main conditions contained in the set of positions of $A_i$, ordered from left to right. Thus $1\leq i_\nu<i, (\nu=1, \ldots, r)$ and $i<j_1<j_2<\cdots<j_r<\bar i$.
Since all entries in rows $i_\nu$ ($\nu=1, \ldots, r$) to the right of the main conditions $(i_\nu, j_\nu)$ are zero, the rank of the matrix $A_i$ is at least $r$ and hence the dimension of the solution space of $A_i \underline\alpha=0$ is at most $\widetilde i-i-r$.

We now turn to the special case, where $[A]\in \hat V$ is a staircase core character. Note that then the condition of column $\bar i$ being a zero column in $A$ is equivalent to the assumption, that there is no right main condition on column $\bar i$ or equivalently, that row $i$ is not an arm $\cA(k, \bar i)$ for some main condition $(k, \bar i)\in \rmc(A)$.
Recall that all entries between right main and minor conditions are zero. Thus for $1\leq i<\tilde n$ such that $(k,\bar i)\notin \rmc(A)$ for all $1\leq k<i$, we have the following situation :
\begin{equation}\label{upcore}
\includegraphics[width=0.7\textwidth]{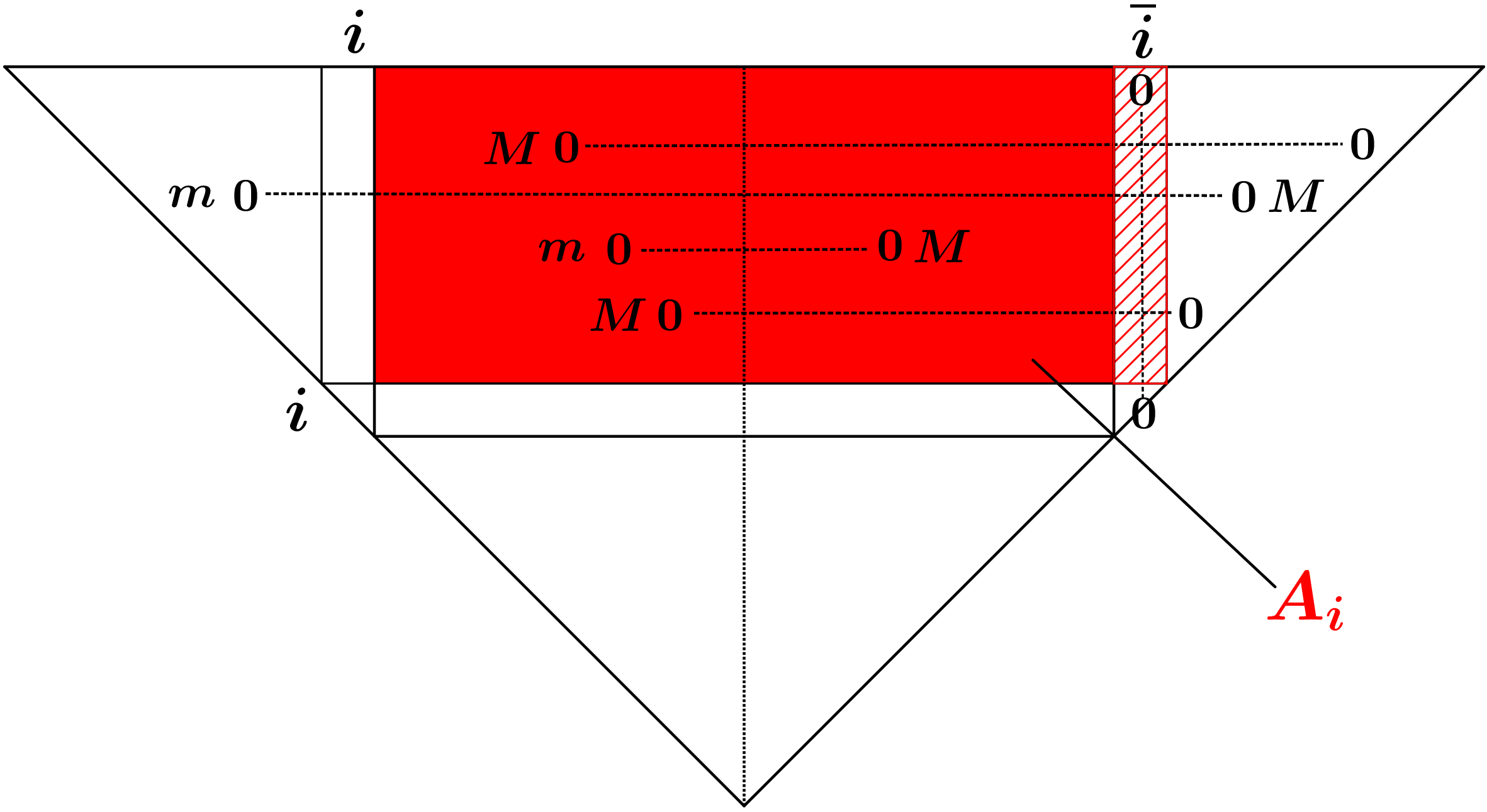}
\end{equation}

Obviously all rows of $A_i$ not containing a main condition are zero rows. Since $[A]\in \hat V$ is staircase, the main conditions in $A$ are in pairwise different columns and hence the rank of $A_i$ is precisely $r$ where again $r$ is the number of main conditions in $A_i$. Consequently the dimension of the solution space of the homogeneous matrix equation $A_i\underline\alpha=0$ in \ref{upcore} is $\widetilde i-i-r$. Note that this depends only on $\verge(A)$ not on the values of matrix entries of $A$ on minor or supplementary conditions. This implies
\begin{equation}\label{stabiajai}
\left |\Stab^i_{U}[A]\right |=\left|\Stab^i_{U}[\verge(A)]\right|=\left|U_{J(A)}\cap \cR_i\right |=q^{\widetilde i-i-r}.
\end{equation}

\begin{Theorem}\label{describestabRi}
Let $[A]\in \hat V$ be a staircase core. For $1\leq i<\tilde n$, let $$A_i=\sum_{\stackrel{1\leq a<i}{i<b\leq\widetilde i}} A_{ab} e_{ab}$$
and let $\fS_i$ be the solution space in $\F_q^{\widetilde i-i}$ of systems $A_i \underline\alpha=0$, the vector $\underline \alpha$ in $\F_q^{\widetilde i-i}$ labelled as $\underline\alpha=(\alpha_{i+1}, \ldots, \alpha_{\,\widetilde i})^t$.
Setting $x(\underline\alpha)=x_{i, i+1}(\alpha_{_{i+1}})\cdots x_{i, \widetilde i}(\alpha_{\,_{\widetilde i}})\in \cR_i$, we obtain:
\[
\Stab^i_U[A]=\Stab_U[A]\cap \cR_i=\begin{cases}
(1),& \text{ for row } i\subseteq \Limb(A) \text{ and $U$ is of type $\fB_n$ or $\fD_n;$}\\
X_{i\,\bar i},&\text{ $(i,\bar i)\in \main(A)$ and $U$ is of type $\fC_n;$}\\
\{x(\underline\alpha)\,|\, \underline\alpha\in \fS_i\}, & \text{ otherwise}.
\end{cases}
\]
Moreover every element $x\in \Stab_U[A]$ can be uniquely written as product $\prod_{i=1}^{\tilde n-1} y_i$ with $y_i\in \Stab_{U}^i[A]$.
\end{Theorem}
\begin{proof}
Since $U$ is the product  $\prod_{(i,j)\in \pUPs}X_{ij}(\alpha), $ we see that
\[
\Stab_U[A]\supseteq \prod_{1\leq i <\tilde n} \Stab^i_U[A]=\{u_1 \cdots u_{\tilde n-1}\,|\, u_i\in  \Stab^i_U[A], 1\leq i<\tilde n\} ,
\]
which contains $ \prod_{1\leq i <\tilde n}\left| \Stab^i_U[A]\right| =  \prod_{1\leq i <\tilde n}\left|U_{J(A)}\cap \cR_i\right|=|U_{J(A)}|$ many elements by \ref{stabiajai} and \ref{typeCstab}. Since  $|\Stab_{U}[A]|=|U_{J(A)}|$ by [\cite{GJD}, 8.5], we obtain the equality.

The uniqueness claim follows, since every element of $U$ can be written uniquely as product of root subgroup elements $x_{ij}(\alpha), (i,j)\in \pUP, \alpha\in \F_q$, where the product can be taken in an arbitrary fixed ordering of the position  $(i,j)\in \pUP$. From this the theorem follows by the discussion above.
\end{proof}

By a special choice of a basis of the solution space of $A_i \underline \alpha=0$ we can even do better. Indeed, the basis of the solution space given below is obtained by Gaussian elimination. 
We produce a $(\,\widetilde i-i)\times (\,\widetilde i-i)$-matrix $\tilde A_i$ as follows:

First we reorder the columns of $A_i$ such that we have columns $j_1, j_2, \ldots, j_r$ first (in that order) and then columns $s_1, s_2, \ldots, s_k\in S_i$, $(k=\widetilde i-i-r)$ where $S_i=\{l\,|\, i<l\leq \widetilde  i, l\notin \{j_1, \ldots, j_r\}\}=\{s_1<s_2<\cdots<s_k\}$. 
Note that in type $\fC_n$ we have automatically $s_k=\bar i$ by our assumption and column $s_k$ of $A_i$ is a zero column.
 Then we remove all zero rows, and take the remaining rows labelled by the row indices of main conditions in order $i_1, i_2, \ldots, i_r$. Thus since $j_1< j_2<\ldots< j_r$, the main conditions are on the diagonal of the upper left hand sided  $r\times r$-submatrix, which is lower triangular with main conditions on the diagonal. Then we use row operations to bring this $r\times r$-submatrix into diagonal form. More precisely, this is done by adding  multiples of rows to the lower ones from top down. Thus all the zero entries on the positions $(i_\nu, t)$ with $t>j_\nu$ ($\nu=1, \ldots, r$) are preserved as well as  the entries on the main conditions sitting on the diagonal of the upper left hand sided $r\times r$-submatrix, (keeping in mind  that the natural reordering of  $(j_1,j_2,\ldots,j_r,s_1,s_2,\ldots,s_k)$ intertwines $j_1< j_2<\ldots< j_r$ and  $s_1<s_2<\cdots<s_k$).

 Next we divide rows $i_1,\ldots, i_r$ by the diagonal values making this $r\times r$-submatrix the identity matrix $E_{r}$. We change the indexing of the rows of the resulting matrix to $j_1, j_2,\ldots, j_r$ and add $k$ many zero rows at the bottom with row indices $s_1, \ldots,s_k$ (in that order). Finally we replace the  south east $k\times k$ zero matrix by $-E_k$, where $E_k$ is the $k\times k$ identity matrix.  Thus we have:
\begin{equation}\label{tildeAi}
\includegraphics[width=0.4\textwidth]{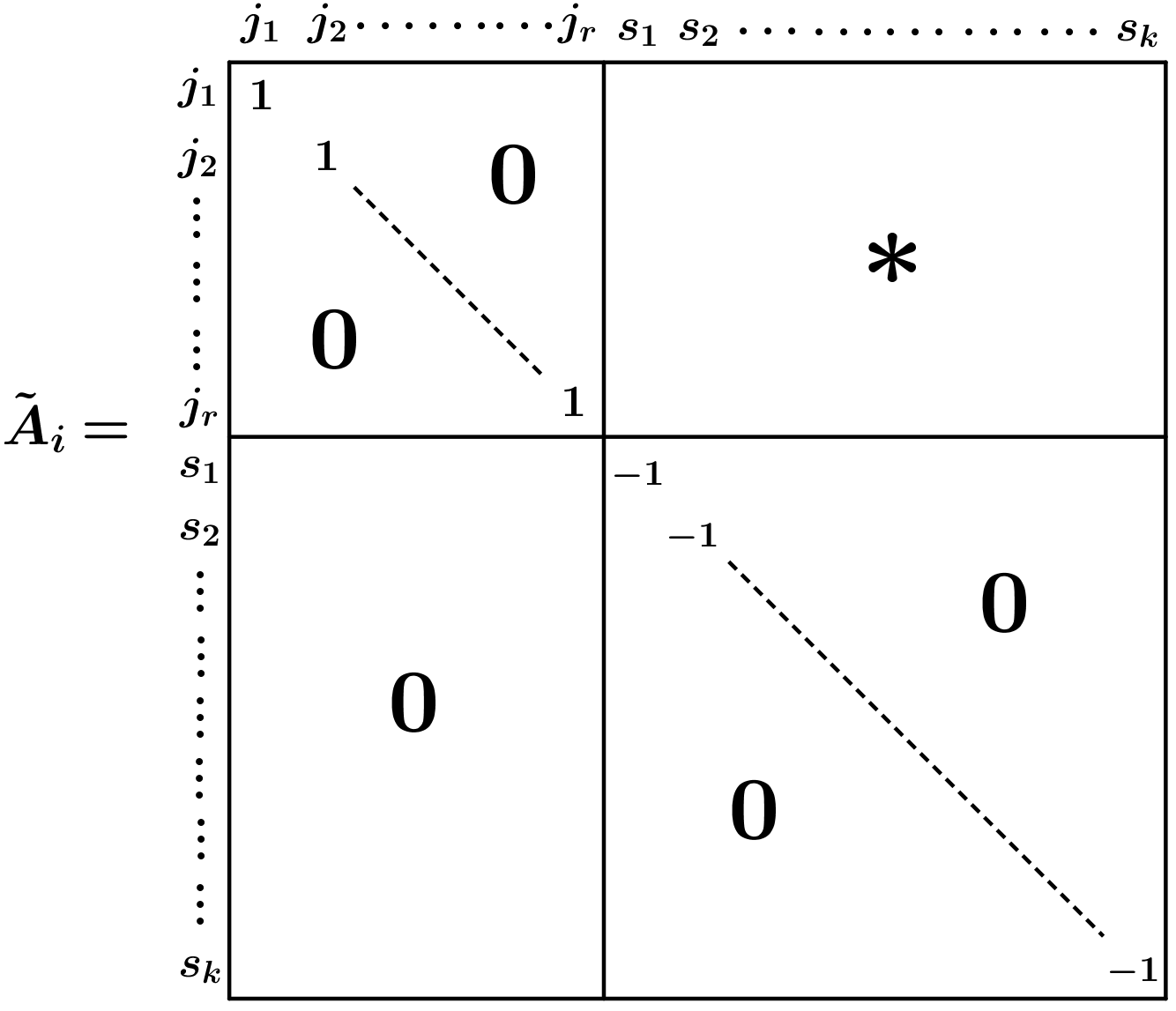}
\end{equation}
Let $\underline\alpha_{s_{_1}}, \underline\alpha_{s_{_2}}, \ldots, \underline\alpha_{s_{_k}}$ be the column vectors of $\tilde A_i$ of columns $s_{_1}, s_{_2},\ldots, s_{_k}=\{i+1, \ldots, \widetilde i\,\}\setminus\{j_1, \ldots, j_r\}$. We reorder the components of these column vectors back into the natural ordering of their labels (reintertwining $j_1< j_2<\ldots< j_r$ and  $s_1<s_2<\cdots<s_k$) and denote the resulting vectors again by $\underline\alpha_{s_\nu}$, ($\nu=1,\ldots, k$). Gaussian elimination tells us and it is easy to check, that $\{\underline\alpha_{s_{_1}}, \ldots, \underline\alpha_{s_{_k}}\}$ is a basis of the solution space of $A_i \underline \alpha=0$.


Recall that to the right of main conditions in $A_i$ all entries of $A_i$ are zero. For any $1\leq \nu \leq r$, we find $1\leq \mu<k$ such that $s_\mu<j_\nu<s_{\mu+1}$ and then all entries of row $j_\nu$ of $\tilde A_i$ to the right of column $s_\mu$ are zero.

For $s\in S_i$, suppose
$\underline \alpha_s=(\alpha^s_{i+1}, \ldots, \alpha^s_{\,\widetilde i})^t\in \F_q^{\widetilde i-i}$. Then by our construction of $\underline\alpha_s$ we have $\alpha^s_s=-1$ and $\alpha^s_t=0$ for $s\neq t\in S_i$. In addition $\alpha^s_\mu=0$ if $\mu<s$ and $\mu\in \{j_1, \ldots, j_r\}$. Thus for $s\in S_i$, 
\begin{equation}\label{basisstab}
x(\underline\alpha_{s})=\prod_{\mu=i+1}^{\widetilde i} x_{i\mu} (\alpha^s_\mu)=x_{is}(-1)y\quad\text{ with }\quad
y:=\prod_{\rho=\nu}^{r}x_{ij_\rho}(\alpha^s_{j_\rho}) 
\end{equation}
where $\nu$ satisfies $j_{\nu-1}<s<j_{\nu}$.
Consequently, 
$[A].x(\underline\alpha_{s})=[A].x_{is}(-1)y$, where $[A].x_{is}(-1)$ adds column $s$ to column $i$ of $A$ and projects into $V_{\pUPs}$. Then $y$ readjusts the entries in column $i$ of $A$ back to the original column using the main conditions in $A_i$ further to the right of column $s$. To achieve this we have to add $-\alpha^s_{j_\rho}$ times column $j_\rho$ (containing main condition $(i_\rho, j_\rho)$) to column $i$ of $A$ for $\rho=\nu, \ldots, r$ to obtain on position $(i, i_\rho)$ of $A$ the original entry.

Let $[A]\in \hat V$ be a core and fix a row $i$, $1\leq i<\tilde n$, such that $\bar i$ is a zero column in $A$. We define
$U_{J(A)}^i=U_{J(A)}\cap \cR_i=U_{J(A)^i}$ with $J(A)^i=J(A)\cap\{(i,k)\,|\,i<k\leq \widetilde i\}$ and $\Limb^i(A)=\Limb(A)\cap\{(i,k)\,|\,i<k<\bar i\}$. Then obviously $\Limb^i(A)=\{(i, j_1), \ldots, (i, j_r)\}$ is the set of positions on row $i$ such that there is a main condition in $A_i$ above $(i, j_\nu)$, ($\nu=1, \ldots, r$). Likewise $J(A)^i=\{(i,k)\,|\,k\in S_i\}$ consists of all positions on row $i$ which do not have a main condition to the north of it.

Inspecting illustration \ref{tildeAi} we see that we can divide every linear combination of the column vectors $\underline \alpha_s, (s\in S_i)$ into two parts, namely those to indices $(j_1, \ldots, j_r)$ and $(s_1, \ldots, s_k)$. We obtain a row wise version of part of [\cite{GJD},7.20], stating that the permutation action by any element of $U_{J(A)^i}$ can be compensated by the action of some uniquely determined element of $U_{\Limb^i(A)} = \prod_{\mu=j_1, \ldots, j_r} X_{i\mu}$.

More precisely:
 For each $s \in S_i$ choose $\lambda_s\in \F_q$ and set  $\underline\alpha=\sum_{s\in S_i} \lambda_s \underline\alpha_s$, $x=x(\underline\alpha)$. 
Then 
$x=\left(\prod_{s\in S_i }x_{i s}(-\lambda_s)\right)y_{\underline\alpha}$ for a uniquely determined $y_{\underline \alpha}\in \prod_{(i,\mu)\in \Limb^i(A)}X_{i\mu}$. The coefficient $\lambda_\mu$ in $y_{\underline \alpha}=\prod_{(i,\mu)\in \Limb^i(A)}x_{i\mu}(\lambda_\mu)$ is given as coefficient
$\lambda_\mu=\sum_{s\in S_i}\lambda_s\alpha^s_{\mu}$
at position $\mu$ in $\underline \alpha_s=(\alpha^s_{i+1}, \alpha^s_{i+2}, \ldots, \alpha^s_{\,\widetilde i})^t$.

From this one obtains now the following description of stabilizer of cores:
\begin{Theorem}
Let $[A]\in \hat V$ be a core, $\Stab^i[A]=\Stab_U[A]\cap \cR_i$ for $1\leq i<\tilde n$. Then every $u\in \Stab_U[A]$ can be written uniquely as 
$u=u_1\cdots u_{\tilde n-1}$ with $u_i\in \Stab^i_U[A]$.
 Moreover, if column $\bar i$ is a zero column in $A$ (or equivalently 
$\{\text{row } i \}\cap \Limb^\circ(A)=\emptyset$), then $\Stab^i_U[A]$ depends only on $A_i$, not on entries on or below row $i$ in $A$. Moreover for types $\fB_n, \fD_n$ the stabilizer $\Stab^i_U[A]$ in row $\cR_i$ is of the form
\[
\Stab^i_U[A]=\prod_{s\in S_i}X_{is}\times H_i, \quad\text{ (direct product)}
\]
where $H_i$ is a subgroup of $\prod_{\mu=j_1, \ldots, j_r} X_{i\mu}$
such that for every $u\in \prod_{s\in S_i}X_{is} $ we have precisely one element $y_u\in H_i$ with $uy_u\in \Stab^i_U[A]$.
If $U$ is of type $\fC_n$, we denote for any set $T$ of $\cR_i$ its image in the abelian group $\overline{\cR_i}=\cR_i/X_{i\,\bar i}$ by $\overline{T}$. Then 
\[
\overline{\Stab^i_U[A]}=\prod_{\bar i\not=s\in S_i}\overline{X_{is}}\times \overline{H_i}, \quad\text{ (direct product)}
\]
where $\overline{H_i}$ is a subgroup of $\prod_{\mu=j_1, \ldots, j_r} \overline{X_{i\mu}}$ (direct product),
such that for every $u\in \prod_{s\in S_i}X_{is} $ we have modulo $X_{i\,\bar i}$ exactly one element $\overline{y_u}\in \overline{H_i}$ with $\overline{u}\,\overline{y_u}\in \overline{\Stab^i_U[A]}$.
\end{Theorem}
\begin{proof}
Assume column $\bar i$ is a zero column in $A$. Then the only claim not immediately obvious by the discussion above is that $H_i$ is a group if $U$ is of types $\fB_n, \fD_n$. This however follows, since  it is contained in the abelian row group $\cR_i$. For type $\fC_n$ we observe that $X_{i\,\bar i}$ acts on $[A]$ by the trivial character.
\end{proof}

There are several nice consequences of this theorem, one being for example that if a core $[A]\in \hat V$ and a staircase $[B]\in \hat V$ coincide at all positions to the north of row $i$, their stabilizers in rows $j\leq i$ coincide as well.

For $[A]\in \hat V$ the stabilizer $\Stab_U[A]=\{u\in U\,|\, [A].u=[A]\}$
acts on $[A]$ by a linear character which we will henceforth denote by $\Psi_A$. Thus  
$[A]u=\Psi_A(u)[A].u=\Psi_A(u)[A]$. Note that $\Psi_A$ is a group homomorphism from $\Stab_U[A]$ into $\C^*$. On the other hand, by \ref{monomial} we have $[A]u=\theta\kappa(-A, \pi_{\pUPt}(u^{-1}))[A].u=\prod_{(a,b)\in \pUPt}\theta(-A_{ab}g_{ab})[A]$ where $g=u^{-1}$, and hence
$
\Psi_A(u)=\prod_{(a,b)\in \pUPt}\theta(-A_{ab}g_{ab})\in \C^*.
$
So let $1\leq i<\tilde n$ and again assume, that column $\bar i$ in $A$ is a zero column. Let $y_i=\prod\limits_{i<k\leq\,\widetilde i}x_{ik}(\alpha_k)\in \Stab_U[A]$, $\alpha_k\in \F_q$, then
\begin{equation}\label{stabcharRi}
\Psi_A(y_i)=\prod_{i<k\leq\,\widetilde i}\theta(A_{ik}\alpha_{k}).
\end{equation}
\begin{Lemma}\label{stabcharbasis}
With the notation of the discussion above for a core $[A]\in \hat V$ with $\bar i$ being a zero column. Let $s\in S_i$. Then there exists $2\leq \nu\leq r$ such that $j_{\nu-1}<s<j_\nu$. Define $x(\underline\alpha_s)$ as in \ref{basisstab}. Then $x(\underline\alpha_s)\in \Stab_U[A]^i\subseteq \Stab_U[A]$ and we have:
\begin{equation*}
\Psi_A(x(\underline\alpha_s))=\theta(-A_{is})\prod_{\rho=\nu}^r \theta(A_{ij_\rho}\alpha^s_{j_\rho}).
\end{equation*}
\end{Lemma}
\begin{proof}It follows directly from the definition of $x(\underline\alpha_s)$ and \ref{stabcharRi}.
\end{proof}

\begin{Remark}
\begin{enumerate}
\item [1)] The stabilizer $\Stab_{\widetilde U}[A]$ of $[A]\in \hat V$ in $\tilde U$ can be similarly treated as  $\Stab_U[A]$. The main difference is that now elements in row groups $\cR_i$ where column $i$ contains a main condition may belong to the stabilizer even if $[A]$ is a core.
\item [2)]  
If $[A]\in \hat V$ is not a core it may happen, that the rank of the coefficient matrix $A_i$ ($i$ a row index with column $\bar i$ a zero column)
 is bigger than the number of main conditions in the set of positions in $A_i$, see illustration:
 \begin{equation}\label{upnotcore}
\includegraphics[width=0.8\textwidth]{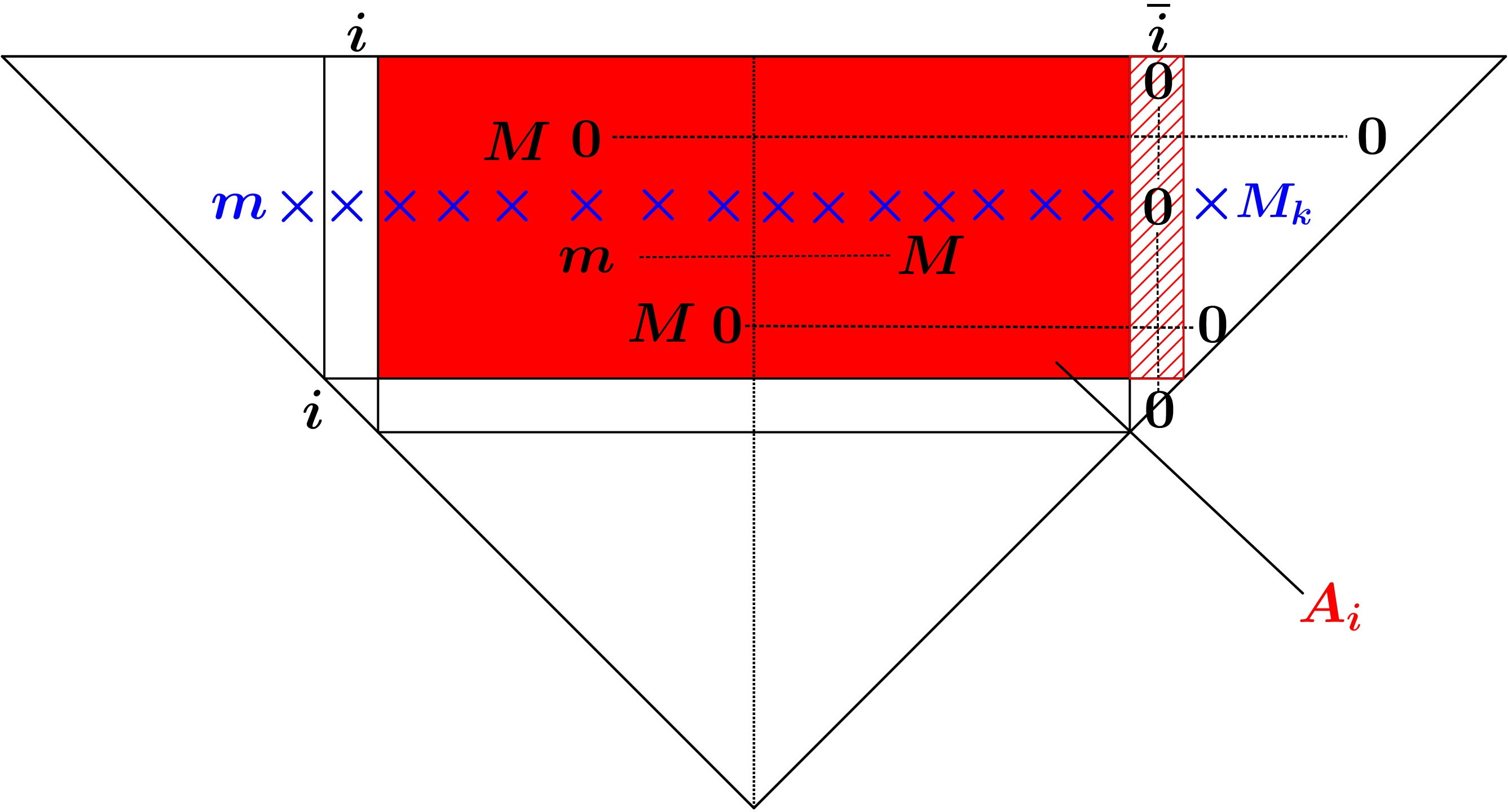}
\end{equation}
Thus the stabilizer elements in rows $\cR_i$, $i$-th column a zero column, will not produce the full stabilizer group. Indeed in the illustration above we may need the main conditions $M_k$ to the right of column $\bar i$ to adjust elements in positions $(k,i)$, where $k$ is the row index of $M_k$. But this is done using the root subgroup $X_{\bar k i}$ on the arm $\cA(i,k).$  \hfill $\square$
\end{enumerate}
\end{Remark}

\section{Verge modules}\label{secverge}
Recall that $\widetilde U=U_N(q)$. Let $\widetilde V=\{u-1\,|\, u\in \widetilde U\}=V_{\URs}=\Lie(\widetilde U)$. Then
$\widetilde U$ acts on $\widetilde V$ by left and right multiplication and hence on the set of linear complex characters $\hat { \widetilde V}$ of the additive group $(\widetilde V, +)$. Then $f: \widetilde U\rightarrow \widetilde V: u\mapsto u-1$ is a left and right 1-cocycle, providing both a left and right monomial action of $\widetilde U$ on $\C \hat {\widetilde V}$ with monomial basis $\hat {\widetilde V}=\{\widetilde {[A]}\,|\, A\in \widetilde V\}$. With this action from both sides $\C \hat {\widetilde V}$ becomes a $\C\widetilde U$-$\C\widetilde U$-bimodule which is isomorphic to the regular bi-representation $_{\C \widetilde U}\C \widetilde U_{\C \widetilde U}.$    It is known that $\hat {\widetilde V}$ decomposes into  monomial biorbits, (see [\cite{andre}, \cite{yan}]).  In [\cite{mindim}] each  biorbit contains precisely one {\bf verge}. Moreover each biorbit is a union of right orbits, all right orbits in a biorbit induce isomorphic right modules, and any two right orbits, being contained in different biorbits, afford orthogonal characters. The different (and hence orthogonal) characters afforded by  the right orbits are the Andr\'e-Yan supercharacters of $\widetilde U$, (cf. [\cite{andre}, \cite{yan}]).
Note that  the main condition defined in [\cite{mindim}], on which the verges  depend, is slightly different from what we are using in this paper. If one defines main conditions for $A\in\widetilde V$ analogous to the definition for $U$-orbits in this paper to consist of the positions of non-zero entries farest to the right in each row in $A$, the main conditions remain invariant for characters on $\widetilde \cO_A$. However right orbits may then have different main conditions, even if they belong to the same $\widetilde U$-biorbit.  Of course, staircase characters $\widetilde{[A]} \in\hat {\widetilde V}$ are again defined by requiring, that all main conditions of  $\widetilde{[A]}$ lie in pairwise different columns of $A$. For  staircase characters both notions of main conditions coincide. Hence, if we restrict ourselves to the staircase characters (and we do this in this paper), then both definitions of main conditions  give the same sets of positions, and hence the same verges.

\begin{Defn}\label{vergemodules}
Let  $[A]\in \hat V$ be a  staircase verge character. We define 
\[
\cV(A)=\{[B]\in \hat V\,|\, \verge(B)=\verge(A)\}.
\]
\end{Defn}

\begin{Remark}
Let  $[A]\in \hat V$ be a  staircase verge character. Since $\supp(A)\subseteq \pUP\subseteq \UR$, the notion of $\widetilde{[A]}\in \hat {\widetilde V}$  is well defined. Moreover the right monomial $\widetilde U$-orbit $\tilde \cO_A$  consists of all $\widetilde{[B]}$ where $B$ is different from $A$ only on the positions in $\UR$ and to the left of main conditions, hence $\verge(B)=\verge(A)$. Since $\main(A)\subseteq \supp(A)\subseteq \pUP$ and $\minc(A)\cup \suppl(A)$ are all to the left of $\main(A)$, we have indeed shown: 
\[
\C \cV(A)=\C \tilde \cO_{A}=\bigoplus_{\stackrel{[B]\in\hat V\text{ is  core} }{\verge(B)=\verge(A)}}\C \cO_{B}.
\]
In particular we call $\C \cV(A)$ {\bf verge module} if $[A]\in \hat V$ is a  staircase verge character.   \hfill $\square$
\end{Remark}

\begin{Defn}\label{defofI}
Given a staircase $\widetilde{[A]}\in \hat {\widetilde V}$ we define
\begin{enumerate} 
\item [i)] $I^\circ_d(A):=\{(i,j)\in \UR\,|\,\exists\, (i,a), (j,b)\in \main(A), i<a< j<b\}$.
\item [ii)] $I_d(A):=\{(i,j)\in \UR\,|\,\exists\, (i,a), (j,b)\in \main(A), i<a\leq j<b\}$.
\item [iii)]$I(A):=\{(i,j)\in \UR\,|\,\exists\, (i,a), (j,b)\in \main(A), a<b\}$.
\end{enumerate}
Note that $I^\circ_d(A)\subseteq I_d(A) \subseteq I(A)$ and that these subsets depend only on $\main(A)$.  \hfill $\square$
\end{Defn}

\begin{Defn}\label{hookintersection}
Let $(i,j)\in \UR$. Define the {\bf $(i,j)$-hook $h_{ij}$} to be
\[
h_{ij}=\{(i,a)\,|\,i<a<j\}\cup\{(a,j)\,|\,i<a<j\}\cup\{(i,j)\}\subseteq \tilde \Phi^+
\]
Let $A\in \widetilde V$. If $(i,a), (j,b)\in \main(A)$, then we call $(j,a)$ {\bf main hook intersection}. \hfill $\square$
\end{Defn}


We illustrate these definitions as follows: 
\begin{equation}\label{lA}
\includegraphics[width=0.85\textwidth]{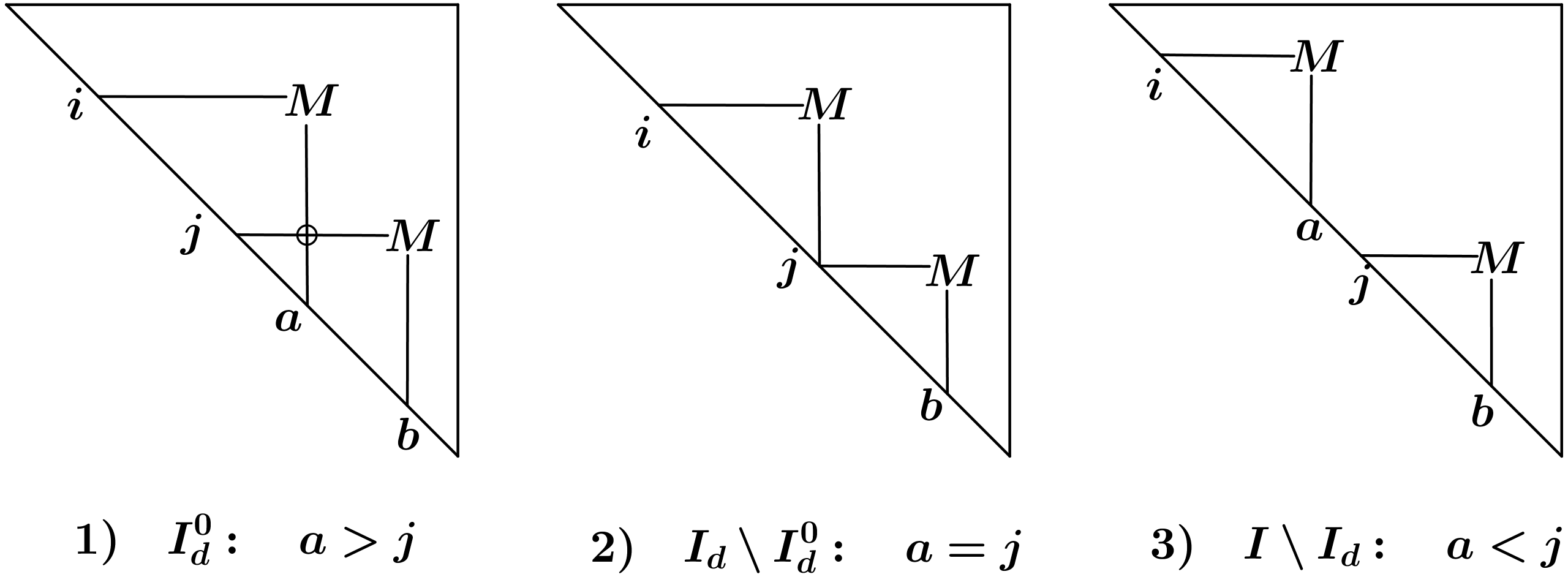}
\end{equation}

For $\widetilde{[A]}\in  \hat {\widetilde V}$ being a verge, the intersection $\cO_A^r\cap \cO_A^l$ of the right and left orbit of $\widetilde{[A]}$, is exactly the set of all  characters $\widetilde{[B]}\in  \hat {\widetilde V}$ such that $B$ coincides with $A$ at all positions except  the main hook intersections  in $\tilde \Phi^+$. For example, in the illustration  \ref{lA} 3), the entries at the main hook intersection $(j,a)$ can be changed either by the left action of the root subgroup $\widetilde X_{ij}$ or the right action of the root subgroup $\widetilde X_{ab}$.

\begin{Lemma}
Let $\widetilde{[A]}\in \hat {\widetilde V}$ be staircase.  Then 
$I=I(A), I_d=I_d(A), I_d^\circ=I_d^\circ(A)$ are closed subsets of $\tilde \Phi^+$. If $\supp(A)\subseteq \pUP$ and hence $A\in V_{\pUPs}=V$, then 
$I, I_d, I_d^\circ$ are closed subsets of $\pUP$ as defined in \ref{Upattern}. 
Thus $\widetilde U_I, \widetilde U_{I_d}, \widetilde U_{I_d^\circ}$ are pattern subgroups of $\widetilde U=U_N(q)$ and $ U_I, U_{I_d}, U_{I_d^\circ}$ are pattern subgroups of $ U=U(\fX_n)$, $\fX_n=\fB_n, \fC_n$ or $\fD_n$. Moreover, $\widetilde U_{I_d}$ (resp. $U_{I_d}$), $\widetilde U_{I_d^\circ}$ (resp. $U_{I_d\circ}$) are normal subgroups of $\widetilde U_I$ (resp, $U_I$).
\end{Lemma}
\begin{proof}
To check if $I, I_d, I_d^\circ$ are closed in $\pUP$, we only need to prove $I, I_d, I_d^\circ$ are closed in $\tilde \Phi^+$ since  the situation ii) in  \ref{Upattern} never occurs. In fact, if
$(i,j), (\bar k,  \bar j)\in I(A)$, then by \ref{defofI} there exist $(i,a), (j,b),(\bar k,c),(\bar j,d)\in \main(A)$ such that $a<b$ and $c<d$. But $ (j,b), (\bar j,d)\in \main(A)$ can not happen since $\main(A)\subseteq \supp(A)\subseteq \pUP$.

Now assume $(i,j), (j,k)\in I$. Then there exists $(i,a), (j,b), (k,c)\in \main(A)$ with $a<b<c$, hence $(i,k)\in I$. Thus $I$ is closed in $\tilde \Phi^+$ and hence also closed in $\pUP$.
By similar argument, one can check easily that this holds also for $I_d$ and $I_d^\circ$. Then by definition $\widetilde U_I, \widetilde U_{I_d}, \widetilde U_{I_d^\circ}$ are pattern subgroups of $\widetilde U=U_N(q)$ and $ U_I, U_{I_d}, U_{I_d^\circ}$ are pattern subgroups of $ U=U(\fX_n)$, $\fX_n=\fB_n, \fC_n$ or $\fD_n$.

 Let $(i,j)\in I_d$. Then there exists $(i,a), (j,b)\in \main(A)$ such that $i<a\leq j<b$. Suppose $(j,k), (s,i)\in I$. Then there exists $(k,c),(s,r)\in \main(A)$ such that $b<c$ and $r<a$. Since $i<a(\leq j)<k<c$ and $s<r<(a\leq) j<b$, we have  $(i,k), (s,j)\in I_d^\circ \subseteq I_d$ and hence $\widetilde U_{I_d}$ (resp. $U_{I_d}$) are normal subgroups of $\widetilde U_I$ (resp, $U_I$). Replacing $a\leq j$ by $a<j$, we have indeed $(i,j)\in I_d^\circ$ and by similar argument  one can prove that $\widetilde U_{I_d^\circ}$ (resp. $U_{I_d\circ}$) are also normal subgroups of $\widetilde U_I$ (resp, $U_I$).
\end{proof}

In [\cite{mindim}], it was shown that for a staircase verge $\widetilde{[A]}\in \hat {\widetilde V}$ the left action of $\widetilde U_I, I=I(A)$, on $\widetilde{[A]}$ induces $\widetilde U$-endomorphisms on $\C \tilde \cO_A$, $\widetilde U_{I_d^\circ}$  acts trivially and $\widetilde U_{I_d}$ acts as linear character $\tilde \Psi_A$ on $\C \widetilde{[A]}$. If $\tilde E_A\in \C \widetilde U_{I_d}$ denotes the (central) idempotent such that $\C \tilde E_A$ affords $\tilde \Psi_A$, then $\tilde E_A$ is central in $\tilde U_I$ and $\C \widetilde U_I/\C \widetilde U_{I}\tilde E_A $ is the endomorphism ring of $\C \tilde \cO_A$.

Now suppose $A\in V=V_{\pUPs}\subseteq V_{\URs}$ is a staircase verge, and consider $\hat V$ as subset of $\hat {\widetilde V}$ by restriction of maps. Then $I(A)$ is contained in $\tril\subseteq \pUP$. As pointed out in \ref{LeftTildeUandUAction}, $x_{ij}(\alpha).[A]=\tilde x_{ij}(\alpha).[A]=[\tilde x_{ji}(-\alpha).A]$ for $(i,j)\in \pUP$ and $\alpha\in \F_q$. This applies in particular to $(i,j)\in I(A)$. Moreover one checks using \ref{isotostaircase} that:
\begin{Lemma}
Let $[A]\in \hat V$ be a staircase verge, $I=I(A)$, and let $u\in U_I\leq U$. Define $\tilde u\in \widetilde U$ as in \ref{LeftTildeUandUAction}. Then
 $\supp(u^{-t}A)\subseteq \pKL$ and hence 
\[
\lambda_{\tilde u}([B])=\tilde u [B]=u[B]=\lambda_u[B]
\]
for all $[B]\in \tilde \cO_A\subseteq \hat V\subseteq \hat {\widetilde V}$.
\end{Lemma}
\begin{proof}
By \ref{isotostaircase} it suffices to prove   $\supp(u^{-t}A)\subseteq \pKL$ for $u=x_{ij}(\alpha)$ with $(i,j)\in I, \alpha\in \F_q$. Suppose $(i,a), (j,b)\in \main(A)$ with $a<b$. Then $1\leq i, j\leq n$ and hence 
\[
u^{-t}A=x_{ij}(\alpha)^{-t}A=\tilde x_{ji}(-\alpha)A
\]
Thus $u^{-t}A$ is different from $A$ only on possible positions $(j,c)$ with $c\leq a$ since $(i,a)\in \main(A)$. But  $(j,c)\in \KL$ since $c\leq a<b$ and $(j,b)\in \pUP.$ So we have shown: $\supp(u^{-t}A)\subseteq \pKL$ as desired.
\end{proof}
\begin{Cor}
Let $[A]\in \hat V$ be a staircase verge character. Then $\C \cV(A)$ is a $\C U_I$-$\C U$- and a $\C \widetilde U_I$-$\C U$-bimodule. \hfill $\square$
\end{Cor}

\section{The Main Theorem}

Having shown in theorem \ref{Main1} that every irreducible $\C U$-module appears as constituent in an orbit module $\C \cO_A$ for some main separated character $[A]\in \hat V$, we now want to show our second main result stating, that the characters of $U$ afforded by main separated verge modules  are pairwise orthogonal and that any two main separated orbit modules afford either identical or orthogonal $U$-characters. However, it will become apparent from our proof of the main theorem, that in type $\fC$ anti-diagonal main conditions cause additional problems. Indeed we have not been able to overcome these so far. Thus we have to restrict ourselves to linear characters $[A]\in\hat V$ satisfying $\main(A)\subseteq \UP$.


\begin{Defn}
For $A\in V$ we define $\Psi_{A}:\Stab_{U}[A]\rightarrow \C^*$ by:
$
[A]u=\Psi_{A}(u)[A] $, for all $ u\in \Stab_{U}[A].
$  \hfill $\square$
\end{Defn}
Thus $\Psi_{A}$ is a linear character of $\Stab_{U}[A]$ and by [\cite{GJD}, 4.9] we have
\[
\Psi_{A}(u)=\theta\kappa(-A,\pi(u^{-1}))=\theta\kappa(-A,u^{-1}).
\]
Note that $\C [A]$ then is the 1-dimensional $\C \Stab_{U}[A]$-module affording $\Psi_A$.

Let $A\in V$. By general theory $\C \cO_A\cong \Ind^U_{S} \C [A]$ where $S=\Stab_{U}[A]$. By Mackey  decomposition we obtain immediately:
\begin{Lemma}\label{Mackey}
Let $A, B\in V, S=\Stab_{U}[A], T=\Stab_{U}[B]$.
\[
\Hom_{\C U}(\C \cO_A, \C \cO_B)\cong \bigoplus_{d\in \cD_{B, A}} \Hom_{\C(S\cap T^d)} (\C [A], \C [B]^d)
\]
where $ \cD_{B, A}$ is a system of $T$-$S$ double coset representatives in $U$ and $T^d=d^{-1}Td$ and $\C [B]^d$ denotes the conjugate $T^d$-module.\hfill$\square$
\end{Lemma}

Note that $\C [B]^d\cong \C [C]$ as $T^d$-module with $[C]=[B].d=[Bd^{-t}]\in \cO_B=\{[B].u\,|\,u\in U\} = \{[B].dS\,|\,d\in \cD_{B,A}\}$. For $s\in S$ we have
\begin{eqnarray*}
\Hom_{\C (S\cap T^{ds})}(\C [A], \C ([B].d.s))&\cong&
\Hom_{\C (S^{s^{-1}}\cap T^d)}(\C [A].s^{-1}, \C ([B].d))\\&\cong&
\Hom_{\C (S\cap T^d)}(\C [A], \C ([B].d)).
\end{eqnarray*}

Let $A, B\in V.$ For convenience we use henceforth the following notation:
\begin{equation}\label{stabintersection}
\Stab_U(A, B)=\Stab_U(A)\cap \Stab_U(B).
\end{equation}

From \ref{Mackey} this we conclude immediately:

\begin{Cor}\label{HomCond}
Let $A,B\in V$.
Then the following statements are equivalent
 \begin{enumerate}
 \item [i)] $ \Hom_{\C U}(\C \cO_A, \C \cO_B)\neq (0)$;
 \item [ii)] There exist $[C]\in \cO_A$, $[D]\in \cO_B$ such that $\Hom_{\Stab_U(C, D)}(\C [C], \C [D])\neq (0);$
 \item [iii)] There exist $[C]\in \cO_A$ and $[D]\in \cO_B$ such that $\Psi_C$ and $\Psi_D$ coincide on $\Stab_{U}(C, D);$
\item [iv)] There exists  $[D]\in \cO_B$ such that $\Psi_A$ and $\Psi_D$ coincide on $\Stab_{U}(A, D)$.\hfill $\square$
 \end{enumerate}
\end{Cor}

The basic idea for the proof of our main result consists of using item iv) (where $[A]$ and $[B]$ are main separated) in the corollary above and the discussion in section \ref{sectionstab} on stabilizers to work inductively our way down the rows of the matrices $A$ and $B$. For this we need some preparation.

\smallskip
Let $[A]\in \hat V$ be a main separated core and let $1\leq i<\tilde n$ be such that column $\bar i$ of $A$ is the zero column (or equivalently does not contain a main condition). Thus row $i$ is not an arm of $A$ and is not contained in $\Limb(A)$. Let $[B]\in \hat V$ be main separated (and hence staircase as well) and suppose, that matrix $A$ and $B$ coincide to the north of row $i$. Thus 
\begin{equation}\label{ABnorth}
B_{kj}=A_{kj}\quad \text{ for } 1\leq k<i, k<j<\bar k
\end{equation}
In addition we assume that $ \Hom_{\C U}(\C \cO_A, \C \cO_B)\neq (0)$.
Recall, that by \ref{describestabRi} the part of $\Stab_U[A]$ lying in the row subgroup $\cR_i=\langle X_{ij} \,|\, i<j\leq \widetilde i\rangle$,  can be described by the solution space of the homogeneous system of linear equations whose coefficient matrix is $$A_i=\sum_{\stackrel{1\leq a<i}{i<b\leq\,  \widetilde i}} A_{ab}e_{ab}.$$
Note that $B_i=A_i$ by the assumption \ref{ABnorth}. Hence observing that by lemma \ref{stabRi} the solution set of $B_i\underline{\alpha}=0$ gives $\Stab_U[B]\cap \cR_i$ we see that 
\begin{equation}
\Stab^i_U[B] = \Stab_U[B]\cap \cR_i=\Stab_U[A]\cap \cR_i = \Stab^i_U[A] \subseteq \Stab_{U}(A,B).
\end{equation}

We again suppose that the main conditions of $A$ in $A_i$ are given as $\{(i_1, j_1), \ldots, (i_r, j_r)\}$ with $1\leq i_\nu<i$ ($\nu=1, \ldots, r$) and $i<j_1<\cdots<j_r<\bar i$. Recall from section \ref{sectionstab} that the set $S_i$ is defined as $S_i =\{k\,|\, i<k\leq \widetilde i, k\notin \{j_1, \ldots, j_r\}\}$. Then $S_i$ is the set of positions on row $i$ such that there is no main condition in its column to the north of it.

\begin{Lemma}\label{rowivalues}
Let $[A], [B]\in \hat V$ and $1\leq i<\tilde n$ as above. Assume that $\Psi_B$ and $\Psi_A$ coincide on
$$
\Stab^i_U(A,B) := \Stab^i_U[A]\cap\Stab^i_U[B].
$$
 Then for $s\in S_i$:
\begin{equation}\label{barelation}
B_{is}=A_{is}-\sum_{\stackrel{j_\nu>s}{1\leq \nu\leq r}} (A_{ij_\nu}-B_{ij_\nu})\alpha^s_{j_\nu}
\end{equation}
where $\underline \alpha_s=(\alpha^s_{i+1}, \alpha^s_{i+2}, \ldots, \alpha^s_{\widetilde i})^t\in \F_q^{\widetilde i-i}$ is again the $s$-th column vector of $\tilde A_i$ in \ref{tildeAi} reordered as in \ref{basisstab}.
\end{Lemma}
\begin{proof}
In \ref{basisstab} it was shown how the solutions space of $A_i\underline\alpha=0$ ($\underline\alpha\in \F_q^{\widetilde i-i}$) determines $\Stab_U^i[A]$, and the basis $\{\underline\alpha_s\in \F_q^{\widetilde i-i}\,|\, s\in S_i\}$
was determined. Moreover it was shown in \ref{stabcharbasis} that the linear character $\Psi_A, \Psi_B$ of $\Stab_U[A]$, and $\Stab_U[B]$  satisfy:
\begin{eqnarray*}
\Psi_A(x(\underline\alpha_s))=\theta(-A_{is}+\sum_{\stackrel{j_\nu>s}{1\leq \nu\leq r}} A_{ij_\nu}\alpha^s_{j_\nu})=\Psi_B(x(\underline\alpha_s))=\theta(-B_{is}+\sum_{\stackrel{j_\nu>s}{1\leq \nu\leq r}} B_{ij_\nu}\alpha^s_{j_\nu}),
\end{eqnarray*}
since $x(\underline\alpha_s)\in \Stab_U(A, B)\cap \cR_i$. Dividing by the right hand side of this equation  yields
\[
1=\theta(B_{is}-A_{is}+\sum_{\stackrel{j_\nu>s}{1\leq \nu\leq r}}(A_{ij_\nu}-B_{ij_\nu})\alpha^s_{j_\nu})=\theta(\sum_{i<l\leq\, \widetilde i} (A_{il}-B_{il})\alpha^s_{l}).
\]
Since the solution set of $A_i\underline\alpha=0$ is a linear space, $\lambda \underline\alpha_s$ is a solution as well for all $\lambda\in \F_q$, hence $x(\lambda \underline\alpha_s)\in \Stab_U(A, B)\cap \cR_i$. Taking characters we obtain
\[
1=\Psi_A(x(\lambda \underline\alpha_s)^{-1})\Psi_B(x(\lambda \underline\alpha_s))=\theta\big(\lambda(\sum_{i<l \leq\, \widetilde i} (A_{il}-B_{il})\alpha^s_{l})\big),\quad \forall \, \lambda\in \F_q
\]
forcing $\sum_{i<l\leq\, \widetilde i} (A_{il}-B_{il})\alpha^s_{l}=0$  since $\theta$ is a non trivial character of $(\F_q, +)$, and the lemma follows.
\end{proof}

\begin{Remark}
Suppose $U$ is of type $\fC_n$. Let $i<j<\bar i$ and $\alpha\in \F_q$, then $x_{ij}(\alpha)$ acts on any $[C]\in \hat V$ by the combination of two restricted column operations as pictured in \ref{illTruncatedColumnOperationG}.
If  $C_{k\bar i}=0$ for all $k<i$, then $[C]x_{ij}(\alpha)=\theta(\alpha C_{ij})[D]$, where $[D]=[C].x_{ij}(\alpha)$ may differ from $[C]$ only on column $i$ and on position $(i, \bar j)$. Indeed $D_{ij}=C_{ij}+\alpha  C_{i\, \bar i}$
  for $j\leq n$ and $D_{ij}=C_{ij}-\alpha  C_{i\, \bar i}$ for $j>\tilde n$.

Suppose in \ref{rowivalues} in addition, that $A_{i\, \bar i}=z\neq 0$. Thus in particular column $\bar i$ is not a zero column. Then $A_{i\, \bar i}=z$ is the only non zero entry of column $\bar i$ in $A$ and $(i, \bar i)$ is a main condition. Since $A$ and $B$ coincide to the north of row $i$ by assumption, $B_{k\bar i}=0$ for $1\leq k<\bar i$.

As an easy consequence, $[X_{ij}, X_{i \bar j}]=X_{i \bar i}$, and hence the  row group $ \cR_i$ is not abelian any more. Moreover $\Stab_U[A]\cap  \cR_i=\Stab_U[B]\cap  \cR_i=X_{i \bar i}$. We shall use equation \ref{barelation} later on to show that acting by suitable row operations will change $B$ into a matrix $C$ with $\Stab_U[C]=\Stab_U[B]$ such that $A$ and $C$ coincide to the north of row $i+1$. This will be our induction step to prove our main result. If $A_{i \bar i}=z\neq 0$, we can only conclude $B_{i\bar i}=A_{i \bar i}$. Note that  $A$ being a core forces then $A_{ik}=0$ for $i<k<\bar i$, whereas the entries in row $i$ of $B$ besides the anti-diagonal one can be arbitrary. 

The vectors $\underline{\alpha}_s$ defined as solutions of $A_i \underline{\alpha}=B_i \underline{\alpha}=0$ are still defined but we do not know, if the character values calculated according to \ref{barelation} of $A$ and $B$ coincide, since neither $x(\underline{\alpha}_s)$ is in $\Stab_U[A]$ nor in $\Stab_U[B]$. This is the problem with type $\fC_n$ as indicated in the introduction to this section. \hfill $\square$
\end{Remark}

In view of the previous remark we shall from now on assume that also for type $\fC$ the matrices $A$ and $B$ in $V$ have support entirely contained in $\UP$.

\begin{Cor}\label{coincide1}
In the situation of \ref{rowivalues}  row $i$ in $A$ and $B$ coincide, if they coincide at positions  $(i, j_\nu)$ for all  $\nu=1,\ldots,r$, (that is at all positions of row $i$ below main conditions).\hfill $\square$
\end{Cor}

\begin{Cor}\label{samemainRi}
Let $[A], [B]\in \hat V$ as above. Then either the $i$-th row of $A$ and $B$ are zero rows or there exists $(i,s)$ with $(i,s)\in \mc(A)\cap \mc(B)$ and $A_{is}=B_{is}$.
\end{Cor}
\begin{proof}
We are done if row $i$ is a zero row in both $A$ and $B$. So let $(i,s)$ be a main condition $i<s<\bar i$ such that $A_{il}=B_{il}=0$ for $s<l<\bar i$. Note that $s\in S_i$. By \ref{rowivalues} we obtain:
$$B_{is}-A_{is}=\sum_{\stackrel{j_\nu>s}{1\leq \nu\leq r}} (B_{ij_\nu}-A_{ij_\nu})\alpha^s_{j_\nu}=0,$$
since $A_{ij_\nu}=0=B_{ij_\nu}$ for $j_\nu>s$. Now  the result follows.
\end{proof}

There is a further obvious application of \ref{rowivalues} and \ref{samemainRi} which we shall need later on.
\begin{Remark}\label{coincideintersection}
In the situation of \ref{samemainRi}, if the $i$-th row of $A$ and $B$ are not zero row, then by \ref{coincide1} row $i$ in $A$ and $B$ coincide, if they coincide at main hook intersections, (see \ref{hookintersection}).  \hfill $\square$
\end{Remark}

We now can prove a proposition which makes induction on the row index $i$ working:
\begin{Prop}\label{inductionstep}
Let $[A], [B]\in \hat V$ staircase and let $[A]$ be a core. Let $1\leq i<\tilde n$ and suppose that $A$ and $B$ coincide to the north of row $i$ and $\bar i$ is a zero column. Let
$
\{(i_1, j_1),\ldots, (i_r, j_r)\}=\{(a,b)\in \mc(A)\,|\,1\leq a<\tilde n, \,i<b<\bar i\}.
$
Furthermore suppose that $\Psi_A$ and $\Psi_B$ coincide on $\Stab_U(A, B)$. Then $g.[B]$ and $[A]$ coincide also on row $i$ and hence to the north of row $i+1$, where $g=x_{i,i_1}(\lambda_1)\cdots x_{i, i_r}(\lambda_r)$ for certain $\lambda_1, \ldots, \lambda_r\in \F_q$. Moreover, $g\in U_{I(A)}\cap U_{I(B)}$.
\end{Prop}
\begin{proof}
Observe that $\supp \big(x_{i_\nu i}(\mu)B\big)\subseteq \pKL$ for arbitrary $\mu\in\F_q$, since $B_{i_\nu k}=0$ for $k>j_\nu$, and hence  $B_{i_\nu k}=0$ for $k>\bar i>j_\nu$.  
One checks easily that $\supp(g^{-t}B)\subseteq \pKL$ for any choice of  $\lambda_1, \ldots, \lambda_r$. Indeed $g$ is contained in the abelian pattern subgroup to the closed subset $\{(k,i)\,|\,1\leq k <i\}$ of $U$ ($\tilde U$). Hence by \ref{leftactionSupp}  for any choice of  $\lambda_1, \ldots, \lambda_r\in \F_q$, $g$ acts by left multiplication as $\C U$-isomorphism from $\C \cO_B$ to $\C \cO_D$ with $[D]=g.[B]=[g^{-t}.B]=[\pi_{\pUPt}(g^{-t}B)]$  (and  as $\C \tilde U$-isomorphism from $\C \tilde\cO_B$ to $\C\tilde  \cO_D$ as well). For $u\in \Stab_{U}[A]$ by [\cite{GJD}, 6.4] we have therefore
\begin{equation}
(g[B])u=g([B]u)=g\big(\Psi_B(u)[B]\big)=\Psi_B(u) (g[B])
\end{equation}
proving that $\Stab_U(g.[B])=\Stab_U(g[B])=\Stab_U([B])$ and $\Psi_B=\Psi_D$, setting $g.[B]=[D]\in \hat V$.
Observe that $x_{ik}(\lambda).[B]$ adds $-\lambda$ times row $k$ to row $i$ and projects the resulting matrix into $V_{\pUPs}$. Thus $g.[B]=[D]$, where $B$ and $D$ coincide in all rows except the $i$-th one. Moreover 
$D_{ij_\nu}=B_{ij_\nu}-\lambda_\nu B_{i_\nu j_\nu}$. Setting 
$\lambda_\nu=B^{-1}_{i_\nu j_\nu}(B_{ij_\nu}-A_{ij_\nu})$ gives $D_{ij_\nu}=A_{ij_\nu}$. See illustration as follows:
\begin{center}
\includegraphics[width=0.75\textwidth]{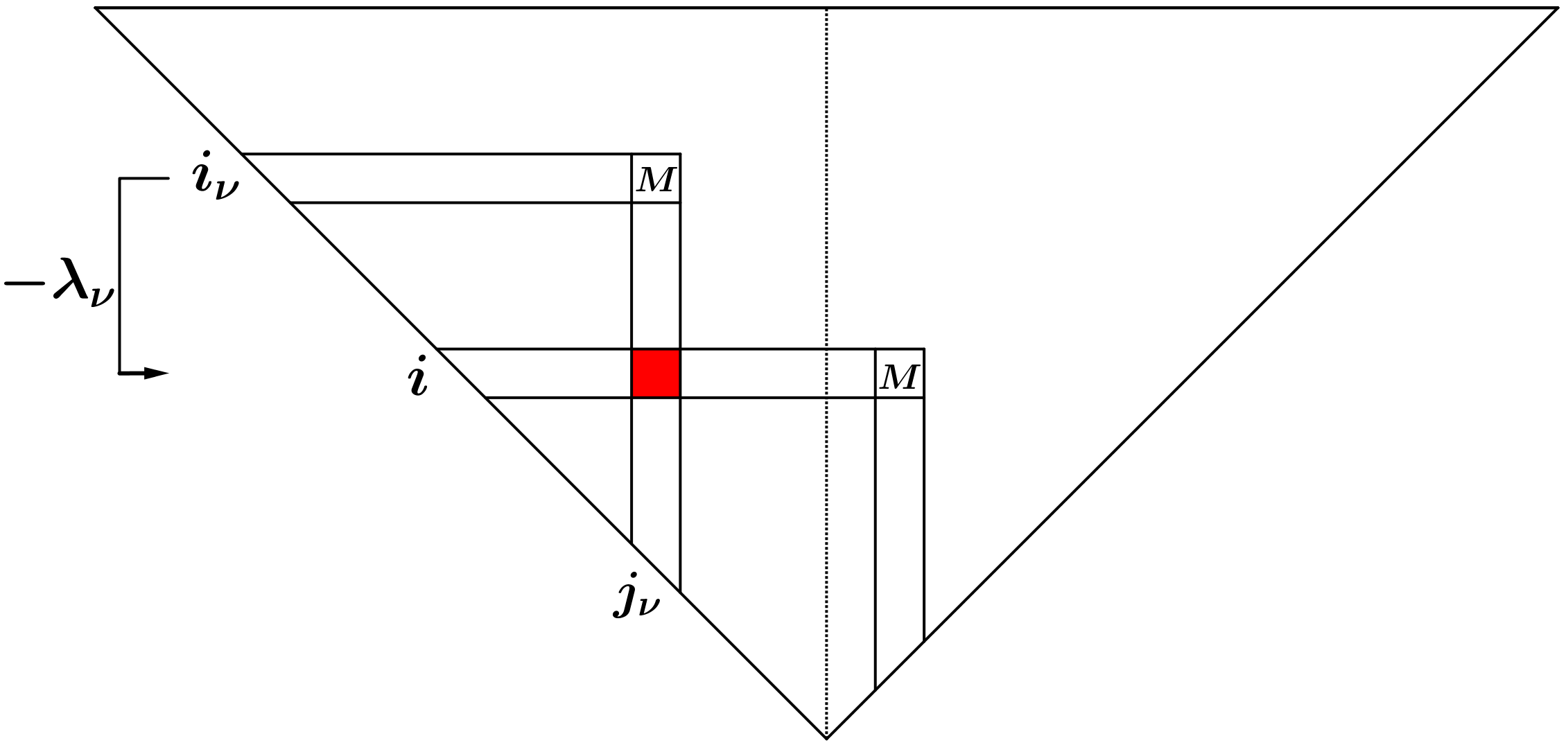}
\end{center}
Note that $D$ coincides with $B$ and hence with $A$ at all entries to the north of row $i$. Moreover by construction, $A$ and $D$ coincide on positions $(i, j_\nu)$ for all $1\leq \nu\leq r$. By the previous corollary \ref{coincide1}, $A$ and $D$ coincide in row $i$, hence $g.[B]=[D]$ coincides with $[A]$ to the north of row $i+1$ as desired. By the construction of $g$, it is obvious that $g\in U_{I(A)}\cap U_{I(B)}$.
\end{proof}

\begin{Remark}
One might worry in the proof above, what happens, if there is a main condition $(i,k)$ on row $i$ of $B$ to the left of some main condition $(i_\nu, j_\nu)$, that is $k<j_\nu$. Could it happen that $g$ inserts a non-zero value on row $i$ of $B$ to the right of that main condition? This does not happen since then, by
\ref{samemainRi}, $(i,k)\in \mc(A)$ as well and $A_{ij_\nu}=B_{ij_\nu}=0$ since $k<j_\nu$, hence $\lambda_\nu=B^{-1}_{i_\nu j_\nu}(B_{ij_\nu}-A_{ij_\nu})=0$. \hfill $\square$
\end{Remark}


We can now prove a further main result of this paper for arbitrary main separated orbits in types $\fB_n$ and $\fD_n$ and in type $\fC_n$ for orbits such that all its main conditions are contained in $\UP$.  

\begin{Theorem}\label{main3}
 Let $[A], [B]\in \hat V$ be main separated with $\supp(A), \supp(B)\subseteq \UP$. Suppose further that $[A]$ is a core, and that $\Psi_A$ and $\Psi_B$ coincide on $\Stab_U(A, B)$. Then there exists $g\in U_{I(A)}\leq U$ such that $g.[B]=[A]$.
\end{Theorem}\begin{proof}
We proceed by induction on the row index $i=1,2, \ldots,\tilde n-1$, showing that we can find $g\in U_{I(A)}\cap U_{I(B)}$, such that $g.[B]$ and $[A]$ coincide north of row $i$. Observe, that $g\in U_{I(C)}$ acting on any $[C]\in \hat V,$ changes only  those rows in $C$, which contain a main hook intersection and hence a main condition, by the discussion in section 5. Thus, in particular if $[C]$ is staircase, left action by $g$ will preserve the verge of $[C]$ and if $[C]$ is main separated so is $g.[C]$

Assume first $i=1$, then $A$ and $B$ trivially coincide to the north of row $i$. Inspecting once more \ref{illTruncatedColumnOperationG} we see that for any linear character $[C]\in\hat V$ the root subgroup $X_{1j}$, ($1<j<N$) acts by a linear character on $[C]$. Hence it is contained in $\Stab_U[C]$, unless $C_{1\bar 1}\neq 0$. Moreover for $\alpha \in \F_q$
\[
[C]x_{1j}(\alpha)=\Psi_C(x_{1j}(\alpha))[C]=\theta\kappa(-C, x_{1j}(-\alpha))[C]=\theta(\alpha C_{ij})[C].
\]

Thus in particular $X_{1j}\subseteq \Stab(A,B)$ and 
\[
\Psi_A(x_{1j}(\alpha)) = \theta(\alpha A_{1j}) = \Psi_B(x_{1j}(\alpha)) = \theta(\alpha B_{1j})
\]

for all $\alpha \in \F_q$ forcing $A_{1j} = B_{1j}$ for all $1<j<N$. Since $A_{1\bar 1} = B_{1\bar 1}= 0$ by assumption we conclude that row $1$ coincides in $A$ and $B$.

Setting $g_1=1\in U_{I(A)}\cap U_{I(B)}$ and $B(1)=B$, we have shown that $A$ and $B(1)$ coincide  to the north of row 2. 

Suppose we have constructed $g_i\in U_{I(A)}\cap U_{I(B)}$ such that $A$ and $B(i) $ coincide to the north  of row $i+1$ where $[B(i)]=g_i.[B]$. Now if column $\overline{i+1}$ (on which $A$ and $B(i)$ coincide) is not a zero column, it contains a main condition of $A$, since $[A]$ is a core, and hence $\cR_{i+1}\cap \Stab_U[A]=(1)$ by \ref{describestabRi}. Since $[A]$ and $[B]$ are main separated by assumption, row $i+1$ cannot have any main condition in neither in $A$ nor in $B$, and hence all entries in $B(i)$ and $A$ on row $i+1$ are zero. Thus $A$ and $B(i)$ coincide to the north of row $i+2$ once we are done with $\tilde g_{i+1}=1$ and $g_{i+1}=\tilde g_{i+1}g_i$. 

We remark that here we cannot apply this argument in case of type $\fC_n$, if the main condition on row $i+1$ is $(i+1, \overline{i+1})$. Then $X_{i+1, \overline{i+1}}\subseteq \Stab_U[A]$ and $B(i)$ can have many non-zero values on row $i+1$, which we cannot remove. Matrix A being a core, has only zeros on row $i+1$ besides $A_{i+1, \overline{i+1}}\neq 0$. Here the induction brakes down for orbits with anti-diagonal main conditions in type $\fC_n$.

So back to the case that $\supp(A)\subseteq \UP$, we assume now, that column $\overline{i+1}$ is a zero column in $A$ and hence in $B$ as well. Now $B(i)$ and $A$ satisfy the assumptions of Proposition \ref{inductionstep} since
\begin{equation}
\Stab_U(g.[B])=\Stab_U([B])\quad \text{ and } \quad\Psi_{B}=\Psi_{g.[B]}
\end{equation}
for all $g\in U$ such that $\supp(g^{-t}B)\subseteq \KL$.
In Proposition \ref{inductionstep}  we constructed $\tilde g\in U_{I(B(i))}\cap U_{I(A)}$ with $[B(i+1)]=\tilde g.[B(i)]=[A]$ such that $B(i+1)$ and $A$ coincide to the north of row $i+2$. Now the theorem follows by induction, setting $g_{i+1}=\tilde g g_i$.
\end{proof}

We draw a few immediate consequences from our theorem:

\begin{Cor}
Let $[A], [C]\in \hat V$ be main separated and in the case of type $\fC_n$, we assume additionally $\supp(A),\supp(C)\subseteq \UP$. If $\Hom_{\C U}(\C \cO_A, \C \cO_C)\neq (0)$, then there exists $[B]\in \cO_C$ such that $\Psi_A$ and $\Psi_B$ coincide on $\Stab_U(A, B)$, and we have:
\begin{enumerate}
\item [i)] $\verge(A)=\verge(B)$ and hence $\cV(A)=\cV(B)$.
\item [ii)] $U_{I(A)}=U_{I(B)}.$
\item [iii)] There exists $g\in U_{I(A)}$ such that $g.[B]=[A].$
\item [iv)] $\C \tilde \cO_A=\C \tilde \cO_B$ and $g.\widetilde{[B]}=\widetilde{[A]}$.
\item [v)] $\C \cO_A\cong \C \cO_B$.
\end{enumerate}
Thus, the corresponding main separated $U$-orbit modules are either isomorphic or afford orthogonal characters.\hfill $\square$
\end{Cor}

\begin{Cor}
If $U$ is of type $\fB_n$ or $\fD_n$ then every irreducible $\C U$-module is constituent of precisely one isomorphism class of main separated $U$-orbit modules in $\C \hat V$. \hfill $\square$
\end{Cor}

\begin{Remark}
It is obvious for $[A]\in \hat V$ that $\cV(A)$ is $\widetilde U$-invariant and in fact affords an Andr\'{e}-Yan supercharacter defined in [\cite{andre}, \cite{yan}]. The $\widetilde U$-orbit modules $\C\widetilde \cO_A$ with $\main(A)\subseteq \pUP$ decompose upon restriction to $U$ into the direct sum of staircase orbit modules $\C\cO_B$, where $[B]$ runs throug the set  of all staircase core characters in $\hat V$ satisfying $\verge(B) = \verge(A)$. In [\cite{andreneto1,andreneto2,andreneto3}] Andr\'{e} and Neto constructed supercharacter theories for types $\fB_n, \fD_n$ and $\fC_n$. Their superclasses arise as intersections of Andr\'{e}-Yan $\widetilde U$-superclasses with $U\leq \widetilde U$. Recall, that the Andr\'{e}-Yan $\widetilde U$-superclasses are labelled by verge matrices, that are matrices in $\widetilde V$ having in each row and in each column at most one nonzero entry (see e.g. [\cite{andre}] and [\cite{yan}]).  Abusing notation we call the positions $(i,j)$ with $1\leq i<j\leq N$ for a verge matrix $A\in \widetilde V$ such that $A_{ij}\neq 0$ again main conditions and denote the set of main conditions of $A$ by $\main(A)$. For a verge matrix $A\in \widetilde V$ with $\supp (A) \subseteq\ \pUP$ let (derived from \ref{rootsubgroupsU})
\begin{equation}\label{Uvergematrix}
\hat A = \sum_{(i,j)\in\main  (A)}A_{ij}e_{ij}- A_{ij}e_{\bar j\bar i}.
\end{equation}
in type $\fD_n$ and for type $\fB_n$, if $j\neq n+1$. For type $\fB_n$ we define in addition 
\begin{equation} \label{UvergematrixB}
\hat A = \sum_{(i,n+1)\in\main  (A)}A_{i,n+1}e_{i,n+1}- A_{i,n+1}e_{n+1,\bar i} - \frac{1}{2}A_{i\bar i}^2.
\end{equation}

For type $\fC_n$ we set 
\begin{equation}\label{UvergematrixC}
\hat A = \sum_{(i,j)\in\main  (A), j< \bar i }A_{ij}e_{ij}-\epsilon A_{ij}e_{\bar j\bar i} + \sum_{1\leq i < n} A_{i\bar i}e_{i\bar i},
\end{equation}
where $\epsilon = -1$ for  $j> n$ and $\epsilon = 1$ otherwise (for types $\fB_n, \fD_n$ we set always $\epsilon = 1$ ).
Then $\hat A$ is a verge matrix if and only if $[A]\in V_{\pUPs}$ is main separated, as seen by the following counterexample:

\begin{equation}\label{mainbar}
\includegraphics[width=0.4\textwidth]{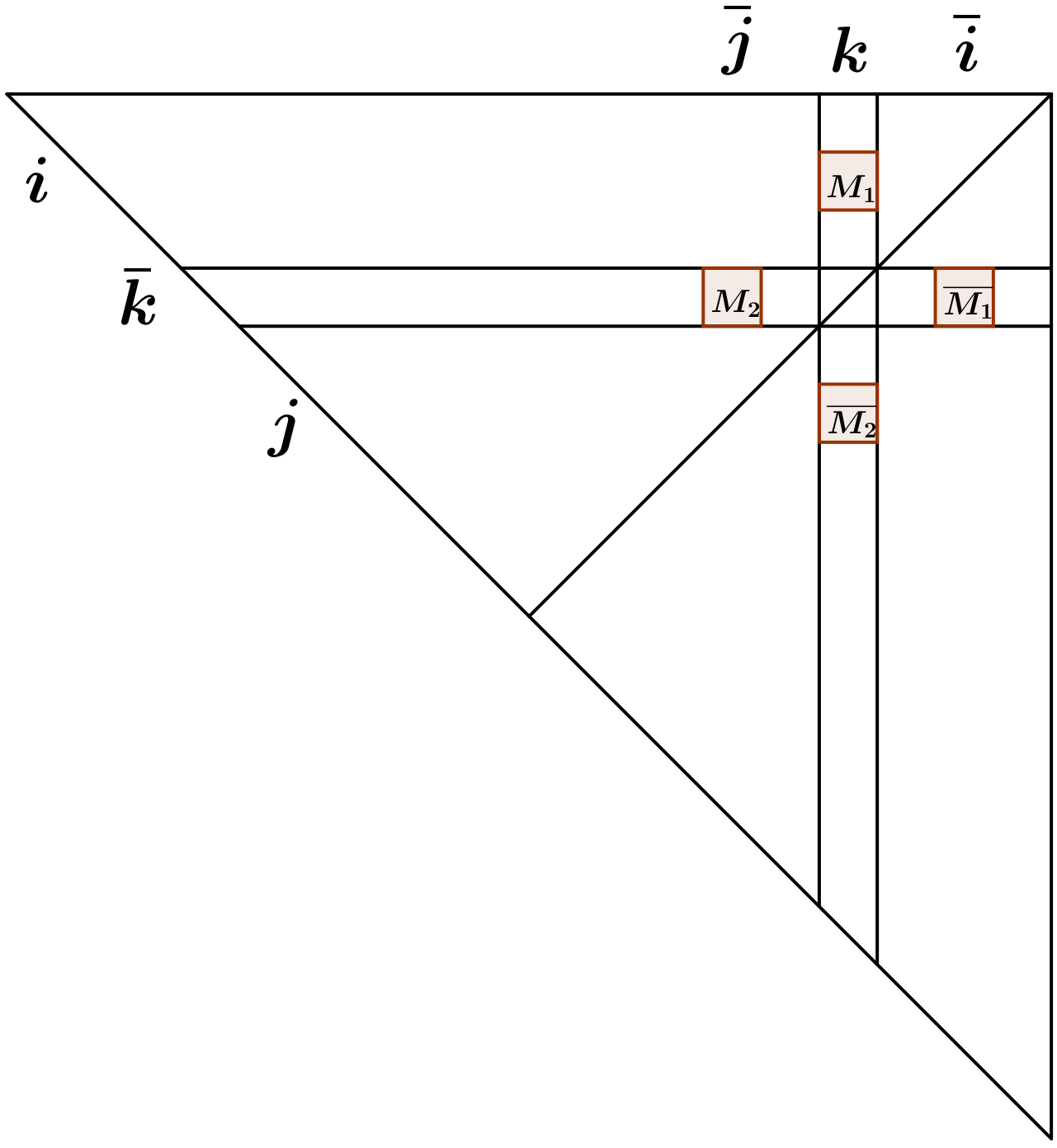}
\end{equation}
 
 The Andr\'{e}-Neto $U$-superclasses are labelled again by those verge matrices $A\in   V = V_{\pUPs}$ such that $\hat A$, as defined in \ref{Uvergematrix} and \ref{UvergematrixC} respectively, is a verge matrix again. Thus $[A]\in\hat V$ is main separated. The corresponding superclass is then given as (using ordinary matrix multiplication):
\begin{equation}\label{ANsuperclasses}
1+\widetilde U \hat A\widetilde U \cap U  = \{u\in U \,|\,u = 1+a\hat A b, a,b\in \widetilde U\}
\end{equation}

This explains, why we have to restrict ourselves to main separated characters $[A] \in \hat V$ for the description of the Andr\'{e}-Neto $U$-supercharacters too. Indeed, as our result \ref{Main1} shows, orbit modules to cores $[A]\in\hat V$ which are not main separated  may
have irreducible constituents in common, without being isomorphic. On the other hand the verge modules as defined in \ref{vergemodules} generated by  main separated cores are the restrictions to $U$ of those $\tilde U$-orbit modules defined by Andr\'{e}-Yan, which afford $\widetilde U$-supercharacters and whose main conditions are contained in $\pUP$. By the discussion above they are in bijection with Andr\'{e}-Neto $U$-superclasses. Thus our results yield for types $\fB_n, \fD_n$ a decomposition of Andr\'{e}-Neto $U$-supercharacters into characters afforded by $U$-orbit modules, which are again orthogonal or equal. For type $\fC_n$ this is still conjectural. Since every irreducible character occurs in precisely one of those, one may ask if they are the supercharacters of a finer supercharacter theory for $U$.\hfill $\square$
\end{Remark}


\providecommand{\bysame}{\leavevmode ---\ }
\providecommand{\og}{``} \providecommand{\fg}{''}
\providecommand{\smfandname}{and}
\providecommand{\smfedsname}{\'eds.}
\providecommand{\smfedname}{\'ed.}
\providecommand{\smfmastersthesisname}{M\'emoire}
\providecommand{\smfphdthesisname}{Th\`ese}

\end{document}